\documentclass[11pt,a4paper]{amsart}
\usepackage[utf8]{inputenc}
\usepackage[english]{babel}
\usepackage{amsmath}
\usepackage{url}
\usepackage{amsfonts}
\usepackage{amssymb}
\usepackage{xcolor}
\usepackage[left=2.5cm,right=2.5cm,top=3cm,bottom=3cm]{geometry}
\usepackage{enumitem}
\usepackage{epsfig}
\usepackage{nicefrac}

\usepackage[mathlines,pagewise]{lineno}
\title[$\LL_p$-maximal regularity for b.c.\ of mixed differentiability orders]{$\LL_p$-maximal regularity for parabolic and elliptic boundary value problems with boundary conditions of mixed differentiability orders}
\author{Bj\"orn Augner}
\address{Technische Universit\"at Darmstadt, Fachbereich Mathematik, Schlossgartenstra\ss{}e 7, 64289 Darmstadt.}
\email{augner@mma.tu-darmstadt.de}

 \newtheorem{theorem}{Theorem}[section]
 \newtheorem{remark}[theorem]{Remark}
 \newtheorem{proposition}[theorem]{Proposition}
 \newtheorem{corollary}[theorem]{Corollary}
 \newtheorem{lemma}[theorem]{Lemma}

 \newtheorem{assumption}[theorem]{Assumption}
 \newtheorem{example}[theorem]{Example}

 \newtheorem{definition}[theorem]{Definition}
 
 \renewcommand{\vec}{\bd}
 
 \DeclareMathOperator{\K}{\mathbb{K}}
 \DeclareMathOperator{\R}{\mathbb{R}}
 \DeclareMathOperator{\N}{\mathbb{N}}
 \DeclareMathOperator{\C}{\mathbb{C}}

 \renewcommand{\S}{\mathbb{S}}
 \DeclareMathOperator{\dv}{\operatorname{div}}

 \DeclareMathOperator{\B}{\Bcal}
 \DeclareMathOperator{\dom}{\operatorname{dom}}
 \DeclareMathOperator{\ran}{\operatorname{ran}}
 \renewcommand{\ker}{\operatorname{ker}}

 \newcommand{\bd}[1]{\boldsymbol{#1}}
 \newcommand{\bb}[1]{\boldsymbol{#1}}
 
 \renewcommand{\Re}{\operatorname{Re}\,}
 \renewcommand{\Im}{\operatorname{Im}\,}
 \newcommand{\dd}{\mathrm{d}}
 \newcommand{\norm}[1]{\left\| #1 \right\|} 
  
 \newcommand{\abs}[1]{\left| #1 \right|}

 \newcommand{\ii}{\mathrm{i}}
 \newcommand{\ee}{\mathrm{e}}
 \newcommand{\BB}{\mathrm{B}}
 \newcommand{\LL}{\mathrm{L}}
 \newcommand{\WW}{\mathrm{W}}
 \newcommand{\HH}{\mathrm{H}}
 \newcommand{\CC}{\mathrm{C}}
 \newcommand{\DD}{\mathrm{D}}
 
 \newcommand{\FF}{\mathrm{F}}
 \renewcommand{\SS}{\mathbb{S}}
 
 \newcommand{\Acal}{\mathcal{A}}
 \newcommand{\Bcal}{\mathcal{B}}
 \newcommand{\Ccal}{\mathcal{C}}
 \newcommand{\Dcal}{\mathcal{D}}
 \newcommand{\Ecal}{\mathcal{E}}
 \newcommand{\Fcal}{\mathcal{F}}
 \newcommand{\Gcal}{\mathcal{G}}
 \newcommand{\Hcal}{\mathcal{H}}
 \newcommand{\Kcal}{\mathcal{K}}
 \newcommand{\Pcal}{\mathcal{P}}
 \newcommand{\Rcal}{\mathcal{R}}
 \newcommand{\Scal}{\mathcal{S}}
 \newcommand{\Tcal}{\mathcal{T}}
 \newcommand{\HTcal}{\mathcal{HT}}
 \newcommand{\BIPcal}{\mathcal{BIP}}
 \newcommand{\Lcal}{\mathcal{L}}
 \newcommand{\Ical}{\mathcal{I}}
 
 \newcommand{\id}{I}
 \newcommand{\supp}{\operatorname{supp}}
 
\begin{document}
 \allowdisplaybreaks[1]

 \begin{abstract}
  We consider vector-valued elliptic boundary value problems and parabolic initial-boundary value problems on abstract Banach spaces of class $\HTcal$.
  For a very general class of boundary conditions, which includes, but is not limited to the classical cases of Dirichlet, Neumann and Robin boundary conditions,  we show $\LL_p$-maximal regularity of the system.
  We do not restrict ourselves to the homogeneous case, but also identify the temporal and boundary trace spaces, which are both sufficient and necessary to obtain a unique solution in the $\LL_p$-maximal regularity space.
  Our results are based on a (partial) Fourier--Laplace transform of the problem in the half-space case, leading to elliptic boundary value problems which have to satisfy variants of the Lopatinskii--Shapiro condition to be solvable.
  The parabolic theory is then developed with use of tools from modern harmonic analysis such as the bounded $\Hcal^\infty$-calculus and the Kalton--Weis theorem.
 \end{abstract}
 


 \maketitle
 
 Version of \today.
 
\section{Introduction}
\label{sec:introduction}

 We will consider a quite general class of elliptic and parabolic systems with values in a general Banach space of class $\HTcal$.
 To give a first motivation for the relevance of the investigation of such systems, let us start with a rather simple example which illustrates the main point:
 From the theory of parabolic equations it is well-known that, for a bounded domain $\Omega \subseteq \R^d$ with smooth boundary, both the \emph{Dirichlet heat equation}
  \begin{equation}
   \begin{cases}
    \partial_t u_1 - d_1 \Delta u_1
     = f_1
     &\text{in } (0,T) \times \Omega,
     \\
    u_1
     = g_1
     &\text{on } (0,T) \times \partial \Omega,
     \\
    u_1(0,\cdot)
     = u_1^0
     &\text{in } \Omega
   \end{cases}
   \label{eqn:Dirichtlet_heat_eqn}
   \tag{D}
  \end{equation}
 (for some $d_1 > 0$) and the \emph{Neumann heat equation}
  \begin{equation}
   \begin{cases}
    \partial_t u_2 - d_2 \Delta u_2
     = f_1
     &\text{in } (0,T) \times \Omega,
     \\
    - d_2 \partial_{\vec n} u_2
     = g_2
     &\text{on } (0,T) \times \partial \Omega,
     \\
    u_2(0,\cdot)
     = u_2^0
     &\text{in } \Omega
   \end{cases}
   \label{eqn:Neumann_heat_eqn}
   \tag{N}
  \end{equation}
 (for some $d_2 > 0$) have the property of $\LL_p$-maximal regularity on the base space $\LL_q(\Omega)$, where we consider parameters $p, q \in (1, \infty)$.
 This means that \eqref{eqn:Dirichtlet_heat_eqn} has a unique solution $u_1$ in the class $\WW_p^1((0,T);\LL_q(\Omega)) \cap \LL_p((0,T);\WW_q^2(\Omega))$ if and only if the data $f_1$, $g_1$ and $u_1^0$ belong to the classes $f_1 \in \LL_p((0,T);\LL_q(\Omega))$, $g_1 \in \WW_p^{1-1/2q}((0,T);\LL_q(\partial \Omega)) \cap \LL_p((0,T);\WW_p^{2-1/q}(\partial \Omega))$ and $u_1^0 \in \WW_q^{2-2/p}(\Omega)$ and $u_1^0$ and $g_1$ fullfil a certain compatibility condition.
 Similarly, the problem \eqref{eqn:Neumann_heat_eqn} has a unique solution $u_2 \in \WW_p^1((0,T);\LL_q(\Omega)) \cap \LL_p((0,T);\WW_q^2(\Omega))$ if and only if $f_2 \in \LL_p((0,T);\LL_q(\Omega))$, $g_2 \in \WW_p^{1/2 - 1/2q}((0,T);\LL_q(\partial \Omega)) \cap \LL_p((0,T);\WW_q^{1-1/q}(\partial \Omega))$ and $u_2^0 \in \WW_q^{2-2/p}(\Omega)$.
 Combining these two scalar results, it easily follows that the two-component system of parabolic equations
  \[
   \begin{cases}
    \partial_t \vec u - \bb D \Delta \vec u
     = \vec f
     &\text{in } (0,T) \times \Omega,
     \\
    u_1
     = g_1
     &\text{on } (0,T) \times \partial \Omega,
     \\
    - d_2 \partial_{\vec n} u_2
     = g_2
     &\text{on } (0,T) \times \partial \Omega,
     \\
    \vec u(0,\cdot)
     = u^0
     &\text{on } \Omega,
   \end{cases}
  \]
 where $\vec u = (u_1, u_2)$, $\vec f = (f_1, f_2)$, $\vec u^0 = (u_1^0, u_2^0)$ are two-component vector-fields and $\bb D = \operatorname{diag}(d_1, d_2)$ is a diagonal matrix, has a unique solution in the class $\vec u \in \WW_p^1((0,T);\LL_q(\Omega;\K^2)) \cap \LL_p((0,T);\WW_q^2(\Omega;\K^2))$ if and only if $\vec f \in \LL_p((0,T);\LL_q(\Omega;\K^2))$, $\vec u^0 \in \WW_q^{2-2/p}(\Omega;\K^2))$ and the components $g_1$ and $g_2$ of the function $\vec g$ on the boundary satisfy $g_1 \in \WW_p^{1-1/2q}((0,T);\LL_q(\partial \Omega)) \cap \LL_p((0,T);\WW_q^{2 - 1/q}(\partial \Omega))$ and $g_2 \in \WW_p^{1/2-1/2q}((0,T);\LL_q(\partial \Omega)) \cap \LL_p((0,T);\WW_q^{1 - 1/q}(\partial \Omega))$ and the respective compatibility conditions.
 However, as soon as we allow for non-diagonal, but, say, symmetric and positive definite, $\bb D \in \R^{2 \times 2}$, the situation becomes less clear, as then the two equations in $\Omega$ do \emph{not decouple} anymore.
 Still, one can \emph{expect} maximal regularity also for this case.
 \newline
 As similar situation occurs, if not the equations in the domain are coupled, but the boundary conditions are of mixed type, say
  \[
   \begin{cases}
    \alpha_1 u_1 + \alpha_2 u_2
     = g_1
     &\text{on } (0,T) \times \partial \Omega,
     \\
    - \beta_1 d_1 \partial_{\vec n} u_1 - \beta_2 d_2 \partial_{\vec n} u_2
     = g_2
   \end{cases}
  \]
 for some parameters $\alpha_1$, $\alpha_2$, $\beta_1$ and $\beta_2 \in \R$.
 We will consider a general class of such parabolic systems.
 More precisely, we will consider systems of the abstract form
  \[
   \begin{cases}
    \partial_t u + \Acal(t,\vec x,\DD) u
     = f
     &\text{in } (0,T) \times \Omega,
     \\
    \mathcal{B}_j(t,\vec x,\DD) u
     = g_j
     &\text{on } (0,T) \times \partial \Omega, \, j = 1, \ldots, m,
     \\
    u(0,\cdot)
     = u^0
     &\text{in } \Omega.
   \end{cases}
  \]
 Here, $\Acal(t,\vec x,\DD)$ is a parameter-elliptic differential operator of order $2m$, where the coefficients may, in general, both depend on the time variable $t \in [0,T]$ and on the spatial position $\vec x \in \Omega$, whereas the operators $\Bcal_j(t,\vec x,\DD)$ are boundary differential operators of order $2m-1$ at most.
 The mixed type boundary conditions are reflected by the boundary symbols $\Bcal_j(t, \vec x, \vec \xi)$ which will not necessarily have a homogeneous (in $\vec \xi$) principal part $\Bcal^\#_j(t,\vec x, \vec \xi)$ -- in contrast to related results in \cite{DeHiPr03} and \cite{DeHiPr07}, which only apply for systems with homogeneous principle parts.
 Note that in the abstract formulation of the above presented second-order in space two-component example, we have $m = 1$, so that the two (Dirichlet and Neumann) boundary conditions will be described by only \emph{one} boundary operator, which thus includes both these scalar boundary conditions at once.
\newline
 Before presenting further instances, for which such mixed type boundary conditions will be of importance, let us give a short overview over the development of the theory of elliptic boundary value problems and parabolic initial--boundary value problems and its close connection to modern harmonic analysis:
 In the middle of the 20th century some for the further development of the theory of partial differential equations, especially evolutionary partial differential equations, fundamentally important works were written, which were groundbreaking not only for the linear theory, but in particular for the treatment of semi- and quasilinear, up to fully nonlinear parabolic systems. Ladyshenskaya, Solonnikov and Ural'ceva \cite{LaSoUr68} basically completely solved the question of $\LL_p$-maximal regularity for the case of sufficiently regular bounded domains and an exhaustive class of scalar or $\R^n$-valued functions.
 Up to today, this work represents a milestone in the theory of parabolic equations. At a similar time, de Simon \cite{Sim64} was able to successfully treat the case of abstract parabolic initial boundary value problems with values in an abstract Hilbert space $X$ instead of the concrete Lebesgue space $\LL_p(\Omega)$, but at first the question remained open whether and to what extent an extension of his results to more general abstract Banach spaces is possible.
 From then on, at the latest, the question of characterizing $\LL_p$-maximal regularity constituted an intensively studied area of research. Sobolevskii \cite{Sob64} was able to show that (for a fixed base space $X$) $\LL_p$-maximal regularity of an operator does not depend on the chosen parameter $p \in (1, \infty)$, i.e.\ if an operator has $\LL_p$-maximal regularity for one $p \in (1, \infty)$, then it already has $\LL_q$-maximal regularity for all $q \in (1, \infty)$.
With the sum method for real interpolation spaces of da Prato and Grisvard \cite{PraGri75} rich subclasses of operators with the maximal regularity could be identified.
Using a functional calculus, Dore and Venni \cite{DorVen87}, Pr\"uss and Sohr \cite{PruSoh90} were able to show that any operator with bounded imaginary power of power angle $\theta_A < \frac{\pi}{2}$ has $\LL_p$-maximal regularity, provided the Banach space $X$ under consideration belongs to the class $\HTcal$, i.e., Banach spaces with a bounded Hilbert transform (which -- as it turned out in works by Bourgain \cite{Bou83} and Burkholder \cite{Bur83} -- are exactly the UMD spaces). 
Meanwhile, further sufficient conditions for $\LL_p$-maximal regularity have been provided for various relevant cases, including boundedness of the semigroup associated to an abstract Cauchy problem on all spaces $\LL_q(\Omega)$ for $1 \leq q \leq \infty$ by Lamberton \cite{Lam87} or Gaussian kernel estimate problems \cite{HiePru97}, \cite{CouDuo00}.
The question of actual characterization of $\LL_p$-maximal regularity, however, proved to be much more difficult: Kalton and Lancien \cite{KalLan01} first showed that de Simon's technique cannot be applied to abstract Banach spaces of class $\HTcal$, consequently not to more general Banach spaces.
For a generalisation of de Simon's results it is mandatory that the basic space $X$ is at least topologically \emph{isomorphic} to a Hilbert space, which is a far too restrictive constraint for many concrete problems. The breakthrough then came with Weis \cite{Wei01}, who was able to characterize maximal regularity of operators on Banach spaces of type $\HTcal$ completely via $\Rcal$-boundedness of the set $\{\ii \rho (\ii \rho + A)^{-1}: \, \rho \in \R \setminus \{0\} \}$. This property, called \emph{Riesz-property} or (later) \emph{randomized boundedness}, provides just the tightening of the notion of uniform boundedness as a requirement in the Mikhlin multiplier theorem, which is necessary for its generalisation to Banach spaces of class $\HTcal$.
Simultaneously, for families of bounded linear operators this gave a new approach to the proof of maximal regularity; namely, Kalton and Weis \cite{KalWei01} showed how to obtain maximal regularity for $\Rcal$-sectorial operators using the $\Hcal^\infty$-calculus.
 Based on these developments in modern harmonic analysis and multiplier theory on Banach spaces of class $\HTcal$, see, e.g., Berksen and Gillespie \cite{BerGil94}, Cl\'ement, de Pagter, Sukochev and Witvliet \cite{CleDePSukWit00}, Kalton and Weis \cite{KalWei01}, Denk, Hieber and Pr\"uss \cite{DeHiPr03},\cite{DeHiPr07} extended a large class of $\R^n$-valued results by Ladyshenskaya, Solonnikov and Ural'ceva \cite{LaSoUr68} on linear parabolic systems on bounded domains to the general vector-valued case.
 The boundary conditions considered there mainly includes boundary conditions of the same type, so either on the Dirichlet data of the functions itself or on their normal flux through the boundary, which corresponds to homogeneity (w.r.t.\ $\vec \xi$) of the principal part of the boundary symbols.
 \newline
 In a recent paper \cite{AugBot21a}, cf.\ Example \ref{exa:fast_surface_chemistry} below, however, while performing an analysis of fast-surface chemistry and fast-adsorption limit models for heterogeneous catalysis systems, we encountered a situation where the boundary conditions have a combined type:
 Instead of a Dirichtlet or Neumann condition in all components equally, a no-flux boundary condition occurred at certain linear combinations of the individual normal fluxes, i.e., the normal direction of the gradient of the individual components, together with nonlinear boundary conditions at the boundary trace of the individual concentrations due to equilibrium conditions in extremely fast surface chemical reactions.
 Though that problem can be solved by employing the results and techniques of Ladyshenskaya, Solonnikov and Uralceva \cite{LaSoUr68}, which are valid for finite-component systems of parabolic type, this problem stimulated the search for more abstract and more general results on vector-valued parabolic and elliptic boundary value problems in the spirit of \cite{DeHiPr03} and \cite{DeHiPr07}.
In contrast to the situation where, say the boundary is decomposed into a Dirichlet and a Neumann part, cf.\ e.g.\ the works by Auscher et al.\ \cite{Auscher_et_al_2015}, Egert \cite{Egert_2018}, Tolksdorff \cite{Tolksdorff_2018}, Bechtel, Egert and Haller-Dintelmann \cite{BeEgHa_2020} and related works, there seems to be no extension of these results for this particular \emph{mixed type} boundary conditions case, and we want to close this gap in the present manuscript.
 \newline
 As a further motivation for the subsequent development of an abstract theory, and before we comment on the organisation of this manuscript, let us shortly discuss more examples, where such mixed type boundary condition play a role, and for which the results of this paper shall be valuable.
 
 \begin{example}[Fast surface chemistry limits]
 \label{exa:fast_surface_chemistry}
  In \cite{AugBot21} and \cite{AugBot21a}, we derived and analysed a class of reaction-diffusion-systems inspired by heterogeneous catalysis, which includes, for example, the following three component model problem \cite{AugBot21}
   \[
    \begin{cases}
     \partial_t c_i - d_i \Delta c_i
      = 0,
      &i= 1, 2, 3, \, t \geq 0, \, \vec x \in \Omega
      \\
     \kappa_f c_1 c_2 - \kappa_b c_3
      = 0,
      &t \geq 0, \, \vec x \in \Sigma,
      \\
     - d_1 \partial_\nu c_1 + d_2 \partial_\nu c_2
      = 0,
      &t \geq 0, \, \vec x \in \Sigma,
      \\
     - d_1 \partial_\nu c_1 - d_2 \partial_\nu c_3
      = 0,
       &t \geq 0, \, \vec x \in \Sigma,
      \\
     c_i(0,\cdot)
      = c_i^0,
      &\vec x \in \Omega
    \end{cases}
   \]
  on some bounded domain $\Omega \subseteq \R^n$ with regular boundary $\Sigma = \partial \Omega$, and parameters $k_f$, $k_b$ and $d_i > 0$, $i = 1, 2, 3$.
  This model describes a chemical process inside a reactor with catalytic boundary $\Sigma$, at which an instantaneous reversible chemical reaction of the form $A_1^\Sigma + A_2^\Sigma \rightleftharpoons A_3^\Sigma$ takes place.
  A linearisation of this system leads to a system of parabolic-initial boundary value problems of the form
   \begin{equation}
    \begin{cases}
    \partial_t v_i - d_i \Delta v_i
     = f_i,
     &i = 1, 2, 3, \, t \geq 0, \, \vec x \in \Omega,
     \\
    \alpha_1 v_1 + \alpha_2 v_2 + \alpha_3 v_3
     = g_1,
     &t \geq 0, \, \vec x \in \Sigma,
     \\
    \beta_1 \partial_\nu v_1 + \beta_2 \partial_\nu v_2 = g_2,
     &t \geq 0, \, \vec x \in \Sigma,
     \\
    \gamma_1 \partial_\nu v_1 + \gamma_3 \partial_\nu v_3 = g_3,
     &t \geq 0, \, \vec x \in \Sigma,
     \\
    v_i(0,\cdot)
     = v_i^0,
     &i = 1, 2, 3, \, \vec x \in \Omega,
     \end{cases}
     \label{ast}
   \end{equation}
  and, evidently, the orders of the boundary differential operator differs between the second and the two subsequent lines.
  Similar parabolic systems with mixed type boundary conditions naturally appear also for more general reaction networks on the surface, see \cite{AugBot21a}.
  Whereas for the particular model \eqref{ast} one may exploit the special structure of the boundary conditions on the surface to derive maximal regularity from the scalar case, cf.\ the strategy used in \cite[Theorem 3.2]{AugBot21}, this will no longer be true for the latter, more general class of reaction-networks.
 \end{example} 
 
 \begin{example}[Navier--Stokes equations]
  In \cite{BotKoePru13}, the authors consider several energy-preserving boundary conditions of the abstract form $\mathcal{B}(u,p) = 0$ for the incompressible Navier--Stokes equation
   \[
    \begin{cases}
     \rho \partial_t \vec u + \dv (\rho \vec u \otimes \vec u - \bb S)
      = \rho f,
      &t \in (0,T), \, \vec x \in \Omega,
      \\
     \dv \vec u
      = 0,
      &t \in (0,T), \, \vec x \in \Omega,
      \\
     \vec u(0,\cdot)
      = \vec u^0
      &\vec x \in \Omega.
    \end{cases}
   \]
  The classes of boundary conditions considered there include
   \begin{itemize}
    \item
     \emph{no-slip boundary conditions} leading to the homogeneous Dirichlet boundary condition
      \[
       \vec u = \vec 0
        \quad
        \text{on } \Omega \times (0,T);
      \]
     \item
      \emph{free slip (or perfect slip) boundary conditions} which lead to mixed type boundary conditions of the form
       \[
        \vec u \cdot \vec n = 0
         \quad \text{and} \quad
        2 \mu P_\Gamma \bb D \nu = \vec 0
         \quad \text{on} \quad \partial \Omega \times (0,T).
       \]
      Here, $P_\Gamma(\vec x) = \bb I - \vec n(\vec x) \otimes \vec n(\vec x)$ is the orthogonal projection onto the tangent space $T_x \Gamma$ of $\Gamma = \partial \Omega$ at $\vec x$ and $\bb D = \frac{1}{2} \big( \nabla u + (\nabla u)^\mathsf{T} \big)$ is the \emph{rate of deformation tensor};
     \item
      \emph{Navier boundary conditions}
       \[
        \vec u \cdot \vec n = 0
         \quad \text{and} \quad
        P_\Gamma \vec u + \alpha P_\Gamma \bb S \vec n = \vec 0
         \quad \text{on} \quad \partial \Omega \times (0,T),
       \]
      where $\alpha > 0$ and $\bb S = 2 \mu \bb D - p \bb I$ is the total stress tensor;
     \item
      \emph{Neumann boundary conditions}
       \[
        \bb S \vec n
         = \vec 0
         \quad \text{on} \quad \partial \Omega \times (0,T).
       \]
   \end{itemize}
  The boundary symbols related to the Dirichlet resp.\ Neumann boundary conditions are homogeneous of order $0$ resp.\ order $1$.
  On the other hand, free slip conditions and the Navier boundary conditions are -- in general -- not homogeneous (neither of order $0$ nor $1$) due to the involvement of the orthogonal projection $P_\Gamma$ on the tangent space.
  \newline
  In this sense, mixed type boundary conditions appear naturally in the theory of Navier--Stokes equations.
 \end{example}
 
 \begin{example}[Pattern formation]
  In a recent research article \cite{KrKlMaHeGa21}, the authors motivate and investigate the pattern which are formed for a particular two component reaction-diffusion system of the form
   \[
    \begin{cases}
     \partial_t u - d_u \Delta u
      = f(u,v)
      &\text{in } (0,T) \times \Omega,
      \\
     \partial_t v - d_v \Delta v
      = g(u,v)
      &\text{in } (0,T) \times \Omega,
    \end{cases}
   \]
  depending on the type of boundary conditions which are imposed.
  One of the several types of boundary conditions considered there is given by the mixed type boundary condition
   \[
    u = R
     \quad \text{and} \quad
    \partial_{\vec n} v = 0
    \quad \text{on } (0,T) \times \partial \Omega,
   \]
  where $R \in \R$ is a constant.
  In their derivation of such specific boundary conditions, they make a case for their relevance.
  Note that in this particular example, maximal regularity for the quasi-autonomous linear system follows from the scalar theory, since the diffusion process is modelled by Fickean diffusion.
  However, as soon as cross-diffusion is allowed for, in the sense of a Fick--Onsager or Maxwell--Stefan diffusion model, this approach is not available.
 \end{example}
 
 \begin{example}[Thermo-diffusion systems]
  For thermo-diffusive reaction-diffusion systems like
   \[
    \begin{cases}
     \partial_t \vartheta - \kappa \Delta \vartheta
      = h(\vartheta, \vec c),
      &t \geq 0, \, \vec x \in \Omega,
      \\
     \partial_t \vec c - \dv (\bb D \nabla \vec c)
      = \vec r(\vartheta, \vec c),
      &t \geq 0, \, \vec x \in \Omega
    \end{cases}
   \]
  it is reasonable to fix boundary conditions as follows: For the temperature $\vartheta$, one might, for example, fix the temperature at the boundary $\partial \Omega$, i.e.\ impose Dirichlet boundary conditions $\vartheta|_\Sigma = \Theta$,  whereas for the concentrations $c_i$ it might be more natural to impose no-flux boundary conditions or controlled in-flow conditions, i.e.\ $- \bb D \partial_{\vec n} \vec c = \bb g$.
  Again, this approach leads to mixed type boundary conditions for the combined vector variable $(\vartheta, \vec c)$.
 \end{example}
 
 With these examples in mind, let us comment on the structural restrictions imposed in \cite{DeHiPr03} and \cite{DeHiPr07}:
 The $\LL_p$-maximal regularity results in \cite[Theorem 7.11]{DeHiPr03} (half-space problem) and \cite[Theorem 8.2]{DeHiPr03} (domains with compact boundary), both rely on a version of the Lopatinskii-Shapiro condition in which all the boundary operators $\Bcal_j$, $j = 1, \ldots, N$, are \emph{homogeneous} of degree $m_j \in \N_0$ up to some lower order perturbation. Therefore, in the definition of the principle symbol as used there, e.g.\ in \cite[Section 8.1]{DeHiPr03}, the principle part of the boundary operators for \emph{mixed type} boundary conditions could not 'see' the Dirichlet type contribution, as their order of differentiation is $0$, i.e.\ strictly smaller than the order $1$ for the Neumann type parts.
 However, this can \emph{not} be considered as a relatively compact, lower order perturbation of the first order part of the boundary operator $\mathcal{B}_j$. We, therefore, extend the results in \cite{DeHiPr03} and \cite{DeHiPr07} to boundary conditions of such mixed type. For this purpose, we have to closely follow the lines of \cite[Chapter 6--8]{DeHiPr03} and \cite{DeHiPr07}, and observe which adjustments are needed to transfer the results there to this more general and, thus, more flexible setup.
 In particular, we have to introduce range conditions for lower order terms which respect the structure of the boundary conditions.
 Note that the latter obstacle does not appear in the case of homogeneous principal parts, since in that case it is clear what should be a lower order perturbation of the boundary condition (namely, all linear boundary operators of strictly lesser differentiability order than the principle part).
\newline
To conclude this introductory section, let us now give a brief overview over this manuscript:
Since we will need notation and results from modern harmonic analysis, in Section \ref{sec:Notation_and_preliminaries} we begin by providing the basic notation as well as introducing important notions such as $\Rcal$-boundedness, $\Rcal$-sectoriality and the Kalton--Weis generalisation of Mikhlin's multiplier theorem.
Also, we present a useful auxiliary result on holomorphic operator families on sectors, see Proposition \ref{prop:R-boundedness_sectors}, where $\Rcal$-boundedness of the family of operators $T(\lambda)$ on one sector implies $\Rcal$-boundedness of $\lambda T'(\lambda)$ on any smaller sector.
In Section \ref{Sec:Elliptic_differential_operators_full_space}, we review the existing theory of vector-valued elliptic differential operators on the full space $\R^n$, where the solution can be constructed via the Fourier transform on the full-space $\R^n$, see Theorem \ref{thm:full-space_problem}.
Then, in Section \ref{Sec:Half-space_problem} we outline the basic ideas for solving the parabolic initial-boundary value problem on the half-space $\R_+^{n+1}$ by employing the Laplace transform with respect to time $t > 0$, and partial Fourier transform with respect to the first $n$ variables $\vec x' = (x_1, \ldots, x_n)$ of the vector $\vec x \in \R_+^{n+1}$, leading to a (vector-valued) system of ordinary differential equations, which can be solved by rewriting it as a first order system in $y := x_{n+1}$ system, provided the \emph{Lopatinskii--Shapiro condition} is met.
The solution operator constructed in this way formally provides the solution of the corresponding elliptic boundary value problem, so that (in view of the multiplier theorems) we need to establish kernel estimates and $\Rcal$-bounds for these solution operators, establishing maximal regularity for parabolic initial-boundary value problems on the half-space.
Then, in Section \ref{Sec:General_domains} by localisation techniques, we may transfer the previous results to general domains with a compact and sufficiently regular boundary.
Whereas the previous sections have been restricted to zero initial data $u(0,\cdot) = 0$, in the last section, Section \ref{Sec:Inhomogeneous_IV}, we extend the results to non-zero initial data and give the corresponding optimal regularity results.
 
\section{Notation and preliminaries}
\label{sec:Notation_and_preliminaries}

\subsection{Notation}
In this manuscript, all Banach spaces under consideration will be Banach spaces over the field $\C$ of complex numbers, and will typically be denoted by capital Roman letters.
Some of the results, however, transfer to the case of real Banach spaces $E$, possibly via complexification of the Banach space, cf.\ e.g.\ \cite{Amann_1995}.
A very important subclass of these Banach spaces consists of the \emph{Banach spaces $E$ of class $\HTcal$} which are characterised by the condition that the \emph{Hilbert transform}
 \[
  (H f)(x)
   := \frac{1}{\pi} \big( \operatorname{PV} \frac{1}{t} \ast f \big)(x)
   = \frac{1}{\pi} \lim_{R \rightarrow \infty} \lim_{\varepsilon \rightarrow 0+} \int_{(-R,-\varepsilon) \cup (\varepsilon,R)} \frac{1}{t} f(x - t) \, \dd t,
   \quad
   f \in \mathcal{S}(\R;E), \, \vec x \in \R
 \]
extends to a bounded linear operator on the $E$-valued Lebesgue-Bochner space $\LL_p(\R;E)$ for some (and then, for all) $p \in (1, \infty)$. Important examples of such Banach spaces of class $\HTcal$ are given by Hilbert spaces (or Banach spaces which are isomorphic to a Hilbert space, which includes any finite dimensional Banach space) and the standard Lebesgue spaces $\LL_q(\Omega)$, $q \in (1,\infty)$, where the set $\Omega \subseteq \R^n$ is open and non-empty, see e.g.\ \cite[Propositions 4.2.14 and 4.2.15]{HvNVW_2016}.
Also, if $E$ is a Banach space of class $\HTcal$, so is the Lebesgue--Bochner space $\LL_p(\Omega;E)$ for any open set $\Omega \subseteq \R^n$ and $p \in (1, \infty)$.
 \begin{lemma}
 \label{lem:HTcal-L_p-HTcal}
  Let $E$ be a Banach space of class $\HTcal$ and $p \in (1, \infty)$.
  Then the Lebesgue--Bochner space $\LL_p(\Omega;E)$ is a Banach space of class $\HTcal$.
 \end{lemma}
 For a proof of this result, see \cite[Proposition 4.2.15(b)]{HvNVW_2016}.

Given two Banach spaces $X$ and $Y$, we denote by $\B(X,Y)$ the Banach space of bounded linear operators $T: X \rightarrow Y$.
Moreover, $\rho(A) = \{ \lambda \in \C: (\lambda I - A): \dom(A) \rightarrow X \text{ is bijective and } (\lambda I - A)^{-1} \in \B(X) \}$ and $\sigma(A) = \C \setminus \rho(A)$ denote the resolvent set and the spectrum of a linear operator $A: \dom(A) \subseteq X \rightarrow X$, and $\dom(A)$, $\ran(A)$ and $\ker(A)$ the domain, range and kernel, resp., of a linear operator $A: \dom(A) \subseteq X \rightarrow Y$.
\newline
For the ($E$-valued) spaces of continuous functions $\CC(\overline{\Omega};E)$, of $k$-times continuously differentiable functions $\CC^k(\overline{\Omega};E)$, Lebesgue-Bochner spaces $\LL_p(\Omega;E)$, Sobolev spaces $\WW_p^k(\Omega;E)$, Sobolev-Slobodetskii spaces $\WW_p^s(\Omega;E)$ and Bessel-potential spaces $\HH_p^s(\Omega;E)$, we use standard notation.
 Moreover, by
  \[
   \CC_\mathrm{l}(\overline{\Omega})
    := \{f \in \CC(\overline{\Omega}): \, \lim_{\abs{\vec x} \rightarrow \infty} f(\vec x) =: f(\infty) \text{ exists, if } \Omega \text{ is unbounded} \}
  \]
 we mean the functions which are continuous on $\Omega$ up to its boundary, and which additionally converge as $\abs{\vec z} \rightarrow \infty$, provided $\Omega$ is unbounded.
 \newline
 Later on, inhomogeneous Lebesgue spaces $\LL_{p,q}(J \times \Omega) := \LL_p(J; \LL_q(\Omega))$ for $p, q \in [1, \infty)$ and any proper interval $J \subseteq \R$ and Sobolev spaces or Sobolev--Slobodetskii spaces $\WW_{p,q}^{(\tau,\sigma)}(J \times \Omega) := \WW_p^\tau(J; \LL_q(\Omega)) \cap \LL_p(J; \WW_q^\sigma(\Omega))$ will appear, as well as Besov spaces $\BB_{p,q}^s(\Omega)$ and Triebel--Lizorkin spaces $\FF_{p,q}^s(\Omega)$.
 As before, we denote by $\LL_{p,q}(J \times \Omega; E)$ etc.\ their vector-valued versions.
 Moreover, sometimes we write $\WW_{p,q}^{(\tau,\sigma) \cdot s}$ for $\WW_{p,q}^{(\tau s, \sigma s)}$.
 For more information on these spaces, we refer to the rich existing literature, e.g.\ \cite{AdaFou03}, \cite{Bre11}, \cite{Tar07}, \cite{Tri10}, \cite{Tri92}, \cite{Tri06}, \cite{Tri20} and \cite{Zie89}.

\subsection{Preliminaries}

Since the results will heavily rely on modern results on harmonic analysis, especially the theory of $\Rcal$-bounded operator families and its relation to maximal regularity of abstract Cauchy problems, let us in a preliminary section recall some of the (for our purpose) most relevant results in this perspective.
A fundamental role will be played by $\Rcal$-bounded subsets in a space of bounded linear operators, cf.\  \cite[Definition 5.3.13]{HvNVW_2016}.

\begin{definition}
 Let $X$ and $Y$ be Banach spaces and $\Tcal \subseteq \B(X,Y)$ be a family of bounded linear operators.
 $\Tcal$ is called \emph{$\Rcal$-bounded} if for some (then: for all) $p \in [1, \infty)$ there is a constant $C_p > 0$ such that for every finite subfamily  $\{T_\nu\}_{\nu=1}^N \subseteq \Tcal$, vectors $\{x_\nu\}_{\nu=1}^N \subseteq X$, and independent random variables $\{\varepsilon_\nu\}_{\nu=1}^N$ with $\bb P(\varepsilon_\nu = \pm 1) = \frac{1}{2}$ on some probability space $(\Omega, \mathcal{M}, \mu)$, for arbitrary $N \in \N$, the estimate
  \begin{equation}
   \norm{\sum_{\nu=1}^N \varepsilon_\nu T_\nu x_\nu}_{\LL_p(\Omega;Y)}
    \leq C_p \norm{\sum_{\nu=1}^N \varepsilon_\nu x_\nu}_{\LL_p(\Omega;X)}
    \label{eqn:R-bound}
  \end{equation}
 is valid.
 In this case, we write
  \[
   \Rcal(\Tcal)
    = \Rcal_p(\Tcal)
    = \inf \{ C_p > 0: \, \eqref{eqn:R-bound} \text{ is valid } \}
  \]
 for its $\Rcal$-bound.
\end{definition}

\begin{remark}
\label{rem:Rcal-boundedness-properties}
 Note the following:
 \begin{enumerate}
  \item
   If $X$ is (topologically isomorphic to) a Hilbert space, $\Tcal \subseteq \B(X,Y)$ is $\Rcal$-bounded if and only if $\Tcal \subseteq \B(X,Y)$ is (uniformly) bounded, cf.\ the remarks right after \cite[Definition 5.3.13]{HvNVW_2016}.
   For general Banach spaces, this equivalence is no longer true, but in any case an $\Rcal$-bounded operator family is always (uniformly) bounded.
  \item
   If $X = Y = \LL_p(\Omega)$ for some open $\Omega \subseteq \R^n$, then via the \emph{Khintchine inequality}, see \cite[Theorem 3.2.23]{HvNVW_2016}, $\Rcal$-boundedness can be characterised by the \emph{square-function estimate}
    \[
     \norm{ \big( \sum_{\nu=1}^N \abs{T_\nu x_\nu}^2 \big)^{1/2}}_{\LL_p(\Omega)}
      \leq M_p \norm{ \big( \sum_{\nu=1}^N \abs{x_\nu}^2 \big)^{1/2}}_{\LL_p(\Omega)}
    \]
   for any finite $\{T_\nu\}_{\nu=1}^N \subseteq \Tcal$ and $\{x_\nu\}_{\nu=1}^N \subseteq \LL_p(\Omega)$, cf.\ \cite[Theorem 1.d.6]{LinTza79} for more general ($q$-concave for some $q < \infty$) Banach spaces.
  \item 
   For any fixed $p \in [1, \infty)$, the $\Rcal$-bounds are subadditive and submultiplicative \cite[Remark 5.3.14(1)]{HvNVW_2016}:
    \begin{align*}
     \Rcal(\Scal + \Tcal)
      &\leq \Rcal(\Scal) + \Scal(\Tcal)
      &&\text{for } \Scal, \Tcal \subseteq \B(X,Y),
      \\
     \Rcal(\Scal \Tcal)
      &\leq \Rcal(\Scal) \Rcal(\Tcal)
      &&\text{for } \Tcal \subseteq \B(X,Y), \, \Scal \subseteq \B(Y,Z).
    \end{align*}
  \item
   If $\Tcal$, $\Tcal_n \subseteq \B(X,Y)$ are such that
    \[
     \sup_{n \geq 1} \Rcal(\Tcal_n) < \infty
      \quad \text{and} \quad
      \lim_{n \rightarrow \infty} \sup_{T \in \Tcal} \inf_{T_n \in \Tcal_n} \norm{T - T_n}_{\B(X;Y)} = 0,
    \]
   then
    \[
     \Rcal(\Tcal)
      \leq \limsup_{n \rightarrow \infty} \Rcal(\Tcal_n)
      < \infty.
    \]
 \end{enumerate}
\end{remark}

\begin{proof}
 We only prove the last statement here.
 Let $N \in \N$ be an arbitrary natural number, $\{ x_\nu \}_{\nu = 1, \ldots, N} \subseteq X$, $\{T_\nu \}_{\nu = 1, \ldots, N} \subseteq \Tcal$ and $\{\varepsilon_\nu \}_{\nu = 1, \ldots, N }$ i.i.d.\ random variables with $\bb P(\varepsilon_\nu = \pm 1) = \frac{1}{2}$ on common probability space $(\Omega, \mathcal{M}, \mu)$.
 Then, setting $C_n := \Rcal(\Tcal_n) + \frac{1}{n}$, and choosing $T_{n,\nu} \in \Tcal_n$ such that $\norm{T_\nu - T_{n,\nu}}_{\B(X;Y)} \leq \inf_{S_n \in \Tcal_n} \norm{T_\nu - S_n} + \frac{1}{n}$, we may conclude that
  \begin{align*}
   \norm{ \sum_{\nu=1}^N \varepsilon_\nu T_\nu x_\nu }_{\LL_p(\Omega;Y)}
    &\leq \norm{ \sum_{\nu = 1}^N \varepsilon_\nu T_{n,\nu} x_\nu }_{\LL_p(\Omega;Y)}
     + \norm{ \sum_{\nu = 1}^N \varepsilon_\nu (T_\nu - T_{n,\nu}) x_\nu }_{\LL_p(\Omega;Y)}
     \\
    &\leq C_n \norm{ \sum_{\nu = 1}^N \varepsilon_\nu x_\nu }_{\LL_p(\Omega;X)}
     + \sum_{\nu=1}^N \norm{T_\nu - T_{n,\nu}}_{\B(X)} \norm{\varepsilon_\nu x_\nu}_{\LL_p(\Omega;X)}.
  \end{align*}
 By assumption, the second sum tends to zero as $n \rightarrow \infty$, whereas the first term can be estimated by \[ \limsup_{n \rightarrow \infty} C_n \cdot \norm{ \sum_{\nu = 1}^N \varepsilon_\nu x_\nu }_{\LL_p(\Omega;X)} = \limsup_{n \rightarrow \infty} \Rcal(\Tcal_n) \norm{ \sum_{\nu = 1}^N \varepsilon_\nu x_\nu }_{\LL_p(\Omega;X)}. \]
 Hence, $\Rcal(\Tcal) \leq \limsup_{n \rightarrow \infty} \Rcal(\Tcal_n)$ and, in particular, $\Tcal$ is $\Rcal$-bounded.
\end{proof}

\begin{corollary}
\label{cor:R-uniform_Neumann}
 Let $\Lambda$ be an arbitrary index set and $\{ T_\lambda: \lambda \in \Lambda \} \subseteq \B(X)$ be a family of bounded linear operators on a Banach space $X$.
 Assume that
  \[
   \Rcal (\{ T_\lambda: \lambda \in \Lambda \})
    =: \rho
    < 1.
  \]
 Then, the resolvent operator $R_\lambda := (I - T_\lambda)^{-1} \in \B(X)$ exists for all $\lambda \in \Lambda$ and
  \[
   \Rcal (\{ R_\lambda: \lambda \in \Lambda \})
    \leq \frac{1}{1-\rho}.
  \]
\end{corollary}

\begin{proof}
 Since any $\Rcal$-bounded family is uniformly bounded with uniform bound less or equal the $\Rcal$-bound, $\sup_{\lambda \in \Lambda} \norm{T_\lambda}_{\B(X)} \leq \rho < 1$, so that the Neumann series $(I - T_\lambda)^{-1} = \sum_{\nu=0}^\infty T_\lambda^\nu$ converges in $\B(X)$ for all $\lambda \in \Lambda$.
 Since the $\Rcal$-bounds are subadditive, for every finite $K \in \N$,  and submultiplicative, we obtain that
  \begin{align*}
   \Rcal (\{ \sum_{\nu=0}^K T_\lambda^\nu: \, \lambda \in \Lambda \})
    &\leq \sum_{\nu=0}^K \Rcal (\{ T_\lambda^\nu: \, \lambda \in \Lambda \})
    \leq \sum_{\nu=0}^K \Rcal (\{ T_\lambda: \, \lambda \in \Lambda \})^\nu
    \\
    &\leq \sum_{\nu=0}^K \rho^\nu = \frac{1-\rho^{K+1}}{1-\rho}.
  \end{align*}
 Therefore, by Remark \ref{rem:Rcal-boundedness-properties}
  \begin{align*}
   \Rcal (\{ R_\lambda: \, \lambda \in \Lambda \})
    &= \Rcal (\{ \sum_{\nu=0}^\infty T_\lambda^\nu: \, \lambda \in \Lambda \})
    \leq \limsup_{K \rightarrow \infty} \Rcal (\{ \sum_{\nu=0}^K T_\lambda^\nu: \, \lambda \in \Lambda \})
    \\
    &\leq \limsup_{K \rightarrow \infty} \frac{1 - \rho^{K+1}}{1 - \rho}
    = \frac{1}{1 - \rho}.
  \end{align*}
\end{proof}

\begin{remark}
 In Corollary \ref{cor:R-uniform_Neumann}, it suffices to assume that
  \[
   \limsup_{\nu \rightarrow \infty} \big( \mathcal{R} \{ T_\lambda^\nu: \, \lambda \in \Lambda \} \big)^{1/\nu} < 1.
  \]
 (The left-hand side may be seen as a $\Rcal$-boundedness based version of the spectral radius.)
 In that case there are $\nu_0 \in \N$, $C_0 \geq 0$ and $\rho \in (0,1)$ such that $\big( \mathcal{R} \{ T_\lambda^\nu: \, \lambda \in \Lambda \} \big)^{1/\nu} \leq \rho$ for all $\nu \geq \nu_0$ and $\sum_{\nu = 0}^{\nu_0 - 1} \mathcal{R} \{ T_\lambda^\nu: \, \lambda \in \Lambda \} \leq C_0$, and in this case,
  \[
   \mathcal{R} \{ R_\lambda: \, \lambda \in \Lambda \}
    \leq C_0 + \frac{1}{1 - \rho}.
  \]
\end{remark}

Lutz Weis showed that, for Banach spaces of class $\HTcal$, Mikhlin's Multiplier Theorem can be generalised from the scalar case by demanding $\Rcal$-boundedness of the (now operator-valued) functions $\abs{\vec \xi}^{\abs{\vec \alpha}} \DD^{\vec \alpha} M(\vec \xi)$ on $\R^n \setminus \{0\}$ instead of merely uniform boundedness as in the classical Mikhlin Multiplier Theorem.
For a proof of this result, see \cite[Theorem 4.4]{StrWei07} (or \cite[Theorem 3.25]{DeHiPr03} for the case $M \in \CC^n(\R^n \setminus \{0\}; \B(X,Y))$).

\begin{theorem}[Mikhlin's Theorem in $n$ variables]
 Let $X, Y$ be Banach spaces of class $\HTcal$.
 Assume that a function $M: \R^n \setminus \{0\} \rightarrow \B(X,Y)$ has distributional derivates $\DD^{\vec \alpha} M$ which for each $\vec \alpha \in \{0,1\}^n$ can be represented by a function and satisfy
  \[
   \Rcal( \{ \abs{\vec \xi}^{\abs{\vec \alpha}} \DD^{\vec \alpha} M(\vec \xi): \quad \vec \xi \in \R^n \setminus \{0\} \})
    =: \kappa_{\vec \alpha}
    < \infty
  \]
 for all $\vec \alpha \in \{0,1\}^n$.
 Then the operator $T_M$ defined by $T_M f := \Fcal^{-1} M \Fcal f \in \Scal'(\R^n;Y)$ for all $\Fcal f \in \Dcal(\R^n;X)$ (here, $\Fcal$ denotes the Fourier transform on $\R^n$) extends to a bounded linear operator $\LL_p(\R^n;X) \rightarrow \LL_p(\R^n; Y)$ with operator norm $\norm{T_M}_{\B(\LL_p(\R^n;X);\LL_p(\R^n;Y))} \leq C \sum_{\vec \alpha \in \{0,1\}^n} \kappa_{\vec \alpha}$.
\end{theorem}

As a rule of thumb, the $\Rcal$-boundedness of the function itself typically is the most critical step in checking the assumptions of Mikhlin's theorem. In fact, we can provide the following result for operator families which are holomorphic and $\Rcal$-bounded on sectors.
We did not find this precise statement in the existing literature, although Robert Denk draw our attention to a paper of Kunstmann and Weis, who used a similar technique in the proof of \cite[Example 2.16]{KunWei14}.
\newline
For every $\vartheta \in (0, \pi]$ we define the \emph{sector} $\Sigma_\vartheta \subseteq \C$ as
  \[
   \Sigma_\vartheta
    = \{ z \in \C \setminus \{0\}: \, \abs{\arg(z)} < \vartheta \},
  \]
see Figure \ref{fig:diagram_sector}.
 \begin{figure}
	\centering
	\includegraphics[scale = 1]{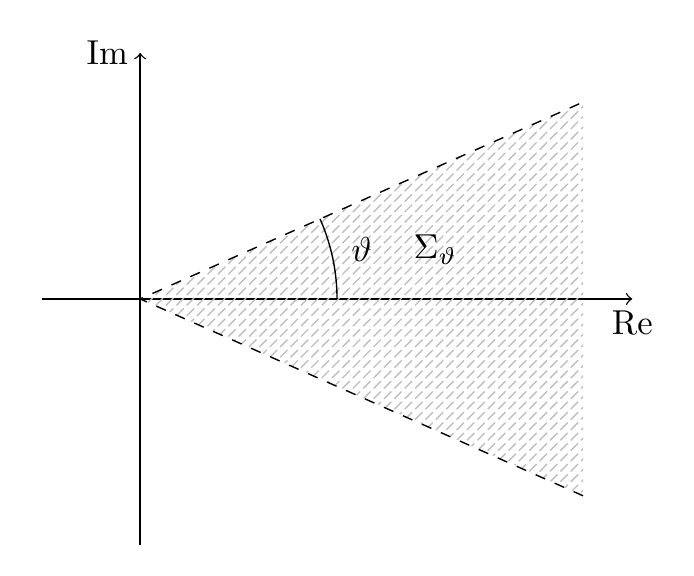}
	\caption{A sector in $\C$.}
	\label{fig:diagram_sector}
 \end{figure}

\begin{proposition}
\label{prop:R-boundedness_sectors}
 Let $X$ be a complex Banach space, $\phi \in (0, \pi)$ and let $T \in \CC(\overline{\Sigma_{\pi-\phi}};\B(X)) \cap \Hcal(\Sigma_{\pi-\phi};\B(X))$ be holomorphic on a sector and continuous up to its boundary.
 If
  \[
   \Rcal (\{ T(\lambda): \, \lambda \in \Sigma_{\pi-\phi} \})
    =: C
    < \infty,
  \]
 then for all $\phi' \in (\phi, \pi)$ it holds that
  \[
   \Rcal (\{ \lambda T'(\lambda): \, \lambda \in \Sigma_{\pi-\phi'} \})
    \leq \frac{C}{\sin(\phi' - \phi)^2}.
  \]
 Consequently, for every $k \in \N$ and $\phi' \in (\phi, \pi)$, $\Rcal (\{ \lambda^k T^{(k)}(\lambda): \, \lambda \in \Sigma_{\pi-\phi'} \}) < \infty$.
\end{proposition}

 \begin{remark}
  Note that it suffices to demand weak holomorphy since every weakly holomorphic function $T: \Sigma_{\pi-\phi} \rightarrow \B(X)$ is already holomorphic.
 \end{remark}

 \begin{proof}[Proof of Proposition \ref{prop:R-boundedness_sectors}]
 Let us fix some $\phi' \in (\phi, \pi)$ and first consider any $\lambda \in \Sigma_{\pi-\phi'}$ and estimate its distance to the boundary $\partial \Sigma_{\pi - \phi}$.
 In fact, we can distinguish between two cases, namely $\lambda \in \overline{\Sigma_{\frac{\pi}{2}-\phi}}$ (only possible if $\phi' < \frac{\pi}{2}$) and $\lambda \in \Sigma_{\pi-\phi'} \setminus \overline{\Sigma_{\frac{\pi}{2}-\phi}}$.
 In the first case, the distance from $\lambda$ to the boundary $\partial \Sigma_{\pi-\phi}$ can easily be seen to be $\abs{\lambda}$ as $0$ is the closest point to $\lambda$ on $\partial \Sigma_{\pi-\phi}$.
 In particular, in that case $B_{\abs{\lambda}}(\lambda) \subseteq \Sigma_{\pi-\phi}$.
 On the other hand, if $\lambda \not\in \overline{\Sigma_{\frac{\pi}{2}-\phi}}$ this will not be the case, but the closest point to $\lambda$ on $\partial \Sigma_{\pi-\phi}$ will be the orthogonal projection onto the line $\{r \ee^{\pm \ii (\pi - \phi)}: \, r > 0 \}$, where the sign is positive or negative depending on whether $\Im \lambda_0$ is positive or negative, resp.
 W.l.o.g.\ let us assume that $\Im \lambda > 0$ (the case $\Im \lambda < 0$ follows similarly).
 Then,
  \begin{align*}
   \operatorname{dist}(\lambda,\partial \Sigma_{\pi-\phi})
    = \Im (\ee^{\ii \phi} \lambda)
    \geq \sin(\phi'-\phi) \abs{\lambda},
  \end{align*}   
 since $\ee^{\ii \phi} \lambda \in \Sigma_{\phi'-\phi} \cap \{z \in \C: \, \Re z < 0, \, \Im z > 0 \}$, and, hence, $B_{\sin(\phi'-\phi) \abs{\lambda}}(\lambda) \subseteq \Sigma_{\pi-\phi}$.
 Using the Cauchy integral formula and the contraction principle for $\Rcal$-bounds of parameter-dependent integrals, we infer that, fixing an arbitrary $\varepsilon \in (0,1)$,
 \begin{align*}
  &\Rcal ( \{ \lambda T'(\lambda): \, \lambda \in \Sigma_{\pi-\phi'} \} )
   \\
   &= \Rcal \big( \big\{ \frac{\lambda}{2\pi \ii} \int_{\partial B_{\varepsilon \sin(\phi'-\phi) \abs{\lambda}}(\lambda)} \frac{T(\omega)}{(\lambda - \omega)^2} \, \dd \omega: \, \lambda \in \Sigma_{\pi-\phi'} \big\} \big)
   \\
  &\leq \Rcal \big( \big\{ \frac{\abs{\lambda}^2}{\abs{\lambda - \omega}^2}: \, \lambda \in \Sigma_{\pi - \phi'}, \, \omega \in \partial B_{\varepsilon \sin(\phi' - \phi) \abs{\lambda}}(\lambda) \} \cdot \Rcal \big( \big\{ T(\omega): \, \lambda \in \Sigma_{\pi - \phi'}, \, \omega \in \partial B_{\varepsilon \sin(\phi' - \phi) \abs{\lambda}}(\lambda) \big\} \big)
   \\
  &= \frac{1}{(\varepsilon \sin(\phi' - \phi))^2} \Rcal ( \{ T(\lambda): \, \lambda \in \Sigma_{\pi - \phi} \} )
  = \frac{C}{\varepsilon^2 \sin(\phi' - \phi)^2}.
 \end{align*}  
Letting $\varepsilon \rightarrow 1-$, we may, thus conclude that
 \[
  \Rcal ( \{ \lambda T'(\lambda): \, \lambda \in \Sigma_{\pi-\phi'} \} )
   \leq \frac{C}{\sin(\phi' - \phi)^2}.
 \]
The $\Rcal$-boundedness of the sets $\{ \lambda^k T^{(k)}(\lambda): \, \lambda \in \Sigma_{\pi - \phi'} \}$ follows by induction over $k \in \N$.
 \end{proof}

This property comes in handy for elliptic boundary value problems and parabolic initial-boundary value problems, since for these usually $\Rcal$-boundedness can be guaranteed on some complex sector rather than merely on $(0, \infty)$, cf.\ the characterisation of $\LL_p$-maximal regularity in Theorem \ref{thm:characterisation_Lp-max} below.
In fact, in the theory of parabolic initial-boundary value problems, sectorial operators appear naturally as generators of analytic semigroups, see e.g.\ \cite[Theorem II.4.6]{EnNa00}.
Combining the property of sectoriality with the notion of $\Rcal$-bounded operators then leads to the (in general much smaller) class of $\Rcal$-sectorial operators, which -- as we will immediately see in Theorem \ref{thm:characterisation_Lp-max} below -- appears inevitably when looking for $\LL_p$-maximal regularity of abstract Cauchy problems.

\begin{definition}[$\Rcal$-sectorial operators]
 A sectorial operator $A$ on a Banach space $X$ (write $A \in \Scal(X)$) is called $\Rcal$-sectorial (write $A \in \Rcal \Scal(X)$), if
  \[
   \Rcal_A(0)
    := \Rcal ( \{ \lambda (\lambda + A)^{-1}: \, \lambda > 0 \} )
    < \infty.
  \]
 In this case, the \emph{$\Rcal$-angle} $\phi_A^{\Rcal}$ of $A$ is defined as
  \[
   \phi_A^{\Rcal}
    := \inf \{ \theta \in (0, \pi): \Rcal_A(\pi - \theta) < \infty \},
  \]
 where
  \[
   \Rcal_A(\theta)
    := \Rcal ( \{ \lambda(\lambda + A)^{-1}: \, \abs{\arg \lambda} \leq \theta \} ).
  \]
\end{definition}

\begin{remark}
 By expanding the resolvent operator $(\lambda - A)^{-1}$ as a Taylor series around some $s > 0$, using the Neumann series and the algebraic properties (subadditive, submultiplicative) of the $\Rcal$-bounds, one finds that $\phi_A^{\Rcal} > 0$ for every $\Rcal$-sectorial operator, i.e.\ the definition of the $\Rcal$-angle is meaningful.
\end{remark}

\begin{proof}
 Let $\eta \in (0, \frac{1}{\Rcal_A(0)})$.
 Then, for every $s > 0$ and $\lambda \in \C$ such that $\Re \lambda = s$ and $\abs{\Im \lambda} < \eta s$ (i.e., $\lambda \in \Sigma_\vartheta$ for $\vartheta = \arctan(\eta) > 0$) it holds that $\lambda + A = (s + A) + (\lambda - s)$, hence,
  \[
   (s + A)^{-1} (\lambda + A)
    = (\lambda + A) (s + A)^{-1}
    = \id + (\lambda - s) (s + A)^{-1},
  \]
 where
  \[
   \Rcal ( \{ (\lambda - s) (s + A)^{-1}: \, s > 0, \, \Re \lambda = s, \, \abs{\Im \lambda} < \eta s \} )
    \leq \eta \Rcal ( \{ s (s + A)^{-1}: s > 0 \} )
    = \eta \Rcal_A(0)
    < 1,
  \]
 so that
  \[
   \Rcal ( \{ (I + (\lambda - s) (s + A)^{-1})^{-1} = \sum_{k=0}^\infty (\lambda-s)^k (s + A)^{-k}: \, s > 0, \, \Re \lambda = s, \, \abs{\Im \lambda} < \eta s \} )
    < \frac{1}{1 - \eta \Rcal_A(0)},
  \]
 and, thus, we may deduce that
  \begin{align*}
   &\Rcal ( \{ \lambda (\lambda + A)^{-1}: \Re \lambda = s, \, \abs{\Im \lambda} < \eta s, \, s > 0 \} )
    \\
    &\leq \Rcal ( \{ \frac{\lambda}{\Re \lambda}: \, \abs{\Im \lambda} < \eta \Re \lambda, \, \Re \lambda > 0 \} ) \cdot \Rcal ( \{ s(s+A)^{-1}: \, s > 0 \} )
     \\ &\quad
     \cdot \Rcal ( \{ (I + (\lambda - s)(s+A)^{-1})^{-1}: \, \Re \lambda = s > 0, \, \abs{\Im \lambda} < \eta s \} )
    \\
    &< \sqrt{1 + \eta^2} \cdot \Rcal_A(0) \cdot \frac{1}{1 - \eta \Rcal_A(0)}
    < \infty.
  \end{align*}
\end{proof}

Similar to the perturbation theorems for contractive or analytic semigroup generators, $\Rcal$-sectorial operators allow for a relative bounded perturbation theorem, see \cite[Propositions 4.2 and 4.3]{DeHiPr03}.
 
\begin{proposition}[Relatively bounded perturbation of $\Rcal$-sectorial operators] \label{prop:Perturbation_Theorem}
 Suppose that $A$ is an $\Rcal$-sectorial operator on a Banach space $X$ and set
  \[
   a := \Rcal ( \{ A (\lambda + A)^{-1}: \, \lambda \in \Sigma_\theta \} )
    < \infty
  \]
 for some fixed $\theta \in (0, \pi - \phi_A^{\Rcal})$.
 \begin{enumerate}
  \item
   Assume that $B$ is a linear operator with $\dom(B) \supseteq \dom(A)$ and such that, for some $\alpha \in (0, \nicefrac{1}{a})$,
    \[
     \norm{B u}
      \leq \alpha \norm{A u}
      \quad
      \text{for all } u \in \dom(A).
    \]
   Then the operator $A + B$ is $\Rcal$-sectorial with $\Rcal$-angle $\phi_{A+B}^\Rcal \geq \pi - \theta$ and
    \[
     \Rcal ( \{ \lambda (\lambda + A + B)^{-1}: \, \lambda \in \Sigma_\theta \} )
      \leq \frac{1}{1-\gamma} \Rcal ( \{ \lambda (\lambda + A)^{-1}: \, \lambda \in \Sigma_\theta \} )
    \]
   where $\gamma := \norm{B A^{-1}}_{\B(X)} \Rcal ( \{ A (\lambda + A)^{-1}: \, \lambda \in \Sigma_\theta \} )$.
  \item
   Let
    \[
     C_A
      = \sup_{\lambda > 0} \norm{A (\lambda + A)^{-1}}_{\B(X)},
      \quad
     M_A
      = \sup_{\lambda > 0} \norm{\lambda (\lambda + A)^{-1}}_{\B(X)}
      < \infty
    \]
   and assume that $B$ is a linear operator with $\dom(B) \supseteq \dom(A)$ and such that, for some $\alpha, \beta > 0$,
    \[
     \norm{B u}
      \leq \alpha \norm{A u} + \beta \norm{u}
      \quad
      \text{for all } u \in \dom(A).
    \]
   Whenever $\alpha < \frac{1}{1+a} C_A$ and $\mu > \beta M_A \frac{1+a}{1 - \alpha C_A(1+a)}$,
   then $A + B$ is $\Rcal$-sectorial with $\Rcal$-angle $\phi_{A+B}^\Rcal \geq \pi - \theta$ and
    \[
     \Rcal ( \{ \lambda (\lambda + \mu + A + B)^{-1}: \, \lambda \in \Sigma_\theta \} )
      < \infty.
    \]
 \end{enumerate}
\end{proposition}

As announced above, $\Rcal$-sectorial operators and $\LL_p$-maximal regularity of abstract Cauchy problems are closely related.
In fact, the following result is known for Banach spaces of class $\HTcal$, see \cite[Theorem 4.4]{DeHiPr03}.

\begin{theorem}[Characterisation of $\LL^p$-maximal regularity]
 \label{thm:characterisation_Lp-max}
 Let $X$ be a Banach space of class $\HTcal$, $p \in (1, \infty)$ and let $A$ be a sectorial operator on $X$ with spectral angle $\phi_A^\sigma = \inf \{\phi > 0: \, \sigma(A) \subseteq \Sigma_{\pi - \phi}\} < \frac{\pi}{2}$.
 In this case, the \emph{abstract Cauchy problem} with homogeneous initial data
  \begin{equation}
   \begin{cases}
   \frac{\dd}{\dd t} u(t) + A u(t)
    &= f(t)
    \quad
    \text{for } t \geq 0,
    \\
   u(0)
    &= 0
   \end{cases}
    \tag{$\mathrm{ACP}_0$}
    \label{eqn:ACP0}
  \end{equation}
 admits $\LL^p$-maximal regularity on the ray $\R_+$, i.e.\ \eqref{eqn:ACP0} has a unique solution in class $u: \R_+ \rightarrow X$ with $A u \in \LL_p(\R_+;X)$ if and only if $f \in \LL_p(\R_+;X)$, if and only if $A$ is $\Rcal$-sectorial with angle $\phi_A^{\Rcal} < \frac{\pi}{2}$.
 More precisely, the following statements are equivalent:
  \begin{enumerate}
   \item
    The abstract Cauchy problem \eqref{eqn:ACP0} has $\LL^p$-maximal regularity on $\R_+$.
   \item
    The set $\{ A (\ii \rho + A)^{-1}: \, \rho \in \R \setminus \{0\} \}$ is $\Rcal$-bounded.
   \item 
    The set $\{ A (\lambda + A)^{-1}: \, \lambda \in \Sigma_\theta \}$ is $\Rcal$-bounded for some $\theta > \frac{\pi}{2}$.
   \item
    The set $\{ \ee^{- A \lambda}: \, \lambda \in \Sigma_\vartheta \}$ is $\Rcal$-bounded for some $\vartheta > 0$.
   \item
    The sets $\{ \ee^{- At}: \, t > 0 \}$ and $\{ tA \ee^{-At}: \, t > 0 \}$ are $\Rcal$-bounded.
  \end{enumerate}
\end{theorem}

\begin{remark}
 It may be interesting to compare the conditions on $\LL_p$-maximal regularity with those of the (weaker) conditions on an operator $A$ to generate a bounded analytic semigroup:
 \begin{enumerate}
  \item
   First note that, for the case that $A$ is invertible, to have maximal regularity on the (unbounded) ray $\R_+$ one has to expect exponential stability of the $C_0$-semigroup $(\ee^{-tA})_{t \geq 0}$ generated by $- A$, since then from $A u \in \LL_p(\R_+;X)$ it follows $u \in \LL_p(\R_+;X)$, and for $f = 0$ and by the Datko lemma this implies exponential stability of the semigroup.
  \item
   In the context of bounded analytic semigroups, the following assertions are equivalent, see e.g.\ \cite[Theorem II.4.6]{EnNa00}:
    \begin{itemize}
     \item
      $-A$ generates a bounded (on some sector $\Sigma_\theta$, $\theta > 0$) analytic $C_0$-semigroup $(\ee^{-\lambda A})_{\lambda \in \Sigma_\theta}$.
     \item
      $-A$ generates a bounded (on $(0, \infty)$) $C_0$-semigroup $(\ee^{-tA})_{t \geq 0}$ such that $\ran(\ee^{-tA}) \subseteq \dom(A)$ for every $t > 0$ and
       \[
        M
         := \sup_{t > 0} \norm{t A \ee^{-tA}}
         < \infty.
       \]
     \item
      There is $\vartheta \in (0, \tfrac{\pi}{2})$ such that $\pm \ee^{\ii \vartheta} (-A)$ generate bounded $C_0$-semigroups.
     \item
      $-A$ generates a bounded $C_0$-semigroup and
       \[
        \sup_{\lambda \in \C_0^+} \norm{(\Im \lambda) (A + \lambda)^{-1}}
         < \infty,
       \]
      where $\C_\omega^+ := \{ z \in \C: \, \Re z > \omega \}$ for $\omega \in \R$.
    \end{itemize}
 \end{enumerate}
\end{remark}

For families of integral operators, $\Rcal$-bounds on their kernel functions may help to derive $\Rcal$-boundedness of the familiy, see \cite[Proposition 4.12]{DeHiPr03}.

\begin{proposition}
 Let $X$ and $Y$ be Banach spaces, $\Omega \subseteq \R^n$ be open and $p \in (1, \infty)$.
 Further, let $\Kcal \subseteq \B(\LL_p(\Omega;X); \LL_p(\Omega;Y))$ be a family of kernel operators, i.e.\ for each $K \in \Kcal$ there is a measurable kernel $k: \Omega \times \Omega \rightarrow \B(X;Y)$ such that
  \[
   (K f)(\vec x)
    = \int_\Omega k(\vec x,\vec x') f(\vec x') \, \dd \vec x'
    \quad
    \text{for all } f \in \LL_p(\Omega;X) \text{ and a.e.\ } \vec x \in \Omega.
  \]
 Further, assume that the family of kernels can be $\Rcal$-bounded pointwise
  \[
   \Rcal ( \{ k(\vec x,\vec x'): \, k \in \Kcal \} )
    \leq \kappa_0(\vec x,\vec x')
    \quad
    \text{for a.e.\ } \vec x,\vec x' \in \Omega
  \]
 by some scalar, measurable function $\kappa_0: \Omega \times \Omega \rightarrow \R_+$ which itself is the kernel of a scalar integral operator $K_0 \in \B(\LL_p(\Omega))$.
 Then the family $\Kcal \subseteq \B(\LL_p(\Omega;X); \LL_p(\Omega;Y))$ is $\Rcal$-bounded with $\Rcal$-bound $\Rcal(\Kcal) \leq \norm{K_0}_{\B(\LL_p(\Omega))}$.
\end{proposition}

\begin{remark}
 For comparison, note that a similar result is true if all $\Rcal$-bounds are replaced by (uniform) bounds.
\end{remark}

\begin{definition}[The classes $\Hcal^\infty(X)$, $\Rcal \Hcal^\infty(X)$ and $\BIPcal(X)$]
 A sectorial operator $A \in \Scal(X)$ is said to have \emph{bounded imaginary powers} (write $A \in \BIPcal(X)$), if $A^{\ii s}$ (defined via the Dunford calculus for sectorial operators) is in $\B(X)$ for every $s \in \R$ and there is a constant $C > 0$ such that $\norm{A^{\ii s}}_{\B(X)} \leq C$ for all $s \in [-1,1]$, cf.\ \cite[Definition 2.4]{DeHiPr03}.
 \newline
 A sectorial operator $A \in \Scal(X)$ is said to have \emph{bounded holomorphic functional calculus} (write $A \in \Hcal^\infty(X)$), if there are $\phi > 0$ and $C_\phi >0$ such that for all $h \in \Hcal_0(\Sigma_\phi) := \{ h: \Sigma_\phi \rightarrow \C: \, h \text{ holomorphic and bounded} \}$ which tend polynomially to zero as $\abs{\vec z} \rightarrow \infty$, i.e.\ there are $C, s > 0$ such that
  \[
   \abs{h(z)}
    \leq C (1 + \abs{z})^{-s},
    \quad
    z \in \Sigma_\phi  
  \]
the operator $h(A) \in \B(X)$, which in this case can be defined via the Dunford calculus, satisfies the norm estimate
  \[
   \norm{h(A)}_{\B(X)}
    \leq C_\phi \norm{h}_\infty.
  \]
 We set $\phi_A^\infty$ to be the infimum of those $\phi$ for which such a constant $C_\phi > 0$ exists, cf.\ \cite[Definition 2.9]{DeHiPr03}.
 \newline
 An operator $A \in \Hcal^\infty(X)$ is said to have an $\Rcal$-bounded functional calculus (write $A \in \Rcal \Hcal^\infty(X)$), if 
  \[
   \Rcal ( \{ h(A): \, h \in \Hcal_0(\Sigma_\phi), \, \norm{h}_\infty \leq 1 \} )
    < \infty
  \]
 for some $\phi > 0$. The infimum of those $\phi$ which are admissible for $\Rcal$-boundedness of the set on the left-hand side is denoted by $\phi_A^{\Rcal,\infty}$, cf. \cite[Definition 4.9]{DeHiPr03}.
\end{definition}

\begin{remark}[Hilbert space and Banach space case]
\label{rem:Inclusions_Operator-Spaces}
 If $X$ is a Hilbert space, then
  \[
   \Rcal \Hcal^\infty(X)
    = \Hcal^\infty(X)
    = \BIPcal(X)
    \subseteq \Rcal \Scal(X)
    = \Scal(X),
  \]
 where the inclusion of the (set of) operators of bounded imaginary powers in the (set of) $\Rcal$-sectorial operators is strict, in general.
 For general Banach spaces the equalities may become strict inclusions, too.
 In any case, for every $A \in \Rcal \Hcal^\infty(X)$
  \[
   \phi_A^{\Rcal,\infty}
    \geq \phi_A^\infty
    \geq \theta_A
    \geq \phi_A^{\Rcal}
    \geq \phi_A^\sigma,
  \]
 but in general these inequalities may be strict.
\end{remark}

A nice property of operators with $\Rcal$-bounded $\Hcal^\infty$-calculus is the following, see \cite[Proposition 4.10]{DeHiPr03}.

\begin{proposition}
\label{prop:holomorphic_image_of_bounded_sets}
 Let $X$ be a Banach space, $A \in \Rcal \Hcal^\infty(X)$ and let $\{h_\lambda\}_{\lambda \in \Lambda}$ be a family of bounded, holomorphic, scalar functions which is uniformly bounded in $\Hcal^\infty(\Sigma_\theta)$ for some $\theta > \phi_A^{\Rcal,\infty}$ and $\Lambda$ be an arbitrary index set.
 Then the family of operators $\{ h_\lambda(A) \}_{\lambda \in \Lambda}$ given by the $\Hcal^\infty$-calculus for $A$ is $\Rcal$-bounded.
\end{proposition}

 \section{Elliptic differential operators on $\R^n$}
 \label{Sec:Elliptic_differential_operators_full_space}
 
Our primary goal is to consider abstract initial-boundary value problems (IBVP) of parabolic type. However, it is well-known that the Laplace transform of parabolic equation with respect to the time variable leads to elliptic boundary value problems (BVP), which in many cases are easier to handle than directly starting with the parabolic equations themselves.
Therefore, we begin our investigation with elliptic problems arising from parabolic inital-boundary value problems by formally taking the Laplace transform of these.

 \subsection{The homogeneous, constant coefficient case:}

Let $k \in \N$, $E$ be an arbitrary Banach space and $\Acal(\vec \xi)$ be a $\B(E)$-valued polynomial in $\R^n$ which is homogeneous of order $k$, i.e.\
 \[
  \Acal(\vec \xi)
   = \sum_{\abs{\vec \alpha} = k} a_{\vec \alpha} \vec \xi^{\vec \alpha}
 \]
for some operator-valued constant coefficients $a_{\vec \alpha} \in \B(E)$.
To this we associate the differential operator $\Acal(\DD)$ by formally plugging in $\DD = - \ii \nabla = (\DD_i)_{i=1,\ldots,n} = (- \ii \partial_{x_i})_{i=1,\ldots,n}$ as $\vec \xi$:
 \[
  \Acal(\DD)
   = \sum_{\abs{\vec \alpha} = k} a_{\vec \alpha} \DD^{\vec \alpha}
   = \sum_{\abs{\vec \alpha} = k} a_{\vec \alpha} (- \ii)^k \prod_{i=1}^n \frac{\partial^{\alpha_i}}{\partial x_i^{\alpha_i}}.
 \]
If, for each given function $f$ which is sufficiently regular and integrable and suitable $\lambda \in \C$, the vector-valued partial differential equation
 \begin{equation}
  \lambda u(\vec x) + \mathcal{A}(\DD) u(\vec x)
   = f(\vec x)
   \quad
   \text{for } \vec x \in \R^n
   \label{eqn:Elliptic_PDE}
 \end{equation}
has a unique solution which can be represented by an integral formula
 \[
  u(\vec x)
   = \int_{\R^n} \gamma_{\lambda}(\vec x - \vec z) f(\vec z) \, \dd \vec z
 \]
for some kernel function $\gamma_\lambda: \R^n \rightarrow \C$, then, by homogeneity of degree $k$ for the polynomial $\Acal$, the kernel function $\gamma_\lambda$ must scale as
 \[
  \gamma_\lambda(\vec x)
   = \abs{\lambda}^{\tfrac{n}{k} - 1} \tilde \gamma_{\theta}(\abs{\lambda}^{1/k} \vec x)
   \quad
   \text{for } \vec x \in \R^n, \, \theta = \arg(\lambda), \, \lambda \neq 0
 \]
and $\tilde \gamma_\theta = \gamma_{\ee^{\ii \theta}}$ is the fundamental solution of the vector-valued partial differential equation \eqref{eqn:Elliptic_PDE}, i.e.\
 \[
  \ee^{\ii \theta} \tilde \gamma_\theta + \Acal(\DD) \tilde \gamma_\theta
   = \delta_0
 \]
in the sense of distributions.
To see this, assume that $\lambda = \rho \ee^{\ii \theta}$ for some radius $\rho > 0$ and some angle $\theta \in (- \pi, \pi)$ and let $u$ be the solution to the elliptic PDE
 \[
  \ee^{\ii \theta} u + \Acal(\DD) u
   = f
 \]
for some given function $f: \R^n \rightarrow E$ and which has the integral representation
 \[
  u(\vec x)
   = \int_{\R^n} \tilde \gamma_\theta(\vec x - \vec x') f(\vec x') \, \dd \vec x',
   \quad
   \vec x \in \R^n.
 \]
Then, since $\Acal(\vec \xi)$ is homogeneous of degree $k$ in $\vec \xi$, the function $u_\lambda(\vec x) := u(\rho^{1/k} \vec x)$ solves the elliptic PDE
 \begin{align*}
  \lambda u_\lambda(\vec x) + \Acal(\DD) u_\lambda(\vec x)
   &= \lambda u(\rho^{1/k} \vec x) + \rho (\Acal(\DD) u)(\rho^{1/k} \vec x)
   \\
   &= \rho \big[ \ee^{\ii \theta} u(\rho^{1/k} \vec x) + (\Acal(\DD) u)(\rho^{1/k} \vec x) \big]
   = \rho f(\rho^{1/k} \vec x)
 \end{align*}
and has the integral representation
 \begin{align*}
  u_\lambda(\vec x)
   &= u(\rho^{1/k} \vec x)
   = \int_{\R^n} \tilde \gamma_\theta(\rho^{1/k} \vec x - \vec z) f(\vec z) \, \dd \vec z
   \\
   &= \int_{\R^n} \rho^{\frac{n}{k}} \tilde \gamma_\theta(\rho^{1/k} (\vec x - \vec y)) f(\rho^{1/k} \vec y) \, \dd \vec y
   \\
   &= \int_{\R^n} \rho^{\frac{n}{k}-1} \tilde \gamma_\theta(\rho^{\frac{1}{k}}(\vec x - \vec y)) \rho f(\rho^{\frac{1}{k}} \vec y) \, \dd \vec y.
 \end{align*}
Thus, we may conclude that $\gamma_\lambda(\vec x) = \rho^{\frac{n}{k}-1} \tilde \gamma_\theta(\rho^{1/k} \vec x)$ has the scaling property which we claimed.

\begin{definition}[Parameter-ellipticity]
 A homogeneous, $\B(E)$-valued polynomial $\Acal(\vec \xi)$ is called parameter-elliptic, if there is an angle $\phi \in [0, \pi)$ such that the spectrum $\sigma(\Acal(\vec \xi))$ of $\Acal(\vec \xi) \in \B(E)$ is contained in the sector $\Sigma_\phi$ for every $\vec \xi \in \S^{n-1} \subseteq \R^n$.
 The infimum $\phi_\Acal^\mathrm{ellipt}$ of these $\phi > 0$ is called \emph{angle of ellipticity}.
\end{definition}

\begin{remark}
 Let us make two short remarks on why later on we will choose $k = 2m$ to be even, and how parameter-ellipticity and the Lopatinskii--Shapiro condition introduced below in Assumption \ref{assmpt:LSC} will carry over from real vectors $\vec \xi \in \R^n$ to complex vectors $\vec \xi + \ii \vec \eta \in \C^n$ with sufficiently small imaginary parts of the components.
 \begin{enumerate}
  \item
   To have an angle of ellipticity $\phi_\Acal^\mathrm{ellipt} < \frac{\pi}{2}$, the degree $k$ of the polynomial $\Acal(\vec \xi)$ must necessarily be even, see, e.g., the remarks after \cite[Definition 5.1]{DeHiPr03}.
   Therefore, soon we will restrict ourselves to the most relevant case that $k = 2m$ for some $m \in \N$.
  \item
   Although parameter-ellipticity is first formulated for real unit vectors $\vec \xi \in \S^{n-1}$, by homogeneity of $\Acal$ of order $k$, the spectrum $\sigma(\Acal(\vec \xi)) \subseteq \Sigma_\phi$ is contained in some sector $\Sigma_\phi$, uniformly for all $\vec \xi \in \R^n \setminus \{\vec 0\}$.
   Moreover, and this will become relevant when considering the Lopatinskii--Shapiro condition for elliptic boundary value problems, for any $\vec \xi \in \R^n$ and $\vec \eta \in \R^n$ sufficiently small compared to $\vec \xi$, say, $\abs{\vec \eta} \leq \varepsilon \abs{\vec \xi}$ for some $\varepsilon > 0$, we may deduce that
    \[
     \sigma(\Acal(\vec \xi + \ii \vec \eta))
      = \sigma \big( \sum_{\abs{\vec \alpha} = k} a_{\vec \alpha} \vec \xi^{\vec \alpha} + \sum_{\abs{\vec \alpha} = k} a_{\vec \alpha} \sum_{\vec 0 \lneq \vec \beta \leq \vec \alpha} \vec \xi^{\vec \alpha - \vec \beta} (\ii \vec \eta)^{\vec \beta} \big)
      \subseteq \Sigma_{\phi'}
    \]
   for some, in general smaller, angle $\phi' \in (0, \phi)$ (which depends on the chosen $\varepsilon > 0$), and where we write $\vec \beta \leq \vec \alpha$ if component-wise $\beta_i \leq \alpha_i$ for $i = 1, \ldots, n$.
 \end{enumerate}
\end{remark}

 For some kernel estimates which we will derive right below, we employ the functions
  \[
   p_{k,\nu}^n: \R_+ \rightarrow \R,
    \quad
   p_{k,\nu}^n(r)
    := \int_0^\infty \frac{s^{n-2}}{(1+s)^{k-1-\nu}} \ee^{-(s+1)r} \, \dd s,
    \quad
    k, n, \nu \in \N, \, r > 0.
  \]
 These are \emph{completely monotone}, i.e.\ $(-1)^p \frac{\dd^p}{\dd r^p} p_{k,\nu}^l \geq 0$ for all $p \in \N_0$, hence, by the Bernstein--Widder theorem, see e.g.\ \cite[Theorem 17]{Wid34}, is the Laplace transform of a non-negative measure, i.e.\ $p_{k,\nu}^l(x) = \int_0^\infty \ee^{-sx} \, \dd \mu_{j,\nu}^l(s)$, and
  \[
   \int_0^\infty r^{n+\rho-1} p_{k,n}^k(r) \, \dd r
    < \infty
    \quad \text{if and only if}
    \quad
    \rho > \max \{-n, \nu-k\}.
  \]
 Moreover, the following integral relation between the functions $p_{k,\nu}^n$ is valid.

  \begin{lemma}
  \label{lem:Lem-2.6}
   For all $c,y > 0$ and $k, n, \nu \in \N_0$, the following identity is valid:
    \[
     \int_0^\infty p_{k,\nu}^n (c (y + r)) r^{n-1} \, \dd r
      = \frac{(n-1)!}{c^n} p_{k+n,\nu}^n(cy).
    \]
  \end{lemma}
  This identity can be validated straightforwardly via integration by parts, see \cite[Corollary 5.3]{DeHiPr03}.

\begin{theorem}[Kernel estimate for parameter-elliptic operators]
 Let $k, n \in \N$ be natural numbers, $E$ be a Banach space and operator-valued coefficients $a_{\vec \alpha} \in \B(E)$ be given such that
  \[
   \Acal(\vec \xi)
    = \sum_{\abs{\vec \alpha} = k} a_{\vec \alpha} \vec \xi^{\vec \alpha}
    \quad
    \text{for } \vec \xi \in \R^n
  \]
 is parameter-elliptic with angle of ellipticity $\phi_\Acal^\mathrm{ellipt} < \pi$.
 Then, for each $\phi > \phi_\Acal^\mathrm{ellipt}$ and $\nu \in \N_0$ there exist constants $c_\phi, C_{\phi,\nu} > 0$ such that the fundamental solution $\gamma_\lambda$ of
  \[
   \lambda u + \Acal(\DD) u
    = \delta_0
  \]
 satisfies the kernel estimates
  \begin{align*}
   \abs{\DD^{\vec \beta} \gamma_\lambda(\vec x)}
    \leq C_\phi \abs{\lambda}^{(\nu+1)\tfrac{n}{k} - 1} p_{k,\nu}^n(c_\phi \abs{\lambda}^{1/k} \abs{\vec x})
    \quad
    \text{for all } \vec x \in \R^n \setminus \{0\}, \, \lambda \in \Sigma_{\pi-\phi} \text{ and } \abs{\vec \beta} = \nu.
  \end{align*}
\end{theorem}

\begin{proof}
 By \cite[Corollary 5.3]{DeHiPr03},
  \[
   \abs{\DD^{\vec \beta} \tilde \gamma_\theta(\vec x)}
    \leq C_{\phi,\nu} p_{k,\nu}^n(c_\phi \abs{\vec x}),
  \]
 so from $\gamma_\lambda(\vec x) = \abs{\lambda}^{\frac{n}{k}-1} \tilde \gamma_\theta(\abs{\lambda}^{1/k} \vec x)$ we find that
  \begin{align*}
   \abs{\DD^{\vec \beta} \gamma_\lambda(\vec x)}
    &= \abs{\lambda}^{(\nu+1)\frac{n}{k}-1} \abs{(\DD^{\vec \beta} \tilde \gamma_\theta)(\abs{\lambda}^{1/k} \vec x)}
    \\
    &\leq \abs{\lambda}^{(\nu+1)\frac{n}{k}-1}  C_{\phi,\nu} p_{k,\nu}^n(c_\phi \abs{\lambda}^{1/k} \abs{\vec x}).
  \end{align*}
\end{proof}

The most important results on the full space problem with constant coefficients can be summarised as follows, cf.\ \cite[Theorems 5.4 and 5.5 and Corollary 5.6]{DeHiPr03}.

\begin{theorem}[Full-space problem]
\label{thm:full-space_problem}
  Let $n, m \in \N$  and $p \in (1, \infty)$.
  Assume that the interior symbol
   \[
    \Acal(\vec  \xi)
     = \sum_{\abs{\vec \alpha} = 2m} a_{\vec \alpha} \vec \xi^{\vec \alpha},
     \quad
     \vec \xi \in \R^{n+1}
   \]
  is parameter-elliptic with angle of ellipticity $\phi_\Acal^\mathrm{ellipt} < \pi$.
  \begin{enumerate}
   \item
    The $\LL_p(\R^{n+1})$-realisation $A_{\R^{n+1}}$ of $\mathcal{A}$, defined as the closure $A_{\R^ {n+1}} = \overline{A_0}$ of the minimal realisation of the interior symbol on $\LL_p(\R^{n+1})$
   \begin{align*}
    A_0: \quad
     \dom(A_0)
      &= \WW_p^{2m}(\R^n;E)
      \subseteq \LL_p(\R^n;E)
      \rightarrow \LL_p(\R^n;E),
      \\
     [A_0 u](\vec x)
      &= \Acal(\DD)u(\vec x),
      \quad
      \text{a.e.\ } \vec x \in \R^{n+1}
   \end{align*}
  is sectorial with \emph{spectral angle} $\phi_{A_{\R^{n+1}}}^\sigma \leq \phi_\Acal^\mathrm{ellipt}$.
   For its domain $\dom(A_{\R^{n+1}})$ the inclusions
   \[
    \WW_p^{2m}(\R^n;E)
     \subseteq \dom(A_{\R^{n+1}})
     \subseteq \WW_p^{2m-1}(\R^n;E)
   \]
  are valid.
  \item
  If $E$ is a Banach space of class $\HTcal$, then $A_{\R^{n+1}} = A_0$, i.e.\ $\dom(A_{\R^{n+1}}) = \WW_p^{2m}(\R^{n+1})$, and $A_{\R^{n+1}} \in \Hcal^\infty(\LL_p(\R^{n+1};E))$ with $\Hcal^\infty$-angle $\phi_{A_{\R^{n+1}}}^\infty \leq \phi_\Acal^ {\mathrm{ellipt}}$.
 In particular, $A_{\R^{n+1}} \in \Rcal \Scal(\LL_p(\R^n;E))$ with angle $\phi_{A_{\R^{n+1}}}^{\Rcal} \leq \phi_\Acal^\mathrm{ellipt}$ and $- A_{\R^{n+1}}$ generates a bounded analytic semigroup $(\ee^{- t A_{\R^{n+1}}})_{t \geq 0}$ on $\LL_p(\R^{n+1};E)$.
 Moreover, for every angle $\phi > \phi_\Acal^\mathrm{ellipt}$,
  \[
   \Rcal ( \{ \lambda^{1 - \tfrac{\nu}{k}} \DD^{\vec \beta} (\lambda + A_{\R^{n+1}})^{-1}: \, \lambda \in \Sigma_{\pi-\phi}, \, 0 \leq \abs{\vec \beta} = \nu \leq 2m \} )
    < \infty.
  \]
 \end{enumerate}
\end{theorem}

 For the proof of this theorem, we refer to \cite[Theorem 5.5 and Corollary 5.6]{DeHiPr03}.
 It crucially uses that the $n$-dimensional sphere $\S^n \subseteq \R^{n+1}$ is compact and $\Acal(\vec \xi)$ is holomorphic and homogeneous of degree $k = 2m$, so that by Proposition \ref{prop:holomorphic_image_of_bounded_sets} the family $\{ \vec \xi^{\vec \alpha} \Acal(\vec \xi)^{-1}: \, \vec \xi \in \R^{n+1} \setminus \{0\} \}$ is $\Rcal$-bounded for each fixed $\vec \alpha \in \N_0^{n+1}$ with $\abs{\vec \alpha} = k$, and similarly $\{ \vec \xi^{\vec \beta} \DD^{\vec \beta} [\vec \xi^{\vec \alpha} \Acal(\vec \xi)^{-1}]: \vec \xi \in \R^{n+1} \setminus \{0\} \}$ is $\Rcal$-bounded as well.
 The bounded $\Hcal^\infty$-calculus is then established via Cauchy's theorem and again Proposition \ref{prop:holomorphic_image_of_bounded_sets}.

\subsection{Spatially dependent coefficients and lower order perturbations.}

Next, we consider the variable coefficient case with lower order coefficients, i.e.\
 \[
  \Acal(\vec x,\vec \xi)
   = \sum_{\abs{\vec \alpha} \leq k} a_{\vec \alpha}(\vec x) \DD^{\vec \alpha}
 \]
for some spatial-dependent operator-valued coefficients $a_{\vec \alpha}(\vec x) \in \B(E)$, $k \in \N$, and its minimal $\LL_p(\R^n;X)$-realisation
 \[
  [A_0 u](\vec x)
   = \Acal(\vec x,\DD)u(\vec x)
   \quad
   \text{for } u \in \dom(A_0) = \WW_p^k(\R^n;E) \text{ and a.e.\ } \vec x \in \R^n.
 \]

Using perturbation arguments, the following result can be derived, provided the Banach space $E$ is of class $\HTcal$, see \cite[Theorem 5.7]{DeHiPr03}.
 
\begin{theorem}
 Let $E$ be a Banach space of class $\HTcal$, $n, k \in \N$, $p \in (1, \infty)$ and $\phi_0 \in (0, \pi)$.
 Assume that $\mathcal{A}(\vec x,\DD) = \sum_{\abs{\vec \alpha} \leq k} a_{\vec \alpha}(\vec x,\DD)$ is a differential operator of order $k$ with variable coefficients which are subject to the following assumptions:
  \begin{enumerate}
   \item
    $a_{\vec \alpha} \in \CC_\mathrm{l}(\R^n; \B(E))$ for the top order coefficients where $\vec \alpha \in \N_0^n$ has length $\abs{\vec \alpha} = k$.
   \item
    The principal part $\Acal_\#(\vec x,\vec \xi) = \sum_{\abs{\vec \alpha} = k} a_{\vec \alpha}(\vec x) \vec \xi^{\vec \alpha}$ is parameter elliptic with angle of ellipticity less or equal $\phi_0$, for every $\vec x \in \R^n \cup \{\infty\}$, where we formally write $a_{\vec \alpha}(\infty) := \lim_{\abs{\vec x} \rightarrow \infty} a_{\vec \alpha}(\vec x)$.
   \item
    $a_{\vec \alpha} \in [\LL_\infty + \LL_{r_\nu}](\R^n; \B(E))$ for each lower order multi-index $\vec \alpha \in \N_0^n$ of length $\abs{\vec \alpha} = \nu < k$ for some $r_\nu \geq p$ such that $k - \nu > \nicefrac{n}{r_\nu}$.
  \end{enumerate}
 Then, for each angle $\phi > \phi_0$ there is a constant $\mu_\phi \geq 0$ such that the shifted operator $\mu_\phi + A_{\R^n}$ is $\Rcal$-sectorial with $\Rcal$-angle $\phi^\Rcal_{\mu_\phi + A_{\R^n}} \leq \phi$.
 In particular, if $\phi_0 < \nicefrac{\pi}{2}$, then the abstract Cauchy problem
  \[
   \frac{\, \dd}{\, \dd t} u(t) + A_{\R^n} u(t)
    = f(t),
    \quad
    t \in (0,T),
    \quad
   u(0)
    = 0
  \]
 has maximal regularity in the class $\LL_q((0,T);\LL_p(\R^n;E))$ for each $q \in (1, \infty)$, i.e.\ has a unqiue solution in the class $u \in \WW_{q,p}^{(1,2)}((0,T) \times \R^n; E)$ if and only if $f \in \LL_{q,p}((0,T) \times \R^n;E)$.
\end{theorem}
 
 \section{Elliptic vector-valued PDE: The half-space problem}
 \label{Sec:Half-space_problem}

In the preceding section, we have mainly recalled results from \cite{DeHiPr03} on the full-space problem, and no adjustments have been necessary.
Next, we consider the half-space problem as a prototypical domain and an intermediate step towards general domains.
In this context, we need to adjust the notation and revise the proofs in \cite{DeHiPr03} to cover the more general class of systems with mixed type boundary conditions.

 \subsection{Partial Fourier transforms}
 Starting with the prototypical case of a half-space $\R_+^{n+1} := \R^n \times (0,\infty)$ instead of a general domain $\Omega \subseteq \R^n$, we are interested in well-posedness and maximal regularity properties of linear systems of parabolic equations of the form
  \begin{equation}
   \begin{cases}
   \partial_t u(t,\vec x) + \Acal(\DD) u(t,\vec x)
    = f(t,\vec x)
    &\text{for } (t, \vec x) \in (0, \infty) \times \R_+^{n+1},
    \nonumber
    \\
   \mathcal{B}_j(\DD) u(t,\vec x)
    = g_j(t,\vec x)
    &\text{for } (t, \vec x) \in (0, \infty) \times \partial \R_+^{n+1},
   \end{cases}
    \tag{PIBVP}
    \label{eqn:PIBVP}
  \end{equation}
 accompanied with some suitable initial conditions at $t = 0$.
 Revisiting the strategy in \cite{DeHiPr03}, for the moment, let us assume that the differential operators $\Acal(\DD)$ in the interior and $\Bcal_j(\DD)$ on the boundary are homogeneous of degrees $2m$ and $m_j < 2m$, respectively, i.e.\
  \[
   \Acal(\DD) = \sum_{\abs{\vec \alpha} = 2m} a_{\vec \alpha} \DD^{\vec \alpha},
   \quad
   \Bcal_j(\DD) = \sum_{\abs{\vec \beta} = m_j} b_{j,\vec \beta} \DD^{\vec \beta},
  \]
 where $m \in \N$ is a fixed natural number and $m_j \in \{0, 1, \ldots, 2m-1\}$ and corresponds to the differentiability order of the operator $\Bcal_j(\DD)$.
  The coefficients $a_{\vec \alpha}, b_{j,\vec \beta} \in \B(E)$, where $E$ may be some arbitrary Banach space, are assumed to be spatially and temporally constant. We employ the multi-index notation $\DD^{\vec \alpha} = \prod_{i=1}^{n+1} \DD_{x_i}^{\alpha_i}$ with $\DD_{x_i} = - \ii \frac{\partial}{\partial x_i}$ and write $\abs{\vec \alpha} = \sum_{i=1}^{n+1} \alpha_i$ for the length of a vector of natural numbers $\vec \alpha = (\alpha_1, \ldots, \alpha_{n+1})^\mathsf{T} \in \N_0^{n+1}$.
 Then, the interior symbol $\Acal(\vec \xi)$ and each of the boundary symbols $\Bcal_j(\vec \xi)$
  \[
   \Acal(\vec \xi)
    = \sum_{\abs{\vec \alpha} = 2m} a_{\vec \alpha} \vec \xi^{\vec \alpha},
    \quad
   \Bcal_j(\vec \xi)
    = \sum_{\abs{\vec \beta} = m_j} b_{j,\vec \beta} \vec \xi^{\vec \beta}
    \quad (j = 1, \ldots, m)
  \]
 are homogeneous in $\vec \xi \in \R^{n+1}$ of orders $2m$ for $\Acal$ and $m_j$ for $\Bcal_j$, $j = 1, \ldots, m$, respectively.
 (Throughout, we use the short-hand notation $\vec \xi^{\vec \alpha} := \prod_{i=1}^{n+1} \xi_i^{\alpha_i}$ etc.)
 To establish an existence and uniqueness theory for solutions of the parabolic equation \ref{eqn:PIBVP}, in \cite[Subsection 6.1]{DeHiPr03}, one may, first only formally, apply the Laplace transform w.r.\ to the time variable $t \in (0, \infty)$, i.e.\ we consider
  \[
   (\Lcal u)(\lambda, \vec x)
    = \int_0^\infty \ee^{- \lambda t} u(t,\vec x) \, \dd t
    \quad
    \text{for some set of parameters }
    \lambda \in U \subseteq \C \text{ and } \vec x \in \R_+^{n+1}
  \]
 so that the parabolic system transforms into a family of elliptic boundary value problems on the half-space $\R_+^{n+1}$ which are given by
  \begin{align}
   \lambda (\Lcal u)(\lambda,\vec x) + \Acal(\DD) (\Lcal u)(\lambda,\vec x)
    &= (\Lcal f)(\lambda,\vec x) + u(0,\vec x)
    &&\text{for } (\lambda, \vec x) \in U \times \R_+^{n+1},
    \label{eqn:Elliptic_PDE_Laplace-Space}
    \\
   \Bcal_j(\DD) (\Lcal u)(\lambda,\vec x)
    &= (\Lcal g_j)(\lambda,\vec x)
    &&\text{for } (\lambda, \vec x) \in U \times \partial \R_+^{n+1}, \, j = 1, \ldots, m.
    \label{eqn:Elliptic_PDE_Laplace-Space_Boundary}
  \end{align}   
 In \cite{DeHiPr03}, this provides the first step towards more general interior and exterior symbols.
 More general results are then derived using operator-theoretic results from modern harmonic analysis and employing perturbation methods.
 In this manuscript, however, we want to cover situations similar to those in \cite{AugBot21a}, where demanding homogeneity of the (adjusted notion of the) principle part of the boundary symbols cannot be guaranteed. Therefore, in comparison with \cite{DeHiPr03}, we relax the conditions on the \emph{principle parts} of the boundary operators in the following way:
  \begin{assumption}
   The interior symbol $\Acal(\vec \xi) = \sum_{\abs{\vec \alpha} = 2m} a_{\vec \alpha} \vec \xi^{\vec \alpha}$ is homogeneous of degree $2m$ for some $m \in \N$ and parameter-elliptic with angle of ellipticity $\phi^\mathrm{ellipt}_{\Acal} \in [0, \pi)$, i.e.\ the spectrum is uniformly contained in some sector,
  \begin{equation}
   \sigma(\Acal(\vec \xi)) \subseteq \Sigma_{\pi - \phi}
    \quad \text{for all }
    \vec \xi \in \R^{n+1} \text{ with } \abs{\vec \xi} = 1,
  \end{equation}
  for some $\phi \in [0, \pi)$ and where we put $\phi^\mathrm{ellipt}_{\Acal} \in [0, \pi)$ to be the infimum of all $\phi$ which are admissible for this inclusion.
  \newline
   Moreover, we assume that for $j = 1, \ldots, m$ there exist continuous, linear projections $\Pcal_{j,k} \in \B(E)$, $k = 0, 1, \ldots, m_j < 2m$ such that $\Pcal_{j,k} \Pcal_{j,k'} = 0$ for $k \neq k'$.
   The boundary symbols $\Bcal_j(\vec \xi)$, $j = 1, \ldots, m$ are defined by
    \[
     \Bcal_j(\vec \xi)
      = \sum_{k=0}^{m_j} \Bcal_{j,k}(\vec \xi)
      \quad \text{with} \quad
     \Bcal_{j,k}(\vec \xi)
      =\sum_{\abs{\vec \beta} = k} b_{j,k,\vec \beta} \vec \xi^{\vec \beta}  \Pcal_{j,k}
      \quad
      \text{for } j = 1, \ldots, m \text{ and } k = 0, 1, \ldots, m_j.
    \] 
 We assume that $a_{\vec \alpha} \in \B(E)$ and $b_{j,k,\vec \beta} \in \B(E)$ denote constant coefficients and such that $b_{j,k,\vec \beta}(\ran (\Pcal_{j,k})) \subseteq \ran(\Pcal_{j,k})$ for all $\vec \beta \in \N_0^{n+1}$ with $\abs{\vec \beta} = k$.
  \end{assumption}
 With a slight abuse of notation, for the remainder of this section we focus on the elliptic boundary value problem and simply write $u$, $f$ and $g_j$ for what should actually be the Laplace-transformed versions $\Lcal u$, $\Lcal f - u(0,\cdot)$ and $\Lcal g_j$, respectively.
 We fix an integrability exponent from the reflexive range $p \in (1, \infty)$ and an angle $\phi > \phi^\mathrm{ellipt}_{\Acal}$, and consider the boundary value problem 
  \[
   \begin{cases}
    \lambda u + \Acal(\DD) u = f
    &\text{in } \R_+^{n+1}
    \\
    \Bcal_j(\DD) u = g_j
    &\text{on } \partial \R_+^{n+1},
    \quad
    j = 1, \ldots, m
   \end{cases}
  \] 
 for a given parameter $\lambda \in \Sigma_{\pi-\phi}$, and given data $f \in \LL_p(\R_+^{n+1};E)$ and $g_j = \sum_{k=0}^{m_j} g_{j,k}$ such that $g_{j,k} \in \WW_p^{2m-k-1/p}(\R^n;\ran(\Pcal_{j,k}))$ for $j = 1, \ldots, m$ and $k = 0, 1, \ldots, m_j < 2m$.
 The function $f$ can be extended trivially by zero to $\Ecal_0 f \in \LL_p(\R^{n+1};E)$ on the full space $\R^{n+1}$; below $\Ecal_0$ will always denote this trivial extension operator.
 Moreover, by assumption we may write $g_j = \sum_{k=0}^{m_j} g_{j,k} = \sum_{k=0}^{m_j} \Pcal_{j,k} g_j$ (which is a consistent choice since all operators $\Pcal_{j,k}$ are projections such that $\Pcal_{j,k} \Pcal_{j,k'} = 0$ for $k \neq k'$).
 To solve the elliptic boundary value problem, we employ the following strategy (cf.\ \cite{DeHiPr03}):
 First, we solve the inhomogeneous problem on the full space
  \[
   \lambda w^\mathrm{fs} + \Acal(\DD) w^\mathrm{fs}
    = \Ecal_0 f
    \quad
    \text{on } \R^{n+1}.
  \]
  This is somehow the easiest part, since we may employ the full space theory developed in \cite[Section 5]{DeHiPr03}, to write $\vec \xi = (\vec \xi', y)$ and
   \[
    w^\mathrm{fs}(\vec \xi', y, \lambda)
     = (P_{\R_+^{n+1}} (\lambda + A_{\R^{n+1}})^{-1} \Ecal_0 f)(\vec \xi',y).
   \]
  Here, we write $P_{\R_+^{n+1}}$ for the restriction from the full space $\R^{n+1}$ to the half-space $\R_+^{n+1}$, and $A_{\R^{n+1}}$ for the $\LL_p(\R^{n+1})$-realisation of the differential operator $\Acal(\DD)$, i.e.\ $A_{\R^{n+1}} u = \Acal(\DD)u$ on $\dom(A_{\R^{n+1}}) = \{ u \in \LL_p(\R^{n+1}): \, \Acal(\DD)u \in \LL_p(\R^{n+1})\}$.
  We refer to \cite{DeHiPr03} for details, but in view of the inhomogeneous boundary problem on the half-space considered next, let us sketch the basic idea first:
  Using the Fourier transformation on $\R^{n+1}$, initially defined for functions $u \in \LL^1(\R^{n+1};E)$ by
   \[
    (\Fcal_{\R^{n+1}} u)(\vec \xi)
     = \int_{\R^{n+1}} u(\vec x) \ee^{- \ii \vec x \cdot \vec \xi} \, \dd \vec x,
   \]
  the (elliptic) full space problem $(\lambda + \Acal(\DD)) w^\mathrm{fs} = \Ecal_0 f$ in Fourier space becomes the parameter $\vec \xi \in \R^{n+1}$ dependent algebraic system of linear equations on $E$
   \[
    (\lambda + \Acal(\vec \xi)) (\Fcal_{\R^{n+1}} w^\mathrm{fs})(\lambda, \vec \xi) = (\Fcal_{\R^{n+1}} \Ecal_0 f)(\lambda, \vec \xi)
     \quad
     \text{for }
     \lambda \in \Sigma_{\pi-\phi} \text{ and } \vec \xi \in \R^{n+1}.
   \]
 (Note that the solution to this problem is not subject to any boundary conditions prescribed on $\partial \R_+^{n+1}$, so that in a further step, we have to subtract a corrective solution to the problem $(\lambda + \Acal(\DD)) w^\mathrm{corr} = 0$ in $\R_+^{n+1}$ and $\Bcal_j(\DD) w^\mathrm{corr} = - \Bcal_j(\DD) w^\mathrm{fs}$ on $\partial \R_+^{n+1}$ from the constructed solution.)
 Subsequently we consider the half-space problem for the difference $w := u - P_{\R_+^{n+1}} w^\mathrm{fs}$ between the solution $u$ to the original problem which we seek for, and the full-space problem with right-hand side $\Ecal_0 f$.
 This new problem is homogeneous in $\R_+^{n+1}$, but inhomogeneous at the boundary $\partial \R_+^{n+1}$:
 For $u$ to be a solution to \eqref{eqn:Elliptic_PDE_Laplace-Space}--\eqref{eqn:Elliptic_PDE_Laplace-Space_Boundary}, this function $w$ has to solve the elliptic boundary value problem
  \begin{equation}
   \begin{cases}
    \lambda w + \Acal(\DD) w
     = 0
     &\text{in } \R_+^{n+1},
     \\
    \Bcal_j(\DD) w
     = g_j - \Bcal_j(\DD) w^\mathrm{fs}
     &\text{on } \partial \R_+^{n+1}, \, j = 1, \ldots, m.
   \end{cases}
   \label{eqn:EBVP}
   \tag{EBVP}
  \end{equation}
 Due to the special structure of the half-space $\R_+^{n+1} = \R^n \times (0, \infty)$, this inhomogeneous boundary value problem can be approached by performing the \emph{partial Fourier transform} in $\vec x'$, i.e.\ in the first $n$ variables of the vector $\vec x \in \R_+^{n+1}$:
 Starting with functions $u = u(\vec x',y) \in \LL_1(\R_+^{n+1})$ we introduce the partial Fourier transform in $\vec x'$ as
  \[
   \Fcal u(\vec \xi',y)
    = \int_{\R^n} \ee^{- \ii \vec x' \cdot \vec \xi'} u(\vec x',y) \, \dd \vec x'.
  \]
 Formally applying the partial Fourier transformation to the boundary value problem \eqref{eqn:EBVP}, leads to the elliptic boundary value problem in the Laplace--Fourier space
  \begin{align}
   \lambda \Fcal u (\vec \xi',y) + \sum_{l=0}^{2m} \tilde a_l(\vec \xi') \DD_y^{2m-l} \Fcal u(\vec \xi',y)
    &= \Fcal f (\vec \xi',y)
    &&\text{in } \R_+^{n+1},
    \label{eqn:Elliptic_BVP_Laplace-Fourier-Space}
    \\
   \Fcal (\Bcal_j u)(\vec \xi',0)
    = \sum_{k=0}^{m_j} \sum_{l=0}^k \tilde b_{j,k,l}(\vec \xi') \DD_y^{k-l} \Pcal_{j,k} \Fcal u(\vec \xi',0)
    &= \Fcal g_j(\vec \xi',0)
    &&\text{on } \partial \R_+^{n+1}, \, j = 1, \ldots, m, 
    \label{eqn:Elliptic_BVP_Laplace-Fourier-Space_Boundary}
  \end{align}
 where we merge coefficients and factors of the same $(n+1)$-st index $\alpha_{n+1}$ and $\beta_{n+1}$, resp., and write
  \begin{align*}
   \tilde a_l(\vec \xi')
    &= \sum_{\abs{\vec \alpha'} = l} a_{(\vec \alpha',2m-l)} (\vec \xi')^{\vec \alpha'},
    &&l = 0, 1, \ldots, 2m,
    \\    
   \tilde b_{j,k,l}(\vec \xi')
    &= \sum_{\abs{\vec \beta'} = l} b_{j,k,(\vec \beta',k-l)} (\vec \xi')^{\vec \beta'},
    &&j = 1, \ldots, m, \, k = 1, \ldots, m_j, \, l = 0, 1, \ldots, k.
  \end{align*}
 Note that by construction the terms $\tilde a_l(\vec \xi')$ and $\tilde b_{j,k,l} (\vec \xi')$ are homogeneous of order $l$ in $\vec \xi' \in \R^n$.
 The vector-valued ordinary differential equation \eqref{eqn:Elliptic_BVP_Laplace-Fourier-Space} is equivalent to a first order system (in $y$) on the space $E^{2m}$, as can be seen by introducing the $E^{2m}$--valued functions
  \[
   \Fcal \vec v(\lambda, \vec \xi', y)
    := (\Fcal u(\lambda, \vec \xi', y), \frac{1}{\rho} \DD_y \Fcal u(\lambda, \vec \xi', y), \ldots, \frac{1}{\rho^{2m-1}} \DD_y^{2m-1} \Fcal u(\lambda, \vec \xi', y))^\mathsf{T},
  \]
 where $\rho > 0$ is some additional parameter.
 To exploit the homogeneity of $\Acal(\vec \xi)$, $\Bcal_{j,k}(\vec \xi)$ in $\vec \xi$ later on, we closely follow the lines of \cite{DeHiPr03} and introduce
  \[
   \sigma := \frac{\lambda}{\rho^{2m}} \in \Sigma_{\pi-\phi}
    \quad \text{and} \quad
   \vec b := \frac{\vec \xi'}{\rho} \in \R^n
  \] 
 as new variables in the Laplace--Fourier space.
 From now on, whenever $\sigma$ and $\vec b$ will appear, they are related to the variables $\lambda \in \Sigma_{\pi - \phi}$ and $\vec \xi' \in \R^n$ and the parameter $\rho > 0$ by these defining relations:
 In particular, we have
  \[
   (\abs{\lambda}^{1/m} + \abs{\vec \xi'}^2)^{1/2}
    = \rho (\abs{\sigma}^{1/m} + \abs{\vec b}^2)^{1/2}
     \quad \text{and} \quad
   (\abs{\lambda} + \abs{\vec \xi'}^{2m})^{1/2m}
    = \rho (\abs{\sigma} + \abs{\vec b}^{2m})^{1/2}.
  \]
 \begin{remark}
 Later on we will choose the parameter $\rho = \rho(\lambda, \vec \xi')$ depending on the variables $\lambda \in \Sigma_{\pi-\phi}$ and $\vec \xi' \in \R^n$.
 The particular choice of $\rho = \rho(\lambda, \vec \xi')$ will be $\rho = (\lambda + \abs{\vec \xi'}^{2m})^{1/2m}$, in line with the choice in \cite{DeHiPr07}, but in contrast to the earlier choice $\rho = (\abs{\lambda}^{1/m} + \abs{\vec \xi'}^2)^{1/2}$ in \cite{DeHiPr03}.
 The reason behind this are the respective interpretations of $\rho(\lambda, \vec \xi')$ as a multiplication operator in Fourier space: For the choices just mentioned, $\rho(\lambda, \vec \xi')$ will be the Fourier symbol of the pseudo-differential operators $((-\Delta)^m + \lambda)^{1/2m}$ and $((-\Delta) + \abs{\lambda}^{1/m})^{1/2}$, respectively.
 As it turned out, the former choice is more practical, see \cite{DeHiPr07}.
 \end{remark}
 Let us continue with the formulation of the elliptic BVP \eqref{eqn:EBVP} as a first order system and introduce a matrix-valued function $\bb A_0: \R^n \times \Sigma_{\pi - \phi} \ni (\vec \xi',\lambda) \mapsto \bb A_0(\vec \xi',\lambda) \in \B(E^{2m})$ by setting
  \[
   \bb A_0(\vec \xi',\lambda)
    = \left[ \begin{array}{ccccc} 0 & \id_E & 0 & \cdots & 0 \\ 0 & 0 & \id_E && 0 \\ \vdots & \vdots &\, \ddots &\, \ddots & \\ 0 & \cdots && 0 & \id_E \\ c_{2m}(\vec \xi',\lambda) & c_{2m-1}(\vec \xi') & \cdots && c_1(\vec \xi') \end{array} \right],
  \]
 where we choose the (operator-valued) entries of the operator matrix $\bb A_0(\vec \xi', \lambda)$ as
  \[
   c_j(\vec \xi') = - a_0^{-1} \tilde a_j(\vec \xi')
    \quad
    \text{for } j = 1, \ldots, 2m-1 \text{ and } c_{2m}(\vec \xi',\lambda) = - a_0^{-1} (\tilde a_{2m}(\vec \xi') + \lambda \id_E),
  \]
 and $\id_E$ denotes the identity map on $E$.
 By parameter-ellipticity of the interior symbol $\Acal(\vec \xi)$, the ($\vec \xi'$-independent) coefficient $a_0(\vec \xi') = a_0 \in \B(E)$ is invertible because
  \[
   a_0(\vec \xi')
    = a_{(\vec 0, 2m)}
    = \Acal(\vec 0, 2m)
    \quad
    \text{for all } \vec \xi' \in \R^n,
  \]
 where $(\vec 0, 2m) = (0, \ldots, 0, 2m) \in \R^{n+1}$.
 Since each of the symbols $\tilde a_l(\vec \xi')$ is homogeneous in $\vec \xi' \in \R^n$ of order $l$, we may, reformulate the elliptic BVP \eqref{eqn:Elliptic_BVP_Laplace-Fourier-Space} (with the homogeneous choice $f=0$) equivalently as a first order in $y \in (0,\infty)$ system of ordinary differential equations, viz.\
  \[
   \frac{\partial}{\partial y} \Fcal \vec v^\mathrm{hs}(\lambda, \vec \xi',y)
    = \ii \rho \bb A_0(\vec b, \sigma) \Fcal \vec v^\mathrm{hs}(\lambda, \vec \xi', y),
    \quad
    \lambda \in \Sigma_{\pi-\phi}, \, \vec \xi' \in \R^n, \, y > 0.
  \]
 In contrast to the situation in \cite{DeHiPr03}, this time, the boundary symbols $\Bcal_j(\vec \xi) = \sum_{k=0}^{m_j} \Bcal_{j,k}(\vec \xi)$ are \emph{not} homogeneous in $\vec \xi \in \R^{n+1}$ of order $m_j$, in general, but only each of the symbols $\Bcal_{j,k}(\vec \xi) := \Bcal_j(\vec \xi) \Pcal_{j,k} = \Pcal_{j,k} \Bcal_j(\vec \xi) \Pcal_{j,k} $  is homogeneous in $\vec \xi \in \R^{n+1}$ of order $k$, for each $j = 1, \ldots, m$ and $k = 0, 1, \ldots, m_j$.
 Hence, each (operator-valued) coefficient $\tilde b_{j,k,l}(\vec \xi') \in \B(E)$ is homogeneous in $\vec \xi' \in \R^n$ of degree $l$, and we may write
  \begin{align*}
   \Fcal g_j(\rho \vec b, 0)
    &= \Fcal \Bcal_j u(\rho \vec b,0)
    = \sum_{k=0}^{m_j} \sum_{l=0}^k \tilde b_{j,k,l}(\rho \vec b) \DD_y^{k-l} \Pcal_{j,k} \Fcal u(\rho \vec b,0)
    \\
    &= \sum_{k=0}^{m_j} \sum_{l=0}^k \rho^l \tilde b_{j,k,l}(\vec b) \DD_y^{k-l} \Pcal_{j,k} \Fcal u(\rho \vec b,0)
    \\
    &= \sum_{k=0}^{m_j} \sum_{l=0}^k \rho^l \Pcal_{j,k} \tilde b_{j,k,l}(\vec b) \DD_y^{k-l} \Pcal_{j,k} \Fcal u(\rho \vec b,0)
    \quad
    \text{for each } \vec b \in \SS^{n-1} \text{ and } \rho > 0.
  \end{align*}
Therefore, the boundary condition \eqref{eqn:Elliptic_BVP_Laplace-Fourier-Space_Boundary} is equivalently expressed as
  \begin{equation*}
   \vec \Bcal_{j,k}^0(\vec b) \Fcal \vec v(\rho \vec b, 0)
    =\frac{\Pcal_{j,k} \Fcal g_j(\rho \vec b,0)}{\rho^k}
    \quad
    \text{for }
    j = 1, \ldots, m \text{ and } k = 0, 1, \ldots, m_j,
  \end{equation*}
 where we define the row vector $\vec \Bcal_{j,k}^0(\vec b)$ as follows:
  \[
   \vec \Bcal_{j,k}^0(\vec b)
    = (\tilde b_{j,k,k}(\vec b) \Pcal_{j,k}, \ldots, \tilde b_{j,k,0} \Pcal_{j,k}, 0, \ldots, 0)
    \quad
    \text{for }
    j = 1, \ldots, m \text{ and } k = 0, 1, \ldots, m_j.
  \]
 (At this point, note that the inclusion $\tilde b_{j,k,l}(\ran (\Pcal_{j,k})) \subseteq \ran(\Pcal_{j,k})$ for the range of the projections $\Pcal_{j,k}$ is valid, since we assumed that $b_{j,k,\vec{\beta}}(\ran (\Pcal_{j,k})) \subseteq \ran(\Pcal_{j,k})$ for each $\vec \beta \in \N_0^{n+1}$ of length $\abs{\vec \beta} = k$.)
 Then \cite[Proposition 6.1]{DeHiPr03} applies and, therefore, the spectrum of $\bb A_0(\vec b, \sigma) \in \B(E^{2m})$ is strictly separated by the real axis, i.e.\ $\sigma(\bb A_0(\vec b, \sigma)) \cap \R = \emptyset$ and, hence, $\ii \sigma(\bb A_0(\vec b,\sigma)) = \sigma(\ii \bb A_0(\vec b, \sigma)) = S_+(\vec b, \sigma) \cup S_-(\vec b, \sigma)$ splits into two parts $S_\pm(\vec b,\sigma)$ which are \emph{uniformly} (for $(\vec b, \sigma) \in \R^n \times \Sigma_{\pi-\phi}$ with $\abs{\vec b}, \abs{\sigma} \leq 1$) separated by a strip around the imaginary axis.
 This means that, there are constants $c_1, c_2 > 0$ such that
  \begin{align}
   \inf \{ \Re \mu: \, \mu \in S_+(\vec b, \sigma) \}
    &\geq c_1,
    &&\abs{\vec b} \leq 1, \, \abs{\sigma} \leq 1,
    \label{eqn:Spectral_Gap_S+}
    \\
   \sup \{ \Re \mu: \, \mu \in S_-(\vec b, \sigma) \}
    &\leq - c_2,
    &&\abs{\vec b} \leq 1, \, \abs{\sigma} \leq 1.
    \label{eqn:Spectral_Gap_S-}
  \end{align}
  \begin{figure}
	\centering
	\includegraphics[scale = 0.5]{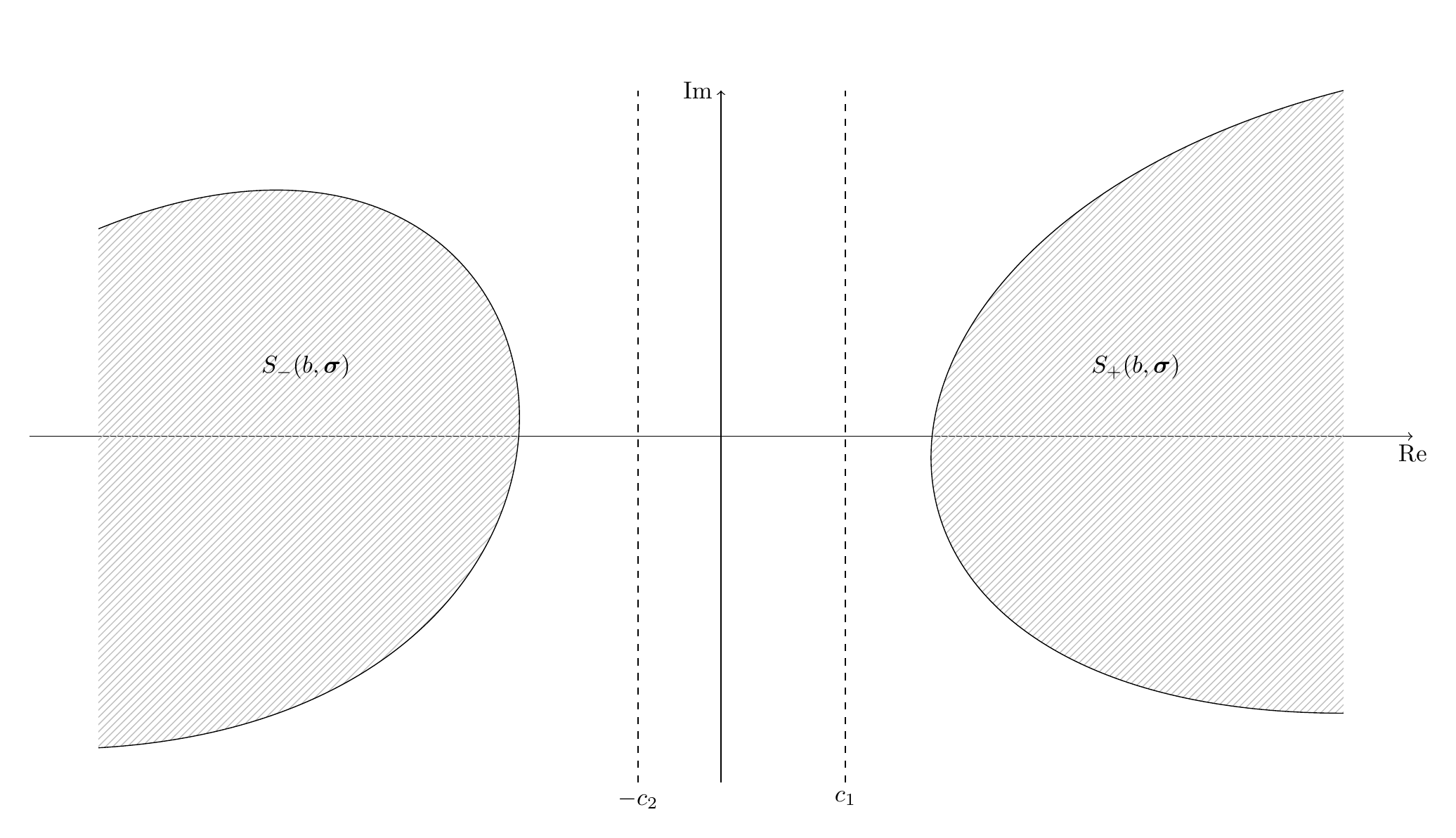}
	\caption{Spectral gap between parts $S_\pm(\vec b, \sigma)$ of the spectrum $\ii \sigma(\bb A_0(\vec b, \Sigma))$.}
	\label{fig:diagram_spectral_gap}
 \end{figure}
 In particular, for each $(\vec b, \sigma)$ the spectral projections $P_\pm(\vec b,\sigma) \in \B(E^{2m})$ onto $S_\pm(\vec b,\sigma)$ exist.
 
 \subsection{The Lopatinskii-Shapiro condition on the half-space}
 
 The \emph{Lopatinskii--Shapiro condition} provides a solvability and uniqueness property for the Fourier space representation of the elliptic boundary value problem \eqref{eqn:Elliptic_PDE_Laplace-Space}--\eqref{eqn:Elliptic_PDE_Laplace-Space_Boundary}.
 It can be stated as follows:
 
 \begin{assumption}[Lopatinskii--Shapiro condition, half-space version]
 \label{assmpt:LSC}
  For each $\lambda \in \overline{\Sigma}_{\pi - \phi}$ and $\vec \xi' \in \R^n$ such that $(\lambda, \vec \xi') \neq (0, \vec 0)$, the initial value problem
   \begin{equation*}
    \begin{cases}
     \lambda v(y) + \Acal(\vec \xi',\DD_y) v(y)
      = 0
      &\text{for } y > 0,
      \\
     \Bcal_j(\vec \xi',\DD_y) v(0)
      = g_j
      &\text{for } j = 1, \ldots, m
    \end{cases}
   \end{equation*}
  admits a unique solution in the class $v \in \CC_0(\R_+; E) = \{ v \in \CC(\R_+;E): \, v(y) \rightarrow 0 \text{ as } y \rightarrow \infty \} $, for each given $\vec g = (g_1, \ldots, g_m)^\mathsf{T} \in E^m$.
 \end{assumption}
 
 \begin{remark}
  Below, we introduce a map $\tilde D \subseteq \C^{n+1} \rightarrow \B(E^m;E^{2m})$ which (roughly) maps, for each set of parameters $(\vec \xi', \lambda) \in \tilde D$, any given data $g_j \in E$ to the time trace at zero $\vec v(0) \in E^{2m}$ of the unique solution to the elliptic initial value problem from the Lopatinskii--Shapiro condition.
  For technical reasons, we want $\tilde D$ to be an open \emph{complex} neighbourhood of $(\R^n \times \Sigma_{\pi-\phi}) \setminus \{\vec 0\} \subseteq \C^{n+1}$.
  We then can prove complex differentiability of this map, provided we know that the Lopatinskii--Shapiro condition is actually also valid for values $(\lambda, \vec \xi')$ from this \emph{complex} neighbourhood $\tilde D$.
  This technical detail had not been addressed in the original work \cite{DeHiPr03}, however, see \cite[Lemma 5.6]{Kam15} in a master thesis supervised by Robert Denk. By this, the Lopatinskii--Shapiro condition is automatically valid in a complex neighbourhood $\tilde D$ of $(\R^n \times \overline{\Sigma_{\pi-\phi}}) \setminus \{\vec 0\} \subseteq \C^{n+1}$, provided it is valid on $(\R^n \times \overline{\Sigma_{\pi-\phi}}) \setminus \{\vec 0\}$.
 \end{remark}

To proceed, we modify the proofs of \cite[Propositions 6.2--6.4]{DeHiPr03} and transfer them to the situation considered here.
 We, thereby, obtain the following preliminary result on existence and uniqueness of solutions and a solution formula for the Fourier space representation of the elliptic boundary value problems in half-space which satisfy the Lopatinskii--Shapiro condition.
 \begin{proposition}
 \label{prop:EBVP_half_space_solution}
  Let $\Acal(\DD)$ be a homogeneous, parameter-elliptic operator of order $2m$ with angle of ellipticity $\phi_\Acal^\mathrm{ellipt} \in [0, \pi)$. Let $\phi > \phi_\Acal^\mathrm{ellipt}$ and $\lambda \in \overline{\Sigma}_{\pi-\phi}$ be given and assume that the Lopatinskii-Shapiro condition is valid.
  Then there is a unique solution $\Fcal u$ of the system \eqref{eqn:Elliptic_BVP_Laplace-Fourier-Space}--\eqref{eqn:Elliptic_BVP_Laplace-Fourier-Space_Boundary}, which is given by the sum
   \[
    \Fcal u
     = \Fcal w^\mathrm{fs} + \Fcal w^\mathrm{hs} + \Fcal w^\mathrm{corr}
     = \Fcal w^\mathrm{fs} + (\Fcal \vec v^\mathrm{hs})_1 + (\Fcal \vec v^\mathrm{corr})_1.
   \]
  Here, $\Fcal w^\mathrm{fs}$ denotes the solution to the full-space problem with right-hand side $\Fcal \Ecal_0 f$, $\Fcal w^\mathrm{hs}$ is the solution of \eqref{eqn:Elliptic_BVP_Laplace-Fourier-Space}--\eqref{eqn:Elliptic_BVP_Laplace-Fourier-Space_Boundary} for homogeneous right-hand side $f = 0$ in $\R_+^{n+1}$ and boundary conditions $\Bcal_j w^\mathrm{hs} = g_j$ on $\partial \R_+^{n+1}$, $j = 1, \ldots, m$, whereas $\Fcal w^\mathrm{corr}$ is a correction term with boundary conditions $ \Bcal_j \Fcal w^\mathrm{corr} = - \Bcal_j \Fcal w^\mathrm{fs}$ matching those from the solution to the full-space problem on $\partial \R_+^{n+1}$ (and $f = 0$ in $\R_+^{n+1}$).
  Moreover, $(\cdot)_1$ denotes the first component of any vector in $E^{2m}$.
  The function $\Fcal \vec v^\mathrm{hs}$ can be computed using the uniformly continuous semigroup generated by $\ii \rho \bb A_0(\vec b, \sigma)$,
   \begin{equation}
    \Fcal \vec v^\mathrm{hs}(\vec \xi',y,\lambda)
     = \ee^{\ii \rho \bb A_0(\vec b,\sigma)y} \Fcal \vec v^\mathrm{hs}(\vec \xi',0,\lambda)
     \quad \text{for }
     \lambda = \rho^{2m} \sigma \in \Sigma_{\pi-\phi}, \, \vec \xi' = \rho \vec b \in \R^n, \, y > 0,
     \label{eqn:First_Order_System_in_y_scaled}
  \end{equation}
 where the initial value $\Fcal \vec v^\mathrm{hs}(\vec \xi',0,\lambda)$ has to be chosen according to
  \[
   \Fcal \vec v^\mathrm{hs}(\vec \xi',0,\lambda)
    = \bb M(\vec b,\sigma) \Fcal \vec g_\rho^\mathrm{hs}(\vec \xi',0).
  \]
  The (operator-valued) function $\bb M: \tilde D \mapsto \B(E^m;E^{2m})$ is jointly holomorphic in $(\vec b, \sigma) \in \tilde D$ in a complex neighbourhood $\tilde D$ of $(\R^n \times \overline{\Sigma_{\pi - \phi}}) \setminus \{\vec 0\} \subseteq \C^{n+1}$, and $\bb M(\vec b,\sigma)$ maps any given data
  \[
   \Fcal \vec g_\rho^\mathrm{hs}(\vec \xi',0)
    = \left( \sum_{k=0}^{m_1} \tfrac{\Pcal_{1,k} \Fcal g_1(\vec \xi',0)}{\rho^k}, \ldots, \sum_{k=0}^{m_m} \tfrac{\Pcal_{m,k} \Fcal g_m(\vec \xi',0)}{\rho^k} \right),
    \quad
   \vec \xi' = \rho \vec b \in \R^n, \, \rho > 0.
  \]
 to the corresponding unique solution of the algebraic problem
  \[
   \begin{cases}
     P_+(\vec b, \sigma) \Fcal \vec v(\vec \xi', 0, \lambda)
      = 0,
      &\vec \xi' \in \R^n,
      \\
     B_0^{j,k}(\vec b) \Fcal \vec v(\vec \xi', 0, \lambda)
      = \frac{\Pcal_{j,k} \Fcal g_j(\vec \xi', 0)}{\rho^k},
      &j = 1, \ldots, 2m, \, k = 0, 1, \ldots, m_j.
   \end{cases}
  \]
   On the other hand, the correction term $\Fcal \vec v^\mathrm{corr}$ is given by
  \[
   \Fcal \vec v^\mathrm{corr}(\vec \xi',y,\lambda)
    = - \ee^{\ii \rho \bb A_0(\vec b,y)} \bb M(\vec b,\sigma) (\Fcal \vec g_\rho^\mathrm{corr})(\vec \xi',0)
   \quad \text{for }
   \lambda = \rho^{2m} \sigma \in \Sigma_{\pi-\phi}, \, \vec \xi' = \rho \vec b \in \R^n, \, y > 0,
  \]
   where
  \[
    \Fcal \vec g_\rho^\mathrm{corr}(\vec \xi',0)
     = \left( \sum_{k=0}^{m_1} \tfrac{\Fcal g^\mathrm{corr}_{1,k}(\vec \xi',0)}{\rho^k}, \ldots, \sum_{k=0}^{m_m} \tfrac{\Fcal g^\mathrm{corr}_{m,k}(\vec \xi',0)}{\rho^k} \right),
    \quad
    \vec \xi' = \rho \vec b \in \R^n, \, \rho > 0.
  \]
   Here, the functions
   \[
    \Fcal g^\mathrm{corr}_{j,k}(\vec \xi',0)
     = \int_0^\infty h^\mathrm{corr}_{j,k}(\vec \xi',s) (\Fcal f)(\vec \xi',s) \, \dd s,
     \quad
     j = 1, \ldots, m, \, k = 0, 1, \ldots, m_j, \, \vec \xi' \in \R^n
   \]
   are defined using the integral operator kernels
  \[
    h^\mathrm{corr}_{j,k}(\vec \xi',s)
     = \int_{\R} \Bcal_j(\vec \xi',\eta) \Pcal_{j,k}
     (\lambda + \Acal(\vec \xi',\eta))^{-1} \ee^{- \ii \eta s} \, \dd \eta
    \quad \text{for } \vec \xi' \in \R^n, \, s > 0.
  \]
 \end{proposition}  
 \begin{proof}
  To establish Proposition \ref{prop:EBVP_half_space_solution}, we proceed as follows and analogously to \cite{DeHiPr03}:
  The linear problem is decomposed into three sub-problems, each of which for itself is easier to handle than the original problem:
   \begin{align*}
    \lambda \Fcal w^\mathrm{fs} + \Acal(\DD) \Fcal w^\mathrm{fs}
     &= \Ecal_0 \Fcal f
     &&\text{in } \R^{n+1}
     \\
     \intertext{(obviously, without any boundary conditions),}
    \lambda \Fcal w^\mathrm{hs} + \Acal(\DD) \Fcal w^\mathrm{hs}
     &= 0
     &&\text{in } \R_+^{n+1},
     \\
    \Bcal_j(\DD) \Fcal w^\mathrm{hs}
     &= g_j,
     &&\text{on } \partial \R_+^{n+1}, \, j = 1, \ldots, m,
     \\
     \intertext{and}
    \lambda \Fcal w^\mathrm{corr} + \Acal(\DD) \Fcal w^\mathrm{corr}
     &= 0
     &&\text{in } \R_+^{n+1},
     \\
    \Bcal_j(\DD) \Fcal w^\mathrm{corr}
     &= - \Bcal_j(\DD) \Fcal w^\mathrm{fs}
     &&\text{on } \partial \R_+^{n+1}, \, j = 1, \ldots, m.
   \end{align*}
  The second problem, i.e.\ the case where $f = 0$, but inhomogeneous boundary conditions are imposed, may be handled similar to \cite[Proposition 5.3]{DeHiPr03} which provides us with a unique solution $(\Fcal \vec v^\mathrm{hs})_1$ to the problem \eqref{eqn:Elliptic_PDE_Laplace-Space}--\eqref{eqn:Elliptic_PDE_Laplace-Space_Boundary} for $f = 0$, and which has the form \eqref{eqn:First_Order_System_in_y_scaled}, where $\bb M(\vec b,\sigma)$ is jointly holomorphic in the variables $(\vec b,\sigma) \in \tilde D \supset (\R^n \times \overline{\Sigma_{\pi-\phi}}) \setminus \{\vec 0\}$. Since the projections $\Pcal_{j,k}$ project the problem to subspaces on which the corresponding system for the boundary operators is homogeneous in $\vec \xi'$ of degree $k$, uniqueness and existence of solutions, as well as the representation formula for the solution can be extracted from the equivalent reformulation as a first order (in $y \in (0,\infty)$) system. It, therefore, remains to ensure holomorphy of the map $\bb M$. This can be done almost literally as in the proof of \cite[Proposition 6.2]{DeHiPr03}:
  First, we employ the closed graph theorem to conclude that for each $\vec z \in D$ the operator $\bb M(\vec z)$ is a linear and closed map $E^m \rightarrow E^{2m}$, and it is uniformly bounded, first on the set $D = \overline{\BB_1(0)} \times \overline{\Sigma}_{\pi-\phi} \subseteq \R^n \times \C$, and, thereafter, in any complex neighbourhood $\tilde D$ of $D$, which is sufficiently close to $D$. With use of the spectral projections $P_-(\vec b, \sigma)$ and the Lopatinskii-Shapiro condition we may then prove continuity of the map $\vec z \mapsto \vec v(\vec z) = \bb M(\vec z) \vec g$, for any fixed $\vec g \in E^m$, on complex lines.
  This uses, in particular, that the Lopatinskii--Shapiro condition is valid on a complex neighbourhood $\tilde D$ of $(\R^n \times \overline{\Sigma_{\pi-\phi}}) \setminus \{\vec 0 \}$.
  Moreover, it uses that the spectral projections $P_-$ and the operators $\bb \Bcal_{j,k}^0$ are continuous.
  Complex differentiability follows.

 To derive an integral representation of the solution $\Fcal w^\mathrm{corr}$ to the correction problem, we use that, analogous to the half-space problem with given boundary data, $\Fcal w^\mathrm{corr}$ is the first component of the $E^{2m}$-valued function $\Fcal \vec v^\mathrm{corr}$, which can be formally written down as
  \[
   (\Fcal \vec v^\mathrm{corr})(\lambda, \vec \xi', y)
    = \ee^{\ii \rho \bb A_0(\vec b, \sigma) y} \bb M(\vec b, \sigma) \Fcal \vec g_\rho^\mathrm{corr}(\lambda, \vec \xi').
  \]
 The $\rho$-weighted version of the boundary data is
  \[
   \Fcal \vec g_\rho^\mathrm{corr}(\lambda, \vec \xi')
    = \big( \sum_{k=0}^{m_1} \frac{\Fcal g_{1,k}^\mathrm{corr}(\lambda, \vec \xi')}{\rho^k}, \ldots, \sum_{k=0}^{m_m} \frac{\Fcal g_{m,k}^\mathrm{corr}(\lambda, \vec \xi')}{\rho^k} \big),
  \]
 where the boundary data $\Fcal g_{j,k}^\mathrm{corr}$ result from taking the trace of $\Bcal_{j,k}(\vec \xi',\DD_y) \Fcal w^\mathrm{fs}(\vec \xi', y)$ at $y = 0$.
 We note that
  \[
   (\Fcal_{\vec x',y} w^\mathrm{fs})(\vec \xi', \eta)
    = (\lambda + \Acal(\vec \xi', \eta))^{-1} \Fcal_{\vec x',y} (\Ecal_0 f)(\vec \xi',y)
    = (\lambda + \Acal(\vec \xi, \eta))^{-1} (\Lcal_y \Fcal_{\vec x'} f)(\vec \xi', \ii y)
  \]
 since by definition $(\Ecal_0 f)(\vec x', y) = 0$ for $y < 0$.
 (By $\Fcal_{\vec x'}$ and $\Fcal_{\vec x', y}$ we mean the Fourier transform w.r.t.\ $\vec x'$ and $\vec x = (\vec x', y)$, respectively.)
 Thus,
  \begin{align*}
   \Bcal_{j,k}(\vec \xi', \eta) \Fcal_{\vec x',y} w^\mathrm{fs}(\vec \xi',\eta)
    &= \Fcal_{\vec x',y} \big( \Bcal_{j,k}(\bb D) w^\mathrm{fs} \big)(\vec \xi',\eta)
    \\
    &= \Bcal_{j,k}(\vec \xi',\eta) (\lambda + \Acal(\vec \xi',\eta))^{-1} \Fcal_{\vec x',y} (\Ecal_0 f)(\vec \xi', \eta)
    \\
    &= \int_{\R} \Bcal_{j,k}(\vec \xi',\eta) (\lambda + \Acal(\vec \xi',\eta))^{-1} \Fcal_{\vec x'} (\Ecal_0 f) (\vec \xi', y) \ee^{-\ii \eta y} \, \dd y.
  \end{align*}
 $\Bcal_{j,k}(\vec \xi', \DD_y) \Fcal_{\vec x'} w^\mathrm{fs}(\vec \xi',0)$ can, thus, be obtained by taking the inverse Laplace transform
  \begin{align*}
   \Bcal_{j,k}(\vec \xi', \DD_y) \Fcal_{\vec \xi'} w^\mathrm{fs}(\vec \xi',0)
    &= \int_0^\infty \int_{\R} \Bcal_{j,k}(\vec \xi',\eta) (\lambda + \Acal(\vec \xi',\eta))^{-1} \ee^{-\ii \eta y} (\Fcal_{\vec x'} \Ecal_0 f)(\vec \xi', y) \, \dd y \, \dd \eta.
  \end{align*} 
  For the last of the three sub-problems, the crucial point is to derive an integral representation of the boundary values $\Fcal g_j^\mathrm{corr}(\vec \xi',0) = (\Bcal_j(\DD) \Fcal P_{\R_+^{n+1}} (\lambda + A_{\R^{n+1}})^{-1} \Ecal_0 f)(\vec \xi',0)$ as $\Fcal g_j^\mathrm{corr}(\vec \xi',0) = \sum_{k=0}^{m_j} \Fcal g^\mathrm{corr}_{j,k}(\vec \xi,0)$ with
   \begin{align*}
    \Fcal g^\mathrm{corr}_{j,k}(\vec \xi',0)
     &= \int_0^\infty h^\mathrm{corr}_{j,k}(\vec \xi',s) (\Fcal f)(\vec \xi',s) \, \dd s
     \\
    h^\mathrm{corr}_{j,k}(\vec \xi',s)
     &= \int_{\R} \Bcal_j(\vec \xi',\eta) \Pcal_{j,k} (\lambda + \Acal(\vec \xi',\eta))^{-1} \ee^{-\ii \eta s} \, \dd \eta,
     \quad
     1 \leq j \leq m, \, 0 \leq k \leq m_j < 2m
   \end{align*}
 and to ensure that the kernels $h^\mathrm{corr}_{j,k}$ have good decay properties.
 \end{proof}
 
 \subsection{Kernel estimates}
 To actually solve the elliptic boundary value problem \eqref{eqn:EBVP} we may now use the kernel representation of the solution corresponding to the boundary value problem in Fourier space, and perform the inverse Laplace--Fourier transform.
 To check that the solution operators defined by this procedure are at least continuous on $\LL_p(\R_+^{n+1}; E)$, we may employ theory of integral operators.
 Roughly speaking, the following is true:
 If a linear operator on $\LL_p(\R_+^{n+1};E)$ has a kernel representation,  estimates on the kernel imply estimates on its operator norm.
 For a precise formulation of this statement, see e.g.\ \cite[Lemma 6.7]{DeHiPr03}.
 Its proof relies on Young's and H\"older's inequality.
 
 \begin{lemma}
  Let $p, p' \in (1,\infty)$ such that $\frac{1}{p} + \frac{1}{p'} = 1$ and $E$ be a Banach space and $k \in \LL_{p,p',1}(\R^n \times (0,\infty) \times (0,\infty))$.
  Then,
   \[
    (K f)(\vec x)
     := \int_0^\infty \int_{\R^n} k(\vec x' - \bar {\vec x}', y, \bar y) f(\bar {\vec x}', \bar y) \, \dd \bar {\vec x}' \, \dd \bar y
     \quad
     \text{for } f \in \LL_p(\R^{n+1};E), \, \text{a.e.\ } \vec x = (\vec x', y) \in \R_+^{n+1}
   \]
  defines a bounded linear operator $K \in \B(\LL_p(\R_+^{n+1};E))$ with operator norm which is less or equal the norm $\norm{k}_{p,p',1}$ of the kernel $k$.
 \end{lemma}

 For the kernel functions considered here, it proves helpful to define (analogously to \cite[Section 6.3]{DeHiPr03}) weighted versions of the functions $\Fcal \vec g_\rho^\mathrm{hs}$ and the semigroup associated to $\ii \rho \bb A_0(\vec b, \sigma)$ according to
  \begin{align*}
   \Fcal \vec g_\rho^{2m}(\vec \xi',0)
    &= \rho^{2m} \Fcal \vec g_\rho^\mathrm{hs}(\vec \xi',0)
    \in E^m
    &&\text{for } \vec \xi' \in \R^n, \, \rho > 0,
    \\
   \Fcal \bb V_\rho^{2m}(\vec \xi',y,\lambda)
    &= \frac{1}{\rho^{2m}} \ee^{\ii \rho \bb A_0(\vec b,\sigma) y} \bb M(\vec b, \sigma)
    \in \B(E^m,E^{2m})
    &&\text{for } \vec \xi' \in \R^n, \, y > 0 \text{ and } \rho > 0.
  \end{align*} 
 By formula \eqref{eqn:First_Order_System_in_y_scaled}, the identity
  \[
   \Fcal \vec v^\mathrm{hs}(\vec \xi',y,\lambda)
    = \Fcal \bb V_\rho^{2m}(\vec \xi',y,\lambda) \Fcal \vec g_\rho^{2m}(\vec \xi',0)
    \in E^{2m}
  \]
 is valid.
 Further, let us denote by
  \begin{equation}
   \bb K^\lambda(\vec x',y)
    = \Fcal^{-1} (\Fcal \bb V_\rho^{2m})(\vec x',y)
    = \frac{1}{(2\pi)^n} \int_{\R^n} \ee^{\ii \vec x' \cdot \vec \xi'} \Fcal \bb V_\rho^{2m}(\vec \xi',y,\lambda) \, \dd \vec \xi'
    \in \B(E^{2m}),
    \quad
    \vec x' \in \R^n, \, y > 0
    \label{eqn:Def_K_lambda}
  \end{equation}
 the inverse (partial) Fourier transform (w.r.t.\ $\vec \xi' \in \R^n$) of the $\Bcal(E)$-valued function $\Fcal \bb V_\rho^{2m}$.
 Then, the solution to the elliptic boundary value problem
  \begin{equation*}
   \begin{cases}
    (\lambda + \Acal(\DD)) w^\mathrm{hs}
     = 0
     &\text{in } \R_+^{n+1},
     \\
    \Bcal_j(\DD) w^\mathrm{hs}
     = g_j
     &\text{on } \partial \R_+^{n+1}, \, j = 1, \ldots, m
   \end{cases}
  \end{equation*}
 is (formally) given by
  \[
   w^\mathrm{hs}(\lambda, \vec x', y)
    = \int_{\R^n} [\bb K^\lambda(\vec x-\vec x',y)]_1 \Fcal^{-1} \vec g_\rho^{2m}(\vec x') \, \dd \vec x'
    \quad
    \text{for } \lambda \in \Sigma_{\pi-\phi}, \, \vec x' \in \R^n, \, y > 0
  \]
 where we denote by $[\bb K^\lambda(\vec x',y')]_1$ the first row of $\bb K^\lambda(\vec x',y')$.
  Here, $\Fcal \vec g_\rho^{2m}(\vec \xi', y) \in E^{2m}$ with components
  \[
   (\Fcal \vec g_\rho^{2m}(\vec \xi', y))_j
    = \sum_{k=0}^{m_j} \rho^{2m-k} \Fcal g_{j,k}(\vec \xi',y)
    = \sum_{k=0}^{m_j} ( \abs{\xi'}^{2m} + \abs{\lambda})^{(2m-k)/2m} \Fcal g_{j,k}(\vec \xi',y).
  \]
 Since $(\abs{\xi'}^{2m} + \abs{\lambda})^{(2m-k)/2m}$ is the Fourier symbol of the pseudo-differential operator $((- \Delta_{\R^n})^m + \abs{\lambda})^{(2m-k)/2m}$ for the realisation of the Laplacian $\Delta_{\R^n}$ in $\LL_p(\R^n;E)$, we infer that
  \[
   \Fcal^{-1} (\Fcal \vec g_\rho^{2m})_j(\vec x', y)
    = \sum_{k=0}^{m_j} ( (- \Delta_{\R^n})^m + \abs{\lambda})^{(2m-k)/2m} g_{j,k}(\vec x',y).
  \]
  Denoting by $K^\lambda_j := (\bb K^\lambda)_{1,j}$ the components $(1,j)$ of the inverse Fourier transform $\bb K^\lambda$ of the operator-valued function $\Fcal \bb V_\rho^{2m}$ as defined by \eqref{eqn:Def_K_lambda} and $K^\lambda_{j,k} := K^\lambda_j \Pcal_{j,k}$, $j = 1, \ldots, m$ and $k = 0, 1, \ldots, m_j$, we observe the following:
  \[
   \Fcal \vec v^\mathrm{hs}
    = \Fcal \bb V_\rho^{2m} \Fcal \vec g_\rho^{2m},
    \quad \text{then} \quad
   \vec v^\mathrm{hs}
    = \bb K^\lambda \ast \Fcal^{-1} (\Fcal \vec g_\rho^{2m}).
  \]
 With the help of Lemma \ref{lem:Lem-2.6}, we may now derive kernel estimates.

 \begin{proposition}
 \label{prop:kernel_estimates_half-space}
  Let $\Acal(\DD)$ be a parameter-elliptic operator of order $2m$ with angle of ellipticity $\phi_\Acal^\mathrm{ellipt} \in [0,\pi)$.
  Let $\phi > \phi_\Acal^\mathrm{ellipt}$ and assume that the Lopatinskii-Shapiro condition is valid.
  Then, for every multi-index $\vec \alpha \in \N_0^{n+1}$, there are constants $M,c > 0$ such that
   \begin{align*}
    \abs{\DD^{\vec \alpha} \bb K^\lambda(\vec x',y)}
     &\leq M \cdot \abs{\lambda}^{\frac{n-2m+\abs{\vec \alpha}}{2m}} p_{2m,\abs{\vec \alpha}-1}^{n+1} \left( c \abs{\lambda}^{1/(2m)} \left( \abs{\vec x'} + \abs{y} \right) \right),
     \quad
     \vec x' \in \R^n, \, y > 0, \, \lambda \in \Sigma_{\pi-\phi}.
   \end{align*}
 \end{proposition}
 
 \begin{proof}
  We proceed as in the proof of \cite[Propositions 6.5 \& 6.6]{DeHiPr03}.
  The proof heavily relies on theory of complex functions and, in particular, Cauchy's Theorem.
  We use that the identity
   \[
    \Fcal \vec v^\mathrm{hs}(\vec \xi',y,\lambda)
     = \Fcal \bb V_\rho^{2m}(\vec \xi',y,\lambda) \Fcal \vec g_\rho^{2m}(\vec \xi',0),
     \quad
     \vec \xi' \in \R^n, \, y > 0, \, \lambda \in \Sigma_{\pi-\phi}
   \]
  is valid independent of the particular choice of $\rho > 0$, so that later on, we may choose variables $(a,r)$ and then choose $\rho$ depending on $(a,r)$ as $\rho = (\abs{\lambda} + \abs{\vec \xi'}^{2m})^{1/2m} = (\abs{\lambda} + (a^2 + r^2)^m)^{\nicefrac{1}{2m}}$ for $\lambda \in \Sigma_{\pi-\phi}$ and $\vec \xi' \in \R^n$ such that $\abs{\lambda} + \abs{\vec \xi'} > 0$.
 Fix $\vec x' \in \R^n$ and let $\bb Q \in \R^{n \times n}$ be the (uniquely determined) rotation in $\R^n$ which maps the vector $\vec x'$ to a positive multiple of the first standard basis vector $\bb Q \vec x' = (\abs{\vec x'},0, \ldots, 0)^\mathsf{T} =: \abs{\vec x'} \vec e_1$. Using spherical coordinates in $\R^n$, we then write $\bb Q \vec \xi' = (a,r \vec \varphi)^\mathsf{T}$ where $a \in \R$, $r > 0$ and $\vec \varphi \in \S^{n-2} = \{\vec z' \in \R^{n-1}: \, \abs{\vec z'} = 1\}$ (the $n-2$-dimensional sphere), hence, $a^2 + r^2 = \abs{\vec \xi'}^2$. 
 Since, by convention, $\sigma = \tfrac{\lambda}{\rho^{2m}}$ and $\vec b = \tfrac{\vec \xi'}{\rho}$, the transformation rule will help us to estimate the kernels and their derivatives:
 Using orthogonality of the rotation matrix $\bb Q$ and the transformation rule, the integral kernel operator $\bb K^\lambda = \Fcal^{-1} (\bb V_\rho^{2m})$ can be written as
  \begin{align*}
   &\bb K^\lambda(\vec x',\lambda)
    \\
    &= \frac{1}{(2\pi)^n} \int_{\R^n} \ee^{\ii \vec x' \cdot \vec \xi'} \Fcal \bb V_\rho^{2m}(\vec \xi',y,\lambda) \, \dd \vec \xi'
    \\
    &= \frac{1}{(2\pi)^n} \int_{\R^n} \ee^{\ii \bb Q \vec x' \cdot \bb Q \vec \xi'} \ee^{\ii \rho \bb A_0(\vec b,\sigma) y} \bb M(\vec b,\sigma) \frac{1}{\rho^{2m}} \, \dd \vec \xi'
    \\
    &= \frac{1}{(2\pi)^n} \int_{\S^{n-2}} \int_0^\infty r^{n-2} \int_{-\infty}^\infty \ee^{\ii \abs{\vec x'} a} \ee^{\ii \rho \bb A_0(\bb Q^\mathsf{T} (\tfrac{a}{\rho}, \tfrac{r}{\rho} \vec \varphi), \tfrac{\lambda}{\rho^{2m}}) y} \bb M(\bb Q^\mathsf{T}(\tfrac{a}{\rho}, \tfrac{r}{\rho} \vec \varphi), \tfrac{\lambda}{\rho^{2m}}) \frac{1}{\rho^{2m}} \, \dd a \, \dd r \, \dd \sigma(\vec \varphi)
    \\
    &= \frac{1}{(2\pi)^n} \int_{\S^{n-2}} \int_0^\infty r^{n-2} \int_{-\infty}^\infty \ee^{\ii \abs{\vec x'} a} \frac{1}{(\abs{\lambda} + (a^2 + r^2)^m)^2} \bb F_1(\lambda,a,r,\vec \varphi, y) \bb F_2(\lambda,a,r,\vec \varphi) \, \dd a \, \dd r \, \dd \sigma(\vec \varphi)
  \end{align*}  
  for operator-valued functions $\bb F_1$ and $\bb F_2$ which are defined by
  \begin{align*}
    \bb F_1(\lambda,a,r,\vec \varphi, y)
     &= \ee^{\ii (\abs{\lambda} + (a^2 + r^2)^m)^{1/2m} \bb A_0((\abs{\lambda} + (a^2 + r^2)^m)^{-\nicefrac{1}{2m}} \bb Q^\mathsf{T}(a,r \vec \varphi), \lambda (\abs{\lambda} + (a^2 + r^2)^m)^{-1}) y} \in \B(E^{2m}),
     \\
    \bb F_2(\lambda,a,r,\vec \varphi)
     &= \bb M((\abs{\lambda} + (a^2 + r^2)^m)^{\nicefrac{1}{2m}} \bb Q^\mathsf{T}(a, r \vec \varphi), \lambda (\abs{\lambda} + (a^2 + r^2)^m)^{-\nicefrac{1}{2m}})
     \in \B(E^m; E^{2m})
   \end{align*}
  for $\lambda \in \overline{\Sigma_{\pi - \phi}}$, $a \in \R$, $r \geq 0$, $\varphi \in \S^{n-1}$ and $y > 0$.
  The inner integrand can be interpreted as a function depending holomorphically on $a \in \C$ as the maps $\bb M(\vec b,\sigma)$ depend holomorphically on $(\vec b, \sigma)$, and $\bb M(\vec b, \sigma)$ maps into the stable part (corresponding to $S_-(\vec b, \sigma)$) of $\big( \ee^{\ii \rho \bb A_0(\vec b, \sigma) y} \big)_{y \geq 0}$.
  Thus, we may and will replace the inner real integral $\int_{-\infty}^\infty \ldots \, \dd a$ by a complex contour integral $\int_{\Gamma_\varepsilon} \ldots \, \dd \omega$ for a complex path $\Gamma_\varepsilon$ which is parametrised by
   \[
    \gamma_\varepsilon(s)
     = s + \ii \varepsilon \left( \abs{\lambda} + (r^2 + s^2)^m \right)^{1/2m},
     \quad
     s \in (-\infty,\infty)
   \]
  for some (sufficiently small) parameter $\varepsilon > 0$.
  \begin{figure}
	\centering
	\includegraphics[scale = 0.75]{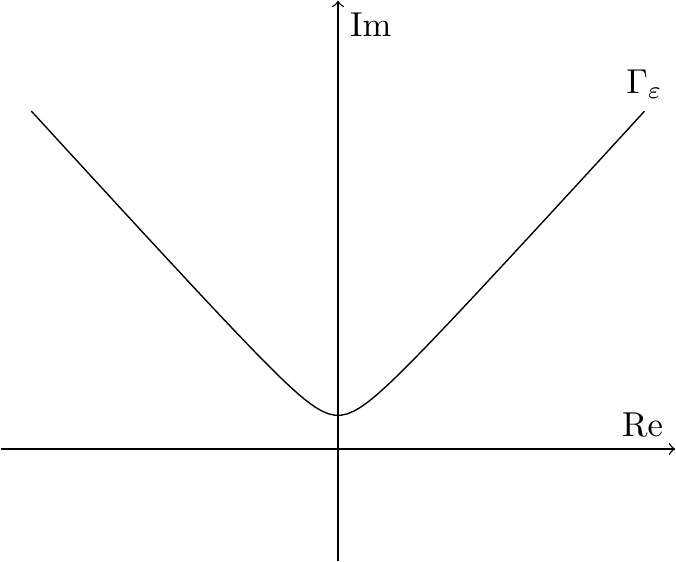}
	\caption{The path $\Gamma_\varepsilon$.}
	\label{fig:diagram_path}
 \end{figure}
  Hence,
  \begin{align*}
   \bb K^\lambda(\vec x',\lambda)
    &= \frac{1}{(2\pi)^n} \int_{\S^{n-2}} \int_0^\infty r^{n-2} \int_{\Gamma_\varepsilon} \ee^{\ii \abs{\vec x'} a} \frac{1}{(\abs{\lambda} + (a^2 + r^2)^m)^2}
     \\
    &\qquad \cdot \bb F_1(\lambda,a,r,\vec \varphi, y) \bb F_2(\lambda,a,r,\vec \varphi) \, \dd a \, \dd r \, \dd \sigma(\vec \varphi).
  \end{align*}
 We proceed just as in the proof of \cite[Proposition 6.5]{DeHiPr03}.
 First, we find that
  \[
   \abs{\ee^{\ii \abs{\vec x'} \gamma_\varepsilon(s)}}
    = \ee^{- \abs{\vec x'} \Im (\gamma_\varepsilon(s))}
    = \ee^{-\varepsilon \abs{\vec x'} (\abs{\lambda} + (r^2 + s^2)^m)^{1/2m}}.
  \]
 Moreover, we may estimate the function $\bb F_1$ as follows:
  \[
   \norm{\bb F_1(\lambda, \gamma_\varepsilon(s), r, \vec \varphi, y)}_{\B(E^n)}
    \leq M \ee^{- c_1 y}.
  \]
 
 In fact, from $\sup \{\Re \mu: \mu \in S_-(\vec b, \sigma)\} < - c_1 < 0$ it follows that for every $(\vec b, \sigma) \in \S^{n-1} \times \Sigma_{\pi - \phi}$ with $\abs{\sigma} \leq 1$, there is a constant $M_{\vec b, \sigma} \geq 1$ such that
  \[
   \norm{\ee^{\ii \bb A_0(\vec b, \sigma) y} P_-(\vec b,\sigma)}_{\B(E^n)}
    \leq M_{\vec b, \sigma} \ee^{- c_1 y}
    \quad
    \text{for all } y \geq 0.
  \]
 Here, $\bb A_0(\vec b, \sigma) \in \B(E^{2m})$ is a continuous linear operator and depends continuously on the parameters $(\vec b, \sigma) \in \S^{n-1} \times \Sigma_{\pi - \phi}$.
 Therefore, each (uniformly continuous) semigroup $\big( \ee^{\ii \bb A_0(\vec b, \sigma)y} \big)_{y \geq 0} \subseteq \B(E^n)$ satisfies the spectral-determined growth bound condition and, hence, by \eqref{eqn:Spectral_Gap_S+} its restriction to the image under the spectral projection $P_-(\vec b, \sigma)$ on the stable part has a growth bound which is strictly less than $- c_1$.
 Moreover, as $\bb A_0(\vec b, \sigma)$ depends holomorphically on the parameters $(\vec b, \sigma)$, so do the (uniformly continuous) semigroups $\big( \ee^{\ii A_0(\vec b, \sigma) y} \big)_{y \geq 0}$.
 Thus, the constant $M_{\vec b, \sigma}$ can be chosen uniformly as $M \geq 1$ for all $(\vec b, \sigma) \in \S^{n-1} \times \Sigma_{\pi - \phi}$ with $\abs{\sigma} \leq 1$.
 This establishes the estimate on $\bb F_1$.
 \newline
 For the function $\bb F_2$, observe that
  \[
   \norm{\bb F_2(\lambda, a, r, \varphi)}_{\B(E^n)}
    = \norm{ \bb M( (\abs{\lambda} + (a^2 + r^2)^m)^{-1/2m} \bb Q^\mathsf{T}(a,r \vec \varphi), \lambda ( \abs{\lambda} + (a^2 + r^2)^m)^{-1})}_{\B(E^n)},
  \]
 where $(\abs{\lambda} + (a^2 + r^2)^m)^{-1/2m} \bb Q^\mathsf{T}(a,r \vec \varphi) \in \R^n$ has norm less or equal $1$ and $\lambda ( \abs{\lambda} + (a^2 + r^2)^m)^{-1} \in \Sigma_{\pi-\phi}$ has modulus less or equal $1$.
 Since the map $(\vec b, \sigma) \mapsto \bb M(\vec b, \sigma) \in \B(E^n)$ is holomorphic,
  \[
   \sup_{(\lambda,a,r,\vec \varphi)} \norm{\bb F_2(\lambda, a, r, \vec \varphi)}_{\B(E^n)}
    < \infty,
  \]
 where we take the supremum over all admissible values for $(\lambda,a,r,\vec \varphi)$.
 For the calculations following next, let us remark that the functions
  \[
   \rho(\lambda, s, r)
    = \big( \lambda + (s^2 + r^2)^m \big)^{1/2m}
    \quad \text{and} \quad
   \tilde \rho(\lambda, s, r)
    = \big( \lambda^{1/m} + s^2 + r^2 \big)^{1/2}
  \]
 on $\R_+ \times \R \times \R$
 (the latter having been the choice employed in \cite{DeHiPr03}) are comparable in the sense that there is a constant $c > 0$ such that
  \[
   \frac{1}{c} \tilde \rho(\lambda, r, s)
    \leq \rho(\lambda, r, s)
    \leq c \tilde \rho(\lambda, r, s).
  \]
 This can be seen, e.g., by employing the binomial formula to estimate
  \[
   \tilde \rho(\lambda, s, r)^{2m}
    = \sum_{k=0}^{m} \binom{m}{k} \abs{\lambda}^{k/m} (s^2 + r^2)^{m-k}
    \leq C (\abs{\lambda} + (s^2 + r^2)^m)
    = \rho(\lambda, s, r)^{2m}
  \]
 and the trivial estimate
  \[
   \rho(\lambda, s, r)^2
    = \big( \abs{\lambda} + (s^2 + r^2)^m \big)^{1/m}
    \leq \big( 2 \abs{\lambda} \big)^{1/m} + \big( 2 (s^2 + r^2)^m \big)^{1/m}
    = 2^{1/m} \tilde \rho(\lambda, s, r)^2.
  \]
 We may, therefore, readily estimate $\DD^{\vec \alpha} \bb K^\lambda(\vec x', y)$ for every multi-index $\vec \alpha = (\vec \alpha', \alpha_{n+1}) \in \N_0^{n+1}$:
  \begin{align*}
   &\norm{\DD^{\vec \alpha} \bb K^\lambda(\vec x', y)}_{\B(E^{2m})}
    \\
    &= \norm{ \frac{1}{(2 \pi)^n} \int_{\R^n} (\vec \xi')^{\vec \alpha'} \ee^{\ii \vec x' \cdot \vec \xi'} \DD_y^{\alpha_{n+1}} (\Fcal \bb V_\rho^{2m}(\vec \xi', y, \lambda)) \, \dd \vec \xi'}_{\B(E^n)}
    \\
    &= \norm{ \frac{1}{(2\pi)^n} \int_{\S^{n-2}} \int_0^\infty r^{n-2} \int_{\Gamma_\varepsilon} \ee^{\ii \abs{\vec x'} a} (\rho \vec b)^{\vec \alpha'} \frac{1}{\abs{\lambda} + (s^2 + r^2)^m} (\rho \bb A_0(\vec b, \sigma))^{\alpha_{n+1}} \bb F_1 \bb F_2 \, \dd a \, \dd r \, \dd \S^{n-2}(\vec \varphi)}
    \\
    &\leq \frac{C}{(2\pi)^n} \int_{\S^{n-2}} \, \dd \S^{n-2}(\vec \varphi) \int_0^\infty r^{n-2} \int_{\Gamma_\varepsilon} \frac{1}{(\abs{\lambda} + (s^2 + r^2)^m)^{\frac{2m - \abs{\vec \alpha}}{2m}}} \ee^{- (\abs{\lambda} + (s^2 + r^2)^m)^{1/2m} \cdot (\varepsilon \abs{\vec x'} + c_1 y)} |\dd a| \, \dd r
    \\
    &\leq C_{\varepsilon_0} \int_0^\infty r^{n-2} \int_0^\infty \ee^{- ( \abs{\lambda} + (s^2 + r^2)^m)^{1/2m} \cdot (\varepsilon \abs{\vec x'} + c_1 y)} \cdot \frac{1}{(\abs{\lambda} + (s^2 + r^2)^m)^{\frac{2m - \abs{\vec \alpha}}{2m}}} \, \dd s \, \dd r
    \\
    &\leq C_{\varepsilon_0} \int_0^\infty r^{n-2} \int_r^\infty \ee^{- c ( \abs{\lambda}^{1/m} + (\tau - r)^2 + r^2)^{1/2} \cdot (\varepsilon \abs{\vec x'} + c_1 y)} \cdot \frac{1}{(\abs{\lambda}^{1/2m} + \tau)^{2m - \abs{\vec \alpha}}} \, \dd \tau \, \dd r
    \\
    &\leq C_{\varepsilon_0} \int_0^\infty r^{n-2} \int_r^\infty \ee^{- c \varepsilon \cdot (\abs{\lambda}^{1/2m} + \tau) \cdot \abs{\vec x'}} \ee^{- c c_1 (\abs{\lambda}^{1/2m} + \tau) y} \cdot  \frac{1}{(\abs{\lambda}^{1/2m} + \tau)^{2m - \abs{\vec \alpha}}} \, \dd \tau \, \dd r
    \\
    &\leq C_{\varepsilon_0} \int_0^\infty r^{n-2} \int_r^\infty \ee^{- c (\varepsilon \abs{\vec x'} + c_1 y) \cdot (\abs{\lambda}^{1/2m} + \tau)}  \frac{1}{(\abs{\lambda}^{1/2m} + \tau)^{2m - \abs{\vec \alpha}}} \, \dd \tau \, \dd r
    \\
    &= C_{\varepsilon_0} \abs{\lambda}^{\frac{1}{2m} - \frac{2m - \abs{\vec \alpha}}{2m} + \frac{n-1}{2m}} \int_0^\infty s^{n-2} \int_s^\infty \ee^{- c ( \varepsilon \abs{\vec x'} + c_1 y) \abs{\lambda}^{1/2m} (1 + \tilde \tau)} \frac{1}{(1 + \tilde \tau)^{2m - \abs{\vec \alpha}}} \, \dd \tilde \tau \, \dd s
    \\
    &= C_{\varepsilon_0} \abs{\lambda}^{\frac{n}{2m} - 1 + \frac{\abs{\alpha}}{2m}} \frac{1}{n-1} \int_0^\infty \frac{\tilde \tau^{n-1}}{(1 + \tilde \tau)^{2m - \abs{\vec \alpha}}} \ee^{- c \abs{\lambda}^{1/2m} (\varepsilon \abs{\vec x'} + c_1 y)(1 + \tilde \tau)} \, \dd \tilde \tau
    \\
    &= \frac{C_{\varepsilon_0}}{n-1} \abs{\lambda}^{\frac{n}{m} - 1 + \frac{\abs{\alpha}}{2m}} p_{2m,\abs{\vec \alpha}-1}^{n+1}(c \abs{\lambda}^{1/2m} (\varepsilon \abs{\vec x'} + c_1 y))
    \\
    &\leq C_{\varepsilon_0}' \abs{\lambda}^{\frac{n-2m+\abs{\vec \alpha}}{2m}} p_{2m,\abs{\vec \alpha}-1}^{n+1}(c_{\varepsilon_0} \abs{\lambda}^{1/2m} (\abs{\vec x'} + y))
  \end{align*}
  for some constants $C_{\varepsilon_0}$, $C_{\varepsilon_0}'$ and $c_{\varepsilon_0} > 0$, for sufficiently small $\varepsilon_0 > 0$, and all $\varepsilon \in (0, \varepsilon_0]$.
 Thus, we have established the kernel estimate
  \begin{align*}
   \abs{\DD^{\vec \alpha} \bb K^\lambda(\vec x',y)}
    &\leq M \abs{\lambda}^{\frac{n-2m+\abs{\vec \alpha}}{2m}} p_{2m,\abs{\vec \alpha}-1}^{n+1}(c \abs{\lambda}^{\nicefrac{1}{2}m}(\abs{\vec x'} + \abs{y}))
    \quad
    \text{for }
    \vec x' \in \R^n, \, y > 0, \lambda \in \Sigma_{\pi-\phi}.
  \end{align*}
 \end{proof}

\begin{proposition}
 Let $\Acal(\DD)$ be a parameter-elliptic operator of order $2m$ with angle of ellipticity $\phi_\Acal^\mathrm{ellipt} \in [0, \pi)$.
 Let $\phi > \phi_\Acal^\mathrm{ellipt}$ and assume that the Lopatinskii--Shapiro condition is satisfied.
 Then, for every multi-index $\vec \alpha \in \N_0^{n+1}$, there are constants $M, c > 0$ such that
   \[
    \abs{\DD^{\vec \alpha} \bb K^{\mathrm{corr},\lambda}_{j,k}(\vec x',y,\tilde y)}
     \leq M \abs{\lambda}^{\frac{n+1-2m+\abs{\vec \alpha}}{2m}} p_{2m,\abs{\vec \alpha}}^{n+1}\left(c \abs{\lambda}^{1/(2m)} (\abs{\vec x'} + y + \tilde y) \right)
 \]
for all $\vec x' \in \R^n, \, y, \tilde y > 0$ and $\lambda \in \Sigma_{\pi-\phi}$, where
 \[
  \bb K^{\mathrm{corr},\lambda}_{j,k}(\vec x', y, \tilde y) = \frac{1}{(2\pi)^n} \int_{\R^n} \ee^{\ii \vec x' \cdot \vec \xi'} \ee^{\ii \rho \bb A_0(\vec b,\sigma)y} \bb M(\vec b,\sigma) \vec h^\mathrm{corr}_{j,k}(\vec \xi',\tilde y) \, \dd \vec \xi'.
 \]
\end{proposition}

\begin{proof}
Consider the inverse Fourier transform of the function \[ \Fcal w^\mathrm{corr}(\vec \xi', y, \lambda) = - \big[ \ee^{\ii \rho \bb A_0(\vec b,\sigma)y} \bb M(\vec b,\sigma) (\Fcal \vec g_\rho^\mathrm{corr})(\vec \xi',0) \big]_1 \] corresponding to the correction term for $\Bcal_j(\DD) \Fcal w^\mathrm{fs}|_{\partial \R_+^{n+1}}$, hence,
  \begin{align*}
   \bb K^{\mathrm{corr},\lambda}_{j,k}(\vec x', y, \tilde y)
    &= - \frac{1}{(2\pi)^n} \int_{\R^n} \ee^{\ii \vec x' \cdot \vec \xi} \ee^{\ii \rho \bb A_0(\vec b,\sigma) y} \bb M(\vec b,\sigma) \bb h_{j,k}^\mathrm{corr}(\vec \xi',\tilde y) \, \dd \vec \xi'
  \end{align*}
 where $\vec h_j^\mathrm{corr} = (0, \ldots, 0, h_j^\mathrm{corr}, 0, \ldots, 0)^\mathsf{T} = (0, \ldots, 0, \sum_{k=0}^{m_j} h_{j,k}^\mathrm{corr}, 0, \ldots, 0)^\mathsf{T} = \sum_{k=0}^{m_j} \vec h_{j,k}^\mathrm{corr}$.
 We recall that the family $\big( \ee^{\ii \rho \bb A_0(\vec b, \sigma) y} \bb M(\vec b, \sigma) \vec h_{j,k}(\vec \xi', 0) \big)_{y \geq 0}$ is the norm-continuous semigroup corresponding to the (parameter-dependent) initial value problem
  \begin{align*}
   \frac{\partial}{\partial y} \Fcal \vec v^\mathrm{corr}_{j,k}(\vec \xi', y, \lambda)
    &= \ii \rho \bb A_0(\vec b, \sigma) \Fcal \vec v^\mathrm{corr}_{j,k},
    \quad
    y > 0,
    \\
   B_{\tilde j}^0(\vec b) \Fcal \vec v^\mathrm{corr}_{j,k}(\vec \xi', 0, \lambda)
    &= \delta_{j,\tilde j} \sum_{\tilde k=0}^{m_{\tilde j}} \delta_{k,\tilde k} \Pcal_{\tilde j, \tilde k} \frac{\Fcal g_{\tilde j}(\vec \xi')}{\rho^{\tilde k}},
    \quad
    \tilde j = 1, \ldots, m.
  \end{align*}
 Since the function $\Fcal \vec v^\mathrm{corr}$ is the solution to the boundary value problem
  \begin{align*}
   \DD_y \Fcal \vec v^\mathrm{corr}(\vec \xi', y, \lambda)
    &= \rho \bb A_0(\vec b, \sigma) \Fcal \vec v^\mathrm{corr}(\vec \xi', y, \lambda),
    &&\vec \xi' \in \R^n, \, y > 0,
    \\
   B_j^0(\vec b) \Fcal \vec v^\mathrm{corr}(\vec \xi',0)
    &= \sum_{k=0}^{m_j} \Pcal_{j,k} \frac{\Fcal g_j(\vec \xi',0)}{\rho^k},
    &&j = 1, \ldots, m, \, \vec \xi' \in \R^n,
  \end{align*}
 we deduce that
  \begin{align*}
   \Fcal w^\mathrm{corr}(\vec \xi', y, \lambda)
    &= \big[ \ee^{\ii \rho \bb A_0(\vec b, \sigma) y} (\Fcal \vec v^\mathrm{corr})(\vec \xi', 0) \big]_1
    = \big[ \ee^{\ii \rho \bb A_0(\vec b, \sigma) y} M(\vec b, \sigma) \Fcal \vec g_\rho^\mathrm{corr}(\vec \xi', 0) \big]_1
    \\
    &= \big[ \ee^{\ii \rho \bb A_0(\vec b, \sigma) y} \bb M(\vec b, \sigma) \big( \sum_{k=0}^{m_j} \Pcal_{j,k} \frac{\Fcal g_j^\mathrm{corr}(\vec \xi',0)}{\rho^k} \big)_{j} \big]_1.
  \end{align*}
 By Cauchy's theorem and homogeneity considerations
  \begin{align*}
   h_{j,k}^\mathrm{corr} (\vec \xi', \tilde y)
    &= \rho^{k+1-2m} \int_{\R} \Bcal_{j,k}(\vec b,a) (\sigma + \Acal(\vec b,a))^{-1} \ee^{-\ii \tilde y \rho a} \, \dd a
    \\
    &= \rho^{k+1-2m} \int_{\Gamma^-} \Bcal_{j,k}(\vec b,a) (\sigma + \Acal(\vec b,a))^{-1} \ee^{-\ii \tilde y \rho a} \, \dd a,
  \end{align*}
 where $\Gamma^-$ is a closed curve in the open lower complex plane which encircles $- \ii P_+(\vec b,\sigma)$, uniformly for all $(\vec b, \sigma)$.
  \begin{figure}
	\centering
	\includegraphics[scale = 1]{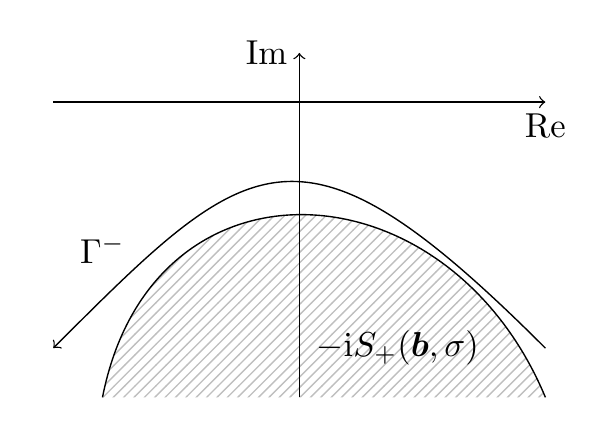}
	\caption{A path $\Gamma^-$ surrounding $- \ii P_+(\vec b, a)$.}
	\label{fig:diagram_path-2}
 \end{figure}
 We then obtain, after using a rotation as above and shifting the path of integration as in \cite{DeHiPr03}, that
  \begin{equation}
   \abs{h_{j,k}^\mathrm{corr} (\vec \xi',\tilde y)}
    \leq C \left(\abs{\lambda} + r^{2m} + s^{2m} \right)^{\nicefrac{k+1-2m}{2m}} \ee^{- c \tilde y (\abs{\lambda} + r^{2m} + s^{2m})^{\nicefrac{1}{2m}}},
    \quad
    \vec \xi' \in \R^n, \, y > 0 
  \end{equation}
 for sufficiently small values of $\abs{\operatorname{arg} (r^2 + a^2)}$.
 For any multi-index $\vec \alpha \in \N_0^{n+1}$, we, thus, obtain the estimate
  \begin{align*}
   \abs{\DD^{\vec \alpha} K^\mathrm{corr,\lambda}_{j,k}(\vec x',y, \tilde y)}
    &\leq \abs{\DD^{\vec \alpha} K^\mathrm{corr,\lambda}_{j,k}(\vec x',y, \tilde y)}
    \\
    &\leq M \abs{\lambda}^{\frac{n-2m+1+\abs{\vec \alpha}}{2m}} \int_0^\infty \frac{s^{n-1}}{(1+s)^{2m-1-\abs{\vec \alpha}}} \ee^{-c (s+1) \abs{\lambda}^{1/(2m)} (\abs{\vec x'} + y + \tilde y)} \, \dd s,
  \end{align*}
 cf.\ the calculations on \cite[p.\ 80]{DeHiPr03}.
\end{proof}

 \subsection{Solution operators}
 Using the kernel estimates from the last subsection, we can now derive $\LL_p$-estimates for $w^\mathrm{fs}$, $w^\mathrm{hs}$ and $w^\mathrm{corr}$.
 For this purpose, let us introduce the $\LL_p$-realisation for the boundary value problem \eqref{eqn:Elliptic_PDE_Laplace-Space}--\eqref{eqn:Elliptic_PDE_Laplace-Space_Boundary} as the closure $A \& B = \overline{(A \& B)_\mathrm{min}}$ of the \emph{minimal operator} $(A \& B)_\mathrm{min}$ which is defined by
  \begin{align*}
   (A \& B)_\mathrm{min}:
    \dom((A \& B)_\mathrm{min})
    &= \WW_p^{2m}(\R_+^{n+1};E)
    \subseteq \LL_p(\R_+^{n+1}; E)
    \\
    &\rightarrow \LL_p(\R_+^{n+1}; E) \times \prod_{j=1}^m \left(\sum_{k=0}^{m_j} \Pcal_{j,k} \WW_p^{2m-k}(\R_+^{n+1};E) \right),
    \\
   (A \& B)_\mathrm{min} u
    &= \left( \begin{array}{c} \Acal(\DD) \\ \Bcal_1(\DD) \\ \vdots \\ \Bcal_m (\DD) \end{array} \right) u.
  \end{align*}
 \begin{remark}
  Note that the differential operators $\Bcal_j(\DD)$ corresponding to the boundary symbols act on the function $u: \R_+^{n+1} \rightarrow E$ and are evaluated on the half-space $\R_+^{n+1}$, and not only on its boundary trace $\partial \R_+^{n+1}$.
  \newline
  Also, it may not be immediately clear that the operator $(A \& B)_\mathrm{min}$ is closable, thus let us comment on that.
  We may introduce the higher-order reflexion $\Ecal u(\vec x',y) := \sum_{i=0}^{2m} \alpha_i u(\vec x', - \frac{y}{i+1})$ for $\vec x' \in \R^n$ and $y < 0$, which for suitable chosen scalar coefficients $\alpha_i \in \R$ defines a continuous extension of functions $u \in \WW_p^{2m}(\R_+^{n+1};E)$ to functions $\Ecal u \in \WW_p^{2m}(\R^{n+1};E)$.
  Then, given any sequence $(u_n)_{n \geq 1} \subseteq \dom((A \& B)_\mathrm{min})$ such that $u_n \rightarrow 0$ and $\Acal(\DD) u_n \rightarrow v$ in $\LL_p(\R_+^{n+1};E)$ for some function $v \in \LL_p(\R_+^{n+1};E)$, from the definition of the extension operator it follows that  $\Ecal u_n \rightarrow 0$ and $\Acal(\DD) u_n \rightarrow \tilde v$ in $\LL_p(\R^{n+1};E)$ for the extension $\tilde v = \Ecal v \in \LL_p(\R^{n+1};E)$ of $v$.
  We claim that $\tilde v = 0$, which will imply that $v = v|_{\R_+^{n+1}} = 0$ and, hence, $(A \& B)_{\mathrm{min}}$ is closable.
  To this end, we use that $\Acal(\vec \xi)$ is parameter-elliptic, so that for every $\vec \alpha \in \N_0^{n+1}$ of length $\abs{\vec \alpha} \leq 2m$, the function $\vec \xi^{\vec \alpha} (\Acal(\vec \xi) + 1)^{-1}$ is an $\Rcal$-bounded Fourier multiplier.
  From $(\Acal(\DD) + 1) \Ecal u_n = \Acal(\DD) \Ecal u_n + \Ecal u_n = \Ecal (\Acal(\DD) u_n + u_n) \rightarrow \tilde v$ it, thus, follows that also $\DD^{\vec \alpha} \Ecal u_n$ converges in $\LL_p(\R^{n+1};E)$ for each multi-index $\vec \alpha \in \N_0^{n+1}$ of length $\abs{\vec \alpha} \leq 2m$.
  Since $\WW_p^{2m}(\R^{n+1};E)$ is a Banach space and $\Ecal u_n \rightarrow 0$, however, this can only be the case if $\DD^{\vec \alpha} (\Ecal u_n) \rightarrow \DD^{\vec \alpha} (\Ecal 0) = 0$, which in turn implies that $\tilde v = 0$.
  This shows that the operator $(A \& B)_\mathrm{min}$ is indeed closable.
 \end{remark}
 For each parameter $\lambda \in \Sigma_{\pi-\phi}$ and functions $g_j \in \sum_{k=0}^{m_j} \Pcal_{j,k} \WW_p^{2m-k}(\R_+^{n+1};E)$ (for the moment defined on $\R_+^{n+1}$, \emph{not} only on $\partial \R_+^{n+1}$) we consider the system of differential equations
  \begin{equation}
   \lambda \bb J v + (A \& B) v
    = \left( \begin{array}{c} 0 \\ \vec g \end{array} \right),
    \label{eqn:Abstract_BVP}
  \end{equation}
 where we write $\vec g = (g_1, \ldots, g_m)^\mathsf{T}$ and $\bb J v = (v, 0, \ldots, 0)^\mathsf{T} \in E^{1 + m}$ for $v \in E$, and get the following result in which we consider two cases: On the one hand, the situation as above or in \cite[Proposition 6.8]{DeHiPr03}, where $\vec g$ is given from the start, and on the other hand, the more subtle case where $\vec g$ is only given on $\partial \R_+^{n+1}$, but we use a particular extension to $\R_+^{n+1}$ which allows us to assume that $\partial_y  g_{j,k} = - L_\lambda^{1/2m} g_{j,k}$ in $\R_+^{n+1}$, thus relating the partial derivative in $y$-direction with a fractional differential operator acting on the $\vec x'$-components of the function $\vec g$.
  \begin{remark}
  \label{rem:generator}
   We will use that the pseudo-differential operator $- L_\lambda^{1/2m} = - (\abs{\lambda} + (-\Delta)^m)^{1/2m}$, i.e.\ the linear operator associated to the Fourier symbol $- (\abs{\lambda} + \abs{\vec \xi'}^{2m})^{1/2m}$, for fixed $\lambda \in \Sigma_{\pi-\phi}$ and $\vec \xi' \in \R^n$, generates a bounded $\CC_0$-semigroup on the Bochner--Lebesgue space $\LL_p(\R^n;E)$ for any given $p \in [1,\infty)$.
   In fact, for every $\mu > 0$ and $\lambda \in \Sigma_{\pi-\phi}$, the resolvent problem
    \[
     (\mu + L_\lambda^{1/2m}) u
      = f
    \]
   can be solved in Fourier space as
    \[
     (\mu + (\abs{\lambda} + \abs{\vec \xi'}^{2m})^{1/2m}) \hat u
      = \hat f
      \quad \Leftrightarrow \quad
     \hat u
      = \frac{1}{\mu + (\abs{\lambda} + \abs{\vec \xi'}^{2m})^{1/2m}} \hat f
      =: m_\mu(\vec \xi') \hat f,
    \]
   where $m_\mu(\vec \xi')$ is scalar, and for boundedness of the resolvent, it, thus, suffices to prove boundedness of $m_\mu(\vec \xi')$ and $(\vec \xi')^{\vec \alpha} D_{\vec \xi'}^{\vec \alpha} m_\mu(\vec \xi')$ for all $\vec \alpha \in \N_0^n$ of length $\abs{\vec \alpha} \leq \lfloor \frac{n}{2} \rfloor$, which can easily be checked, and in fact gives $\norm{(\mu (\mu + L_\lambda^{1/2m})^{-1})^k} \leq C$ for all $\mu > 0$. Hence, by the Hille--Yosida Theorem $- L_\lambda^{1/2m}$ generates a bounded $C_0$-semigroup on $\LL_p(\R^n;E)$, for every $\lambda \in \Sigma_{\pi-\phi}$.
   For details, we refer to Proposition \ref{prop:aux} in the appendix.
  \end{remark}
  
   We denote by $L_\lambda$ the operator $\abs{\lambda} + (- \Delta_{\R^n})^m$ on $\LL^p(\R_+^{n+1};E)$, where $\Delta_{\R^n}$ is the closure of the minimal realisation of the ($E$-valued) Laplace operator on $\R^n$, and extend the functions $g_{j,k}$ on the boundary $\partial \R_+^{n+1} \cong \R^n$ with use of the $C_0$-semigroup $(\ee^{- L_\lambda^{1/2m} y})_{y \geq 0}$ on $\LL_p(\R^n;E)$, which is generated by $- L_\lambda^{1/2m}$ to functions
    \begin{equation}
     (\Ecal_\lambda g_{j,k})(\vec x', y)
      = \ee^{- L_\lambda^{1/2m} y} g_{j,k}(\vec x', y)
      \quad
      \text{for } \vec x' \in \R^n, \, y \geq 0.
      \label{eqn:Ecal_lambda}
    \end{equation}
   This defines functions $\Ecal_\lambda g_j = \sum_{k=0}^{m_j} \Ecal_\lambda g_{j,k} \in \sum_{k=0}^{m_j} \Pcal_{j,k} \WW_p^{2m-k}(\R_+^{n+1};E)$.
   This can be seen as follows:
   By assumption $\Pcal_{j,k'} (\ran (\Pcal_{j,k})) = \{0\}$ for $k \neq k'$, and then by continuity of the projection operators also $\Pcal_{j,k'} (\overline{ \ran (\Pcal_{j,k})}) = \{0\}$.
   Moreover, both the Fourier-transform and its inverse leave the spaces $\overline{ \ran (\Pcal_{j,k}) }$ invariant, e.g.\ $\ran (\Fcal v) \subseteq \overline{\ran (\Pcal_{j,k})}$ whenever $\ran (v) \subseteq \overline{\ran (\Pcal_{j,k})}$ and the same is true for the scalar Fourier-symbol $(\abs{\lambda} + \abs{\vec \xi}^{2m})^{1/2m}$ of the operator $L_\lambda^{1/2m}$.
   Therefore, all spaces $\overline{ \ran (\Pcal_{j,k}) }$ are $L_\lambda^{1/2m}$-invariant, i.e.\ $L_\lambda^{1/2m} ( \dom(L_\lambda^{1/2m}) \cap \overline{ \ran (\Pcal_{j,k}) } ) \subseteq \overline{ \ran (\Pcal_{j,k}) }$, and, then, these spaces are $(\ee^{- y L_\lambda^{1/2m}})_{y \geq 0}$-invariant as well.

  \begin{proposition}
  \label{prop:Estimates_Solution_Operator}
  Let $p \in (1, \infty)$ and $\lambda \in \Sigma_{\pi-\phi}$ be given.
   \begin{enumerate}
    \item\label{pt:a}
    Let $g_j = \sum_{k=0}^{m_j} g_{j,k} \in \sum_{k=0}^{m_j} \WW_p^{2m-k}(\R^n;\ran(\Pcal_{j,k}))$ for $j = 1, \ldots, m$ be given.
   \newline
   Then there is a unique solution $v^\mathrm{hs} \in \WW_p^{2m-1}(\R_+^{n+1};E) \cap \dom(A \& B)$ of the initial-boundary value problem \eqref{eqn:Abstract_BVP} for $\Ecal_\lambda g_{j,k}$ as in \eqref{eqn:Ecal_lambda} taking the role of $g_{j,k}: \R_+^{n+1} \rightarrow E$.
   It is given by $v^\mathrm{hs} = \sum_{j=1}^m \sum_{k=0}^{m_j} v^\mathrm{hs}_{j,k}$ for functions
    \[
     v^\mathrm{hs}_{j,k}(\vec x',y)
      = T^\lambda_{j,k} h^\lambda_{j,k}(\vec x',y)
    \]
   where
    \begin{align*}
     (T^\lambda_{j,k} h^\lambda_{j,k})(\vec x',y)
      &= \int_0^\infty \int_{\R^n} T_{j,k}(\lambda, \vec x' - \vec {\tilde x}, y + \tilde y) h^\lambda_{j,k}(\vec {\tilde x}', \tilde y) \, \dd \vec {\tilde x}' \, \dd \tilde y
       \\
      &= \int_0^\infty \int_{\R^n} T_{j,k}(\lambda, \vec x' - \vec {\tilde x}, y + \tilde y) L_\lambda^{\frac{2m-k}{2m}} \Ecal_\lambda g_{j,k}(\vec {\tilde x}', \tilde y) \, \dd \vec {\tilde x}' \, \dd \tilde y
   \end{align*}
   and
    \begin{align*}
     h^\lambda_{j,k}(\vec x', y)
      &= L_\lambda^{\frac{2m-k}{2m}} \ee^{- L_\lambda^{1/2m} y} \Ecal_\lambda g_{j,k}(\vec x',y)
      \\
      &= \ee^{- L_\lambda^{1/2m} y} L_\lambda^{\frac{2m-k}{2m}} \Ecal_\lambda g_{j,k}(\vec x',y)
      \quad
      \text{for } \vec x' \in \R^n \text{ and } y > 0.
    \end{align*}
   Here, the operator-valued kernel function $T^\lambda_{j,k}(\vec x', y) = T(\lambda, \vec x',y)$ is defined by
    \[
     T^\lambda_{j,k}(\vec x', y)
      = T_{j,k}(\lambda, \vec x', y)
      = \big( \frac{\partial}{\partial y} K_{j,k}^\lambda - K_{j,k}^\lambda L_\lambda^{1/2m} \big)(\vec x', y)
    \]
   Moreover, there are constants $C > 0$ (independent of the parameter $\lambda \in \Sigma_{\pi-\phi}$ and the boundary data $g_j$) such that for every $\lambda \in \Sigma_{\pi-\phi}$, $\vec \alpha \in \N_0^{n+1}$ and $j = 1, \ldots, m$ the following estimates are valid:
    \begin{align*}
     \norm{\abs{\lambda}^{1-\frac{\abs{\vec \alpha}}{2m}} \DD^{\vec \alpha} T^\lambda_{j,k} h^\lambda_{j,k}}_{\LL_p(\R_+^{n+1};E)}
      &\leq C \norm{((- \Delta_{\R^n})^m + \abs{\lambda})^{\frac{2m-k}{2m}} g_{j,k}}_{\LL_p(\R^n;E)},
      &&0 \leq \abs{\vec \alpha} \leq 2m,
      \\
     \norm{\abs{\lambda}^{\frac{k-2m}{2m}} \DD^{\vec \alpha} T^\lambda_{j,k} h^\lambda_{j,k}}_{\LL_p(\R_+^{n+1};E)}
      &\leq C \norm{g_{j,k}}_{\LL_p(\R^n;E)},
      &&0 \leq \abs{\vec \alpha} \leq k \leq m_j,
      \\
     \norm{ \abs{\lambda}^{2-\frac{\abs{\vec \alpha}}{2m}} \DD^{\vec \alpha} \partial_\lambda T^\lambda_{j,k} h^\lambda_{j,k}}_{\LL_p(\R_+^{n+1};E)}
      &\leq C \norm{g_{j,k}}_{\LL_p(\R^n;E)},
      &&0 \leq \abs{\vec \alpha} \leq k \leq m_j.
    \end{align*}
     \item\label{pt:b}
      Given functions $g_{j,k} \in \WW_p^{2m-k}(\R_+^{n+1};\ran(\Pcal_{j,k}))$, we set $h^\lambda_{j,k}(\vec x', y) = L_\lambda^{\frac{2m-k}{2m}} g_{j,k}$.
      Then, the following estimates are valid:
       \begin{align*}
        \norm{\abs{\lambda}^{1-\frac{\abs{\vec \alpha}}{2m}} \DD^{\vec \alpha} S^{\lambda,I}_{j,k} g_{j,k}}_{\LL_p(\R_+^{n+1};E)}
         &\leq C \norm{((-\Delta_{\R^n})^m + \abs{\lambda})^{\frac{2m-k}{2m}} g_{j,k}}_{\LL_p(\R_+^{n+1};E)},
         \\
        \norm{\abs{\lambda}^{1-\frac{\abs{\vec \alpha}}{2m}} \DD^{\vec \alpha} S^{\lambda,II}_{j,k} g_{j,k}}_{\LL_p(\R_+^{n+1};E)}
         &\leq C \norm{((-\Delta_{\R^n})^m + \abs{\lambda})^{\frac{2m-k-1}{2m}} \partial_y g_{j,k}}_{\LL_p(\R_+^{n+1};E)},
         \\
        \norm{\DD^{\vec \alpha} S^{\lambda,I}_{j,k} g_{j,k}}_{\LL_p(\R_+^{n+1};E)}
         &\leq C \abs{\lambda}^{\frac{\abs{\vec \alpha}-k}{2m}} \norm{g_{j,k}}_{\LL_p(\R_+^{n+1};E)},
         \\
        \norm{\DD^{\vec \alpha} S^{\lambda,II}_{j,k} g_{j,k}}_{\LL_p(\R_+^{n+1};E)}
         &\leq C \abs{\lambda}^{\frac{\abs{\vec \alpha} - k - 1}{2m}} \norm{\frac{\partial}{\partial_y} g_{j,k}}_{\LL_p(\R_+^{n+1};E)}.
       \end{align*}        
    \end{enumerate}
  \end{proposition}

 \begin{remark}
  A result analogous to part \eqref{pt:b} can be found in \cite[Proposition 6.8]{DeHiPr03}.
  On the other hand, in part \eqref{pt:a} (and in \cite{DeHiPr07}) one starts with functions defined on $\R^n$ and uses the particular extension $\Ecal_\lambda g_j$ to a function on $\R_+^{n+1}$.
 A convenient property of this particular extension is the identity
   \[
    \frac{\partial}{\partial y} h^\lambda_{j,k}
     = \frac{\partial}{\partial y} \ee^{- L_\lambda^{1/2m} y} L_\lambda^{\frac{2m-k}{2m}} g_{j,k}
     = - L_\lambda^{1/2m} h^\lambda_{j,k}.
   \]
 \end{remark}

  \begin{proof}[Proof of Proposition \ref{prop:Estimates_Solution_Operator}]
   We will focus on part \eqref{pt:a}. Part \eqref{pt:b} follows with slight modifications, cf.\ \cite[Proposition 6.8]{DeHiPr03}.
   Having the integral formulas for the solution of the IBVP at hand, it remains to prove the asserted $\LL_p$-estimates.
   In fact, from the previous subsection (i.e.\ for $f = 0$, hence, $\Fcal w^\mathrm{fs} = \Fcal w^\mathrm{corr} = 0$), we already know that
    \[
     w^\mathrm{hs}
      = [\vec v^\mathrm{hs}]_1
      \quad \text{for }
     \vec v^\mathrm{hs}
      = \sum_{j=1}^m \sum_{k=0}^{m_j} \vec v^\mathrm{hs}_{j,k}
    \]
   with
    \begin{align*}
     \vec v^{\mathrm{hs}}(\vec x',y)
      &= \Fcal^{-1} \big( \Fcal \bb V_\rho^{2m} \cdot \Fcal \vec g_\rho^{2m} \big)(\vec x', y)
      \\
      &= (2\pi)^n \big( \Fcal^{-1} \bb V_\rho^{2m} \ast_{\vec x'} \Fcal^{-1} (\Fcal \vec g_\rho^{2m}) \big)(\vec x',y)
      \\
      &= \int_{\R^n} \bb K^\lambda(\vec x' - \vec {\tilde x}', y) \Fcal^{-1} \big( \sum_{k=0}^{m_j} \rho^{2m-k} \Fcal g_{j,k}(\vec {\tilde x}') \big)_{j=1}^m \, \dd \vec {\tilde x'}
      \\
      &= \int_{\R^n} \big[ \bb K^\lambda(\vec x' - \vec {\tilde x}', y + \tilde y) \Fcal^{-1} (\sum_{k=0}^{m_j} \rho^{2m-k} \ee^{- \rho \tilde y} \Fcal g_{j,k}(\vec {\tilde x}') \big]_{y=0}
      \\
      &= - \int_{\R^n} \int_0^\infty \frac{\partial}{\partial \tilde y} \big( \bb K^\lambda(\vec x' - \vec {\tilde x}', y + \tilde y) \Fcal^{-1} \big( \sum_{k=0}^{m_j} \rho^{2m-k} \ee^{-\rho \tilde y} \Fcal g_{j,k}(\vec {\tilde x'}) \big)_{j=1}^m \, \dd \tilde y \, \dd \vec {\tilde x}'
      \\
      &= - \int_{\R^n} \int_0^\infty \big( \frac{\partial}{\partial \tilde y} \bb K^\lambda(\vec x' - \vec {\tilde x}', y + \tilde y) - \bb K^\lambda(\vec x' - \vec {\tilde x}', y + \tilde y) ((-\Delta_{\R^n})^m + \abs{\lambda})^{1/2m} \big)
       \\ &\qquad \cdot
       \big( \sum_{k=0}^{m_j} ((-\Delta_{\R^n})^m + \abs{\lambda})^{(2m-k)/2m} \ee^{- L_\lambda^{1/2m} \tilde y} g_{j,k}(\vec {\tilde x}') \big)_{j=1}^m
      \\
      &= - \big( \sum_{k=0}^{m_j} \int_{\R^n} \int_0^\infty \big( \frac{\partial}{\partial y} K^\lambda_{j,k}(\vec x - \vec {\tilde x}, y + \tilde y) - K^\lambda_{j,k}(\vec x - \vec x', y + \tilde y) L_\lambda^{1/2m} \big)
       \\ &\qquad
       \cdot((- \Delta_{\R^n})^m + \abs{\lambda})^{\frac{2m-k}{2m}} \ee^{- L_\lambda^{1/2m} y} g_{j,k}(\vec {\tilde x}) \dd \vec {\tilde x} \big)_{j=1,\ldots,m}
      \\
      &=: \big( \sum_{k=0}^{m_j} S^{\lambda,\mathrm{I}}_{j,k} g_{j,k} + S^{\lambda,\mathrm{II}}_{j,k} g_{j,k} \big)_{j=1,\ldots,m}
      =: \big( \sum_{k=0}^{m_j} S^{\lambda}_{j,k} g_{j,k} \big)_{j=1,\ldots,m} \big)_{j=1,\ldots,m}
    \end{align*}
   where we write $\big( \ee^{- L_\lambda^{1/2m} y} \big)_{y \geq 0}$ for the strongly continuous semigroup generated by the pseudo-differential operator $- L_\lambda^{1/2m}$, cf.\ Remark \ref{rem:generator}.
   Moreover, we exploited the growth bounds for $\bb K^\lambda(\vec x - \vec {\tilde x}, y + \tilde y)$ as $\tilde y \rightarrow \infty$ and the special form of the extension of $g_{j,k}$ from $\R^n$ to $\R_+^{n+1}$ by $\Ecal_\lambda$.
   Analogously to the proof of \cite[Proposition 6.8]{DeHiPr03}, but for the additional parameter $k = 0, 1, \ldots, m_j$, we thus may define
    \begin{align*}
     v^\mathrm{hs}_{j,k}(\vec x',y)
      &= - \int_0^\infty \int_{\R^n}  T_{j,k} (\lambda, \vec x' - \vec {\tilde x}', y + \tilde y) h^\lambda_{j,k}(\vec {\tilde x}', \tilde y) \, \dd \vec {\tilde x}' \, \dd \tilde y
      \\
      &:=  - \int_0^\infty \int_{\R^n}  \big[ \frac{\partial}{\partial \tilde y} K^\lambda_{j,k} - K^\lambda_{j,k} L_\lambda^{1/2m} \big] (\vec x' - \vec {\tilde x}', y + \tilde y) h^\lambda_{j,k}(\vec {\tilde x}', \tilde y) \, \dd \vec {\tilde x}' \, \dd \tilde y.
    \end{align*}
   As in the proof of \cite[Proposition 6.8]{DeHiPr03} and applying Proposition \ref{prop:kernel_estimates_half-space}, we then find that
    \begin{align*}
     &\norm{\DD^{\vec \alpha} \big( \frac{\partial}{\partial \tilde y} K^\lambda_{j,k}(\cdot,\tilde y) + K^\lambda_{j,k}(\cdot, \tilde y) L_\lambda^{1/2m} \big)}_{\LL_1(\R^n;\B(E))}
      \\
      &\leq C \abs{\lambda}^{\frac{n-2m+\abs{\vec \alpha}}{2m}} \int_{\R^n} p_{2m,\abs{\vec \alpha}}^{n+1}(c \abs{\lambda}^{\nicefrac{1}{2m}} \tilde y) \, \dd \vec x', \lambda
      \\
      &\leq C \abs{\lambda}^{\frac{n-2m+\abs{\vec \alpha}}{2m}} \int_0^\infty p_{2m,\abs{\vec \alpha}}^{n+1}(c \abs{\lambda}^{\nicefrac{1}{2m}} (r + \tilde y)) r^{n-1} \, \dd r
      \\
      &= n! \, C \abs{\lambda}^{\frac{\abs{\vec \alpha} - 2m}{2m}} p_{2m+n,\abs{\vec \alpha}}^{n+1}(c \abs{\lambda}^{\nicefrac{1}{2}} \tilde y),
      \quad
      \tilde y > 0.
    \end{align*}
  Then, by \cite[Lemma 6.7]{DeHiPr03},
   \[
    \norm{\lambda^{1-\frac{\abs{\vec \alpha}}{2m}} \DD^{\vec \alpha} T^\lambda_{j,k} g_{j,k}}_{\LL_p(\R^n;E)}
    \leq C \norm{((- \Delta_{\R^n})^m + \abs{\lambda})^{\frac{2m-k}{2m}} g_{j,k}}_{\LL_p(\R^n;E)},
   \]
  so that the first estimate has been established.
  By slight modification of the arguments, the second estimate can be proved as well.
  For the third estimate, we may use that the integral operators $T^\lambda_{j,k}$ depend holomorphically on the parameter $\lambda \in \Sigma_{\pi-\phi}$, so that the third estimate is a direct consequence of the first one when applying the same technique which has been used for the proof of Proposition \ref{prop:R-boundedness_sectors}. 
  \newline
  The representation and regularity of the solution to \eqref{eqn:Abstract_BVP} follow by approximation just as at the end of the proof of \cite[Proposition 6.8]{DeHiPr03}:
  Given $g_{j,k} \in \WW_p^{2m-k}(\R_+^{n+1}; \ran (\Pcal_{j,k}))$, let $g_{j,k}^\nu \in \WW_p^{2m-k+1}(\R_+^{n+1}; \ran(\Pcal_{j,k}))$ with $g_{j,k}^\nu \rightarrow g_{j,k} \in \WW_p^{2m-k}(\R_+^{n+1}; \ran(\Pcal_{j,k}))$ as $\nu \rightarrow \infty$.
  By the estimates above, $v^\nu_{j,k} := T^\lambda_{j,k} g_{j,k}^\nu \in \WW_p^{2m}(\R_+;\ran(\Pcal_{j,k}))$ with $(\lambda J + A)v^{\nu} = \left( \begin{smallmatrix} 0 \\ \vec g^\nu \end{smallmatrix} \right)$, i.e.\ $v^\nu \in \dom(A \& B)$, and further $v^\nu_{j,k} \rightarrow v_{j,k}$ in $\WW_p^{2m-1}(\R_+^{n+1}; E)$.
  As $A \& B$ is a closed operator, $v \in \dom(A \& B) \cap \WW_p^{2m-1}(\R_+^{n+1}; E)$ with $(\lambda \bb J + A \& B)  v = \left( \begin{smallmatrix} 0 \\ \vec g \end{smallmatrix} \right)$.
  \end{proof}
  
 Similarly, the results for the boundary value problem
  \begin{equation}
   \lambda J v + (A \& B) v
    = \left( \begin{array}{c} f \\ \vec 0 \end{array} \right)
   \label{eqn:Abstract_BVP_g=0}
  \end{equation}
 transfer from \cite{DeHiPr03} to the situation considered here.
 
 \begin{proposition}
 \label{prop:Max_Regularity_Zero-bdry-data}
  Let $p \in (1, \infty)$ and differential operators $\Acal(\DD)$ and $\Bcal_j(\DD)$, $j = 1, \ldots, m$, be given as above.
  Let $\lambda \in \Sigma_{\pi-\phi}$, and $f \in \LL_p(\R_+^{n+1};E)$.
  Then there is a unique solution $w \in \WW_p^{2m-1}(\R_+^{n+1};E) \cap \dom(A \& B)$ of the boundary value problem \eqref{eqn:Abstract_BVP_g=0}, and it is given by $w = w^\mathrm{fs} + \sum_{j=1}^m \sum_{k=0}^{m_j} v^\mathrm{corr}_{j,k}$, for functions $v^\mathrm{corr}_{j,k}$ as in Proposition \ref{prop:EBVP_half_space_solution} and $w^\mathrm{fs} = P(\lambda + A_{\R^{n+1}})^{-1} \Ecal_0 f$.
  Moreover, there is a constant $C > 0$ (independent of the data $f$) such that
   \[
    \norm{\lambda^{1-\frac{\abs{\vec \alpha}}{2m}} \DD^{\vec \alpha} w}_{\LL_p(\R_+^{n+1};E)}
     \leq C \norm{f}_{\LL_p(\R_+^{n+1};E)}
     \quad
     \text{ for all }
     \vec \alpha \in \N_0^{n+1} \text{ such that }
     \abs{\vec \alpha} \leq 2m-1.
   \]
 \end{proposition}
 \begin{proof}
  The proof can be transferred almost \emph{expressis verbis} from the proof of \cite[Proposition 6.9]{DeHiPr03}.
 \end{proof} 
 
 Summarising these results, we obtain the following theorem, which can be seen as an analogue to \cite[Theorem 6.10]{DeHiPr03}.
 
 \begin{theorem}
 \label{thm:Max_Regularity_BVP}
  \begin{enumerate}
   \item
  Let $p \in (1, \infty)$ and let $E$ be an arbitrary Banach space.
  Let $\Acal(\DD)$ be a parameter-elliptic operator of order $2m$ with angle of ellipticity $\phi_\Acal^\mathrm{ellipt} \in [0,\pi)$.
  Let $\phi > \phi_\Acal^\mathrm{ellipt}$.
  For each $j = 1, \ldots, m$ let $\Bcal_j(\DD) = \sum_{k=0}^{m_j} \sum_{\abs{\vec \beta} = k} b_{j,k,\vec \beta} \DD^{\vec \beta} \Pcal_{j,k}$ be a boundary differential operator of order $m_j < 2m$.
  Assume that the Lopatinskii--Shapiro condition ${\bf (LS)}$ is valid.
  Let data $f \in \LL_p(\R_+^{n+1};E)$ and $g_j \in \sum_{k=0}^{m_j} \WW_p^{2m-k-\frac{1}{p}}(\R^n;\ran(\Pcal_{j,k}))$, $j = 1, \ldots, m$, be given.
  Let $\lambda \in \Sigma_{\pi-\phi}$ and let $A \& B$ be defined as above.
  Then the boundary value problem
   \[
    \lambda J u + (A \& B) u
     = \left( \begin{array}{c} f \\ \Ecal_\lambda \vec g \end{array} \right)
   \]  
  has a unique solution $u \in \WW_p^{2m-1}(\R_+^{n+1};E) \cap \dom(A \& B)$ which, for $R^\lambda_0 := P_{\R_+^{n+1}} (\lambda + A_{\R^{n+1}})^{-1} \Ecal_0$, is given by
   \[
    u
     = R^\lambda_0 f + \sum_{j=1}^m \sum_{k=0}^{m_j} \left( R^\lambda_{j,k} f + S^\lambda_{j,k} \Ecal_\lambda g_{j,k} \right)
   \]
  and there is a constant $C > 0$ which is independent of the data $(f,\vec g)$ and the parameter $\lambda$ such that for every multi-index $\vec \alpha \in \N_0^{n+1}$ of length $\abs{\vec \alpha} \leq 2m-1$ and every $\lambda \in \Sigma_{\pi-\phi}$ we have
   \begin{align}
    \norm{\lambda^{1-\frac{\abs{\vec \alpha}}{2m}} \DD^{\vec \alpha} u}_{\LL_p(\R_+^{n+1};E)}
     &\leq C \left( \norm{f}_{\LL_p(\R_+^{n+1};E)} + \sum_{j=1}^m \sum_{k=0}^{m_j} \norm{((- \Delta)^m + \abs{\lambda})^{\frac{2m-k}{2m}} \Ecal_\lambda g_{j,k}}_{\LL_p(\R_+^{n+1};E)} \right).
     \nonumber
   \end{align}
    \item
     If we start with given functions $g_{j,k} \in \WW_p^{2m-k}(\R_+^{n+1};E)$ and replace $\Ecal_\lambda \vec g$ in the statement above by $\vec g$, we obtain the same result, except for the modified estimates
   \begin{align}
    \norm{\lambda^{1-\frac{\abs{\vec \alpha}}{2m}} \DD^{\vec \alpha} u}_{\LL_p(\R_+^{n+1};E)}
     &\leq C \left( \norm{f}_{\LL_p(\R_+^{n+1};E)} + \sum_{j=1}^m \sum_{k=0}^{m_j} \norm{((- \Delta)^m + \abs{\lambda})^{\frac{2m-k}{2m}} g_{j,k}}_{\LL_p(\R_+^{n+1};E)} 
     \right. \label{eqn:estimate_solution} \\ &\qquad \left.
     + \sum_{j=1}^m \sum_{k=0}^{m_j} \norm{((-\Delta)^m + \abs{\lambda})^{\frac{2m-k-1}{2m}} \frac{\partial}{\partial y} g_{j,k}}_{\LL_p(\R_+^{n+1};E)} 
    \right).
     \nonumber
   \end{align}
  \end{enumerate}
 \end{theorem}
 
 \subsection{$\Hcal^\infty$-calculus}
 
 In this subsection, we transfer the results on the $\Hcal^\infty$-calculus as in \cite[Section 7.1]{DeHiPr03} to the situation considered here.
 We start by combining Theorem \ref{thm:Max_Regularity_BVP} and  \cite[Lemma 7.1]{DeHiPr03}, and immediately obtain the following theorem.
 
 \begin{theorem}
 \label{thm:Max_Regularity_BVP_HT}
 If, in the situation of Theorem \ref{thm:Max_Regularity_BVP}, the Banach space $E$ is of class $\HTcal$, then $u$ lies in $\WW_p^{2m}(\R_+^{n+1};E)$ and \eqref{eqn:estimate_solution} holds true for all $\vec \alpha \in \N_0^{n+1}$ with $\abs{\vec \alpha} = 2m$ as well.
 Moreover, in this case
  \begin{align*}
   \norm{\DD^{\vec \alpha} S^\lambda_{j,k} \Ecal_\lambda g_j}_{\LL_p(\R_+^{n+1};E)}
    &\leq C \abs{\lambda}^{\frac{\abs{\vec \alpha} - k}{2m}} \norm{\Ecal_\lambda g_{j,k}}_{\LL_p(\R_+^{n+1};E)},
    &&0 \leq \abs{\vec \alpha} \leq k \leq m_j,
    \\
    \intertext{if $g_{j,k}$ is only given on $\R^n \cong \partial \R_+^{n+1}$, else}
   \norm{\DD^{\vec \alpha} S^{\lambda,\mathrm{I}}_{j,k} \Ecal_\lambda g_{j,k}}_{\LL_p(\R_+^{n+1};E)}
    &\leq C \abs{\lambda}^{\frac{\abs{\vec \alpha} - k}{2m}} \norm{g_{j,k}}_{\LL_p(\R_+^{n+1};E)},
    &&0 \leq \abs{\vec \alpha} \leq k \leq m_j,
    \\
   \norm{\DD^{\vec \alpha} S^{\lambda,\mathrm{II}}_{j,k} g_{j,k}}_{\LL_p(\R_+^{n+1};E)}
    &\leq C \abs{\lambda}^{\frac{\abs{\vec \alpha} - k - 1}{2m}} \norm{\tfrac{\partial}{\partial y} g_{j,k}}_{\LL_p(\R_+^{n+1};E)},
    &&0 \leq \abs{\vec \alpha} \leq k + 1 \leq m_j + 1
 \end{align*}
 and the $\LL_p$-realisation $A_B$ of the boundary value problem \eqref{eqn:Abstract_BVP_g=0} is
  \begin{align}
   A_B u
    &= \Acal(\DD) u,
    \label{eqn:Def_Lp-realisation}
    \\
   \dom(A_B)
    &= \{ u \in \WW_p^{2m}(\R_+^{n+1};E): \, \Bcal_j(\DD) u = 0 \quad \text{for all } j = 1, \ldots, m \}.
  \end{align}
 \end{theorem}
 
 Together with the kernel estimates of Proposition \ref{prop:kernel_estimates_half-space}, it follows that $A_B$ admits a \emph{bounded $\Hcal^\infty$-calculus}, see the following analogue to \cite[Theorem 7.4]{DeHiPr03}:
 
 \begin{theorem}
  \label{thm:Bounded_Hinfty_Calculus}
  Let $E$ be a Banach space of class $\HTcal$ and let $A_B: \dom(A_B) \rightarrow \LL_p(\R_+^{n+1};E)$ be defined as in Theorem \ref{thm:Max_Regularity_BVP_HT}. Then $A_B \in \Hcal^\infty(\LL_p(\R_+^{n+1};E))$ with $\Hcal^\infty$-angle $\phi_{A_B}^\infty \leq \phi_\Acal^\mathrm{ellipt}$.
 \end{theorem}

 The proof can be executed similarly to the proof of \cite[Theorem 7.4]{DeHiPr03}, using \cite[Lemma 7.1]{DeHiPr03} and an adjusted version of \cite[Remark 7.2c]{DeHiPr03}.
 
 \begin{proof}
  Let us fix any $\phi \in (\phi_\Acal^\mathrm{ellipt},\pi)$ and recall that we may write the resolvents of the operator $A_B$ as
   \[
    (\lambda + A_B)^{-1}
     = R^\lambda_0 + \sum_{j=1}^m \sum_{k=0}^{m_j} R^\lambda_{j,k}
     \quad
     \text{for all }
     \lambda \in \Sigma_{\pi-\phi}.
   \]
  The terms correspond to the resolvent for the full-space problem plus corrective terms due to the homogeneous boundary conditions imposed at $\partial \R_+^{n+1}$.
  \newline
  Let $h \in \Hcal^\infty(\Sigma_\phi)$ be an arbitrary polynomially decaying function, i.e.\ there is $C, s > 0$ such that
   \[
    \abs{h(z)}
     \leq C (1 + \abs{z})^{-s},
   \]
and denote by $\Gamma$ the contour given by $(0,\infty] \ee^{\ii \theta} \cup (0, \infty) \ee^{-\ii \theta}$ for some $\theta \in (\phi_\Acal^\mathrm{ellipt}, \phi)$.
  We may then use the Dunford integral calculus to define
   \[
    h(A_B)
     := \frac{1}{2\pi \ii} \int_\Gamma h(\lambda) (\lambda + A_B)^{-1} \dd \lambda.
   \]
  We want to show that $A_B \in \Hcal^\infty(\LL_p(\R_+^{n+1};E))$ with $\Hcal^\infty$-angle $\phi_{A_B}^\infty \leq \phi_\Acal^\mathrm{ellipt}$, thus it remains to prove that for all such $\phi > \phi_\Acal^\mathrm{ellipt}$ and $h \in \Hcal_0(\Sigma_\phi)$ which polynomially decay to zero at infinity, we have an estimate
   \[
    \norm{h(A_B)}_{\B(\LL_p(\R_+^{n+1};E))}
     \leq C \norm{h}_\infty
   \]
  for some constant $C > 0$ which is independent of $h \in \Hcal_0(\Sigma_\phi)$.
  For the full-space problem related term $R^\lambda_0$ such an estimate has already been established in \cite[Theorem 5.5]{DeHiPr03}, relying on the condition that $E$ is a Banach space of class $\HTcal$.
  The remaining terms $R^\lambda_{j,k}$ are the integral operators associated with the integral kernels $K^{\mathrm{corr},\lambda}_{j,k}$, hence, $\int_\Gamma h(\lambda) R^\lambda_{j,k} \, \dd \lambda$ is the integral operator associated to the kernel $\int_\Gamma h(\lambda) K^{\mathrm{corr},\lambda}_{j,k}(\cdot,\cdot,\cdot) \, \dd \lambda$.
  Since the estimate
   \begin{align*}
    \abs{ \int_\Gamma h(\lambda) K^{\mathrm{corr}, \lambda}_{j,k}(\vec x', y, \tilde y) \, \dd \lambda }
     &\leq C \norm{h}_\infty \int_0^\infty r^{\frac{n+1}{2m}} p_{2m,0}^{n+1} (c r^{\frac{1}{2m}} (\abs{\vec x'} + y + \tilde y) \, \dd r
     \\
     &\leq C \norm{h}_\infty \frac{1}{(\abs{\vec x'} + y + \tilde y)^{n+1}} \int_0^\infty \tau^n p_{2m,0}^{n+1}(c \tau) \, \dd \tau
     \\
     &\leq C \norm{h}_\infty \frac{1}{(\abs{\vec x'} + y + \tilde y)^{n+1}}
   \end{align*}
  is valid thanks to $\int_0^\infty \tau^n p_{2m,0}^{n+1}(\tau) \dd \tau < \infty$, we may employ Proposition \ref{prop:holomorphic_image_of_bounded_sets} and \cite[Remark 7.2 c)]{DeHiPr03} to conclude the assertion.
 \end{proof}

 \begin{corollary}
  \label{cor:Rcal-Bounds}
  Let $E$ be a Banach space of class $\HTcal$.
  Then
   \[
    \Rcal ( \{\lambda^{1-\frac{\abs{\vec \alpha}}{2m}} \DD^{\vec \alpha} (\lambda + A_B)^{-1}: \, \lambda \in \Sigma_{\pi-\phi}, \, \vec \alpha \in \N_0^{n+1} \text{ such that } 0 \leq \abs{\vec \alpha} \leq 2m \} )
     < \infty.
   \]
 \end{corollary}
 
 \begin{proof}
  By Theorem \ref{thm:Bounded_Hinfty_Calculus} and Remark \ref{rem:Inclusions_Operator-Spaces}, $A_B \in \Hcal^\infty(E) \subseteq \BIPcal(E)$, hence, for every $\sigma \in (0,1)$
   \[
    \dom(A_B^\sigma)
     = \big( \LL_p(\R_+^{n+1};E), \WW_p^{2m}(\R_+^{n+1};E) \big)_{\sigma,p}
     = \HH_p^{2m\sigma}(\R_+^{n+1};E).
   \]
  Thus, for each multi-index $\vec \alpha \in \N_0^{n+1}$ of length $\abs{\vec \alpha} \leq 2m$, the linear operator $\DD^{\vec \alpha} A_B^{-\frac{k}{2m}}$ with $k = \abs{\vec \alpha}$ is a bounded operator on $E$.
  As $\{ \lambda (\lambda + A_B)^{-1}: \, \lambda \in \Sigma_{\pi-\phi} \}$ is $\Rcal$-bounded, all the sets
   \[
    \{ \lambda^{- \frac{\abs{\vec \alpha}}{2m}} \DD^{\vec \alpha} \lambda (\lambda + A_B)^{-1}: \, \lambda \in \Sigma_{\pi-\phi} \}
    \quad
    \text{for each } \abs{\vec \alpha} \leq 2m
   \]
  are $\Rcal$-bounded.
  Since, here, the set of admissible multi-indices $\{ \vec \alpha \in \N_0^{n+1}: \, \abs{\vec \alpha} \leq 2m \}$ is finite, the family of operators
   \[
    \{ \lambda^{- \frac{\abs{\vec \alpha}}{2m}} \DD^{\vec \alpha} \lambda (\lambda + A_B)^{-1}: \, \lambda \in \Sigma_{\pi-\phi}, \, \abs{\vec \alpha} \leq 2m \}
   \]
  is $\Rcal$-bounded as well.
 \end{proof}

 \subsection{$\Rcal$-bounds for solution operators}
 
 By Corollary \ref{cor:Rcal-Bounds}, the solution map for the boundary value problem \eqref{eqn:Abstract_BVP_g=0} with homogeneous boundary data $\vec g = \vec 0$ admits $\Rcal$-bounds. Next, we consider the boundary value problem \eqref{eqn:Abstract_BVP_g=0} for $f = 0$ and inhomogeneous boundary data $\vec g \neq 0$ and also conclude $\Rcal$-boundedness for the corresponding solution maps. By Proposition \ref{prop:Estimates_Solution_Operator} the solution $v$ to $(\lambda \bb J + A \& B) v = (0, \Ecal_\lambda \vec g)^\mathsf{T}$ may be expressed as
  \[
   v = \sum_{j=1}^m \sum_{k=0}^{m_j} S^\lambda_{j,k} \Ecal_\lambda g_{j,k}.
  \]
 Our goal is to prove that the boundary value problem
  \[
   \begin{cases}
    (\lambda + \Acal(\DD)) u = f
     &\text{in } \R_+^{n+1},
     \\
    \Bcal_j(\DD) u = g_{j,k}
     &\text{on } \partial \R_+^{n+1}, \, j = 1, \ldots, m
   \end{cases}
  \] 
 has a (unique) solution in the class $u \in \WW_p^{2m}(\R_+^{n+1};E)$ if and only if the given functions lie in the classes $f \in \LL_p(\R_+^{n+1};E)$ and $g_{j,k} \in \WW_p^{2m-k-1/p}(\partial \R_+^{n+1};E)$ for $j = 1, \ldots, m$ and $k = 0, 1, \ldots, m_j$.
 Using the extension operator $\Ecal_\lambda$, the latter is equivalent to $\Ecal_\lambda g_{j,k} \in \WW_p^{2m-k}(\R_+^{n+1};E)$ for all $j = 1, \ldots, m$ and $k = 0, 1, \ldots, m_j$, and since for $\lambda \in \Sigma_{\pi-\phi}$, i.e.\ $\abs{\lambda} > 0$, $((-\Delta)^m + \abs{\lambda})^{\frac{2m-k}{2m}}$ acts as a top-linear isomorphy between $\WW_p^k(\R_+^{n+1};E)$ and $\LL_p(\R_+^{n+1};E)$, we may reformulate the problem as follows:
 Show that that
  \[
   \begin{cases}
    ( \lambda + \Acal(\DD) ) u  = f
     &\text{in } \R_+^{n+1},
     \\
    ((-\Delta_{\R^n})^m + \abs{\lambda})^{\frac{2m-k}{2m}} \Bcal_{j,k}(\DD) u|_{y = 0} = \tilde g_{j,k}
     &\text{on } \partial \R_+^{n+1}, \, 1 \leq j \leq m, \, 0 \leq k \leq m_j
   \end{cases}
  \]
 where $\tilde g_{j,k} := ((- \Delta_{\R^n})^m + \abs{\lambda})^{\frac{2m-k}{2m}} \Ecal_\lambda g_{j,k}|_{y = 0}$, has a (unique) solution $u \in \WW_p^{2m}(\R_+;E)$ if and only if $f \in \LL_p(\R_+^{n+1};E)$ and $((- \Delta_{\R^n})^m + \abs{\lambda})^{\frac{2m-k}{2m}} \Ecal_\lambda g_{j,k} \in \LL_p(\R_+;E)$ for all $j = 1, \ldots, m$ and $k = 0, 1, \ldots, m_j$.
 Therefore, for $\lambda \in \Sigma_{\pi - \phi}$, $j = 1, \ldots, m$ and $k = 0, \ldots, m_j$, and using the operators $S_{j,k}^{\lambda,\mathrm{I}}$ and $S_{j,k}^{\lambda,\mathrm{I}}$ from Proposition \ref{prop:Max_Regularity_Zero-bdry-data} we introduce operators
  \begin{align}
   U^\lambda_{j,k}
    &= S^{\lambda,\mathrm{I}}_{j,k} ((-\Delta)^m + \abs{\lambda})^{\frac{k-2m}{2m}} \Ecal_\lambda,
    \label{eqn:U_lambda}
    \\
   V^\lambda_{j,k}
    &= S^{\lambda,\mathrm{II}}_{j,k}((-\Delta)^m + \abs{\lambda})^{\frac{k-2m}{2m}} \Ecal_\lambda.
    \label{eqn:V_lambda}
  \end{align}
 The families consisting of these operators allow for the following $\Rcal$-bounds, cf.\ \cite[Proposition 7.6]{DeHiPr03}.
 \begin{proposition}
 \label{prop:Rcal-Bounds}
 For each $j = 1, \ldots, m$ and $k = 0, 1, \ldots, m_j$, the sets
  \begin{align*}
   \{\lambda^{1-\frac{\abs{\vec \alpha}}{2m}} \DD^{\vec \alpha} U^\lambda_{j,k}: \, \lambda \in \Sigma_{\pi-\phi}, \, \abs{\vec \alpha} \leq 2m \}
    &\subseteq \B(\LL_p(\R_+^{n+1};E)),
    \\
   \{\lambda^{1-\frac{\abs{\vec \alpha}}{2m}} \DD^{\vec \alpha} V^\lambda_{j,k}: \, \lambda \in \Sigma_{\pi-\phi}, \, \abs{\vec \alpha} \leq 2m \}
    &\subseteq \B(\LL_p(\R_+^{n+1};E)),
    \\
   \{ \lambda^{2-\frac{\abs{\vec \alpha}}{2m}} \DD^{\vec \alpha} \partial_\lambda U^\lambda_{j,k}: \, \lambda \in \Sigma_{\pi-\phi}, \, \abs{\vec \alpha} \leq 2m \}
    &\subseteq \B(\LL_p(\R_+^{n+1};E)),
    \\
   \{ \lambda^{2-\frac{\abs{\vec \alpha}}{2m}} \DD^{\vec \alpha} \partial_\lambda V^\lambda_{j,k}: \, \lambda \in \Sigma_{\pi-\phi}, \, \abs{\vec \alpha} \leq 2m \}
    &\subseteq \B(\LL_p(\R_+^{n+1};E))
  \end{align*}
 are $\Rcal$-bounded.
 \end{proposition} 
 \begin{proof}
  By Proposition \ref{prop:Estimates_Solution_Operator}, the operators $\DD^{\vec \alpha} U^\lambda_{j,k}$ and $\DD^{\vec \alpha} V^\lambda_{j,k}$ are integral operators on $\LL_p(\R_+^{n+1};E)$ with kernels $\DD^{\vec \alpha} K^\lambda_{U^{j,k}}$ and $\DD^{\vec \alpha} K^\lambda_{V_{j,k}}$, resp., where
   \begin{align*}
    \Fcal K^\lambda_{U^{j,k}}(\vec \xi',y)
     &= \Fcal \big( \frac{\partial}{\partial y} K^{\lambda,\mathrm{I}}_{j,k} \big)(\vec \xi',y) (\abs{\vec \xi'}^{2m} + \abs{\lambda})^{\frac{2m-k}{2m}} \Pcal_{j,k} (\abs{\vec \xi'}^{2m} + \abs{\lambda})^{\frac{k-2m}{2m}} \Ecal_\lambda,
     \\
    \Fcal K^\lambda_{V^{j,k}}(\vec \xi',y)
     &= \Fcal K^{\lambda,\mathrm{II}}_{j,k}(\vec \xi',y) \frac{\partial}{\partial y} (\abs{\vec \xi'}^{2m} + \abs{\lambda})^{\frac{2m-k-1}{2m}} \Pcal_{j,k} (\abs{\vec \xi'}^{2m} + \abs{\lambda})^{\frac{k+1-2m}{2m}} \Ecal_\lambda,
   \end{align*}
  for $j = 1, \ldots, m$ and $k = 0, 1, \ldots, m_j$.
  By \cite[Proposition 4.1 and Lemma 7.1]{DeHiPr03}, it suffices to prove that there is $C > 0$ such that for all $\vec x' \in \R^n$  and $y, \tilde y > 0$
   \begin{align*}
    \Rcal ( \{ \lambda^{1- \frac{\abs{\vec \alpha}}{2m}} \DD^{\vec \alpha} \frac{\partial}{\partial \tilde y} K^\lambda_{U^{j,k}}(\vec x', y + \tilde y): \, \lambda \in \Sigma_{\pi-\phi}, \, \abs{\vec \alpha} \leq 2m \} )
     &\leq \frac{C}{(\abs{\vec x'} + y + \tilde y)^{n+1}},
     \\
    \Rcal ( \{ \lambda^{1- \frac{\abs{\vec \alpha}}{2m}} \DD^{\vec \alpha} \frac{\partial}{\partial \tilde y} K^\lambda_{V^{j,k}}(\vec x', y + \tilde y): \, \lambda \in \Sigma_{\pi-\phi}, \, \abs{\vec \alpha} \leq 2m \} )
     &\leq \frac{C}{(\abs{\vec x'} + y + \tilde y)^{n+1}}.
   \end{align*}
  This can be proved as follows, analogously to \cite[Proposition 7.7]{DeHiPr03}.
  For each $\vec x' \in \R^n$, $y > 0$, $\lambda \in \Sigma_{\pi-\phi}$ and $\vec \alpha, \vec \beta \in \N_0^n$, $\gamma \in \N_0$ such that $\abs{\vec \alpha} \leq 2m$ and $\abs{\vec \beta} + \gamma = \abs{\vec \alpha}$, we have the identity
   \begin{align*}
    &\DD^{\vec \alpha} \frac{\partial}{\partial y} K^\lambda_{U^{j,k}}(\vec x',y)
     \\
     &= \frac{1}{(2\pi)^n} \int_{\R^n} \ee^{\ii \vec x' \cdot \vec \xi} \frac{\ee^{\ii \rho \bb A_0(\vec b,\sigma) y}}{\rho^{2m}}(\vec \xi')^{\vec \beta} (A_0(b,\sigma) \rho)^\gamma \ii \bb A_0(\vec b,\sigma) M(b,\sigma) \rho^{2m-k} J_{\vec \xi',\lambda}^{2m-k} \, \dd \vec \xi',
     \nonumber
   \end{align*}
  where $J_{\vec \xi',\lambda}^{2m-k} = (\abs{\vec \xi'}^{2m} + \abs{\lambda})^{\frac{k-2m}{2m}}$.
  As in the proof of Proposition \ref{prop:kernel_estimates_half-space}, we choose a rotation $\bb Q \in
\R^{n \times n}$ such that $\bb Q \vec x' = (\abs{\vec x'}, 0, \ldots, 0)^\mathsf{T}$ and write $\bb Q \vec \xi' = (a, r \vec \varphi)$ for some $a \in \R$, $r > 0$ and $\vec \varphi \in \S^{n-2}$. Then
 \begin{align*}
  \DD^{\vec \alpha} \frac{\partial}{\partial y} K^\lambda_{U^{j,k}}(\vec x',y)
   &= \frac{1}{(2\pi)^n} \int_{\S^{n-2}} \int_0^\infty r^{n-2} \int_{-\infty}^\infty \ee^{\ii \abs{\vec x'} a} \frac{\ee^{\ii \rho \bb A_0(\vec b, \sigma) y}}{\rho^{2m-1}} (\vec \xi')^{\vec \beta}
    \\ &\quad \cdot
    (\bb A_0(\vec b,\sigma) \rho)^\gamma \ii \bb A_0(\vec b,\sigma) \bb M(\vec b,\sigma) \rho^{2m-k} J_{\vec \xi',\lambda}^{2m-k} \, \dd a \, \dd r \, \dd \S^{n-2}(\vec \varphi)
   \\
   &= \frac{1}{(2\pi)^n} \int_{\S^{n-2}} \int_0^\infty r^{n-2} \int_{\Gamma^-} \ee^{\ii \abs{\vec x'} a} \frac{\ee^{\ii \rho \bb A_0(\vec b, \sigma) y}}{\rho^{2m-1}} (\vec \xi')^{\vec \beta}
    \\ &\quad \cdot
    (\bb A_0(\vec b,\sigma) \rho)^\gamma \ii \bb A_0(\vec b,\sigma) \bb M(\vec b,\sigma) \rho^{2m-k} J_{\vec \xi',\lambda}^{2m-k} \, \dd a \, \dd r \, \dd \S^{n-2}(\vec \varphi),
 \end{align*}
 where as in the proof of Proposition \ref{prop:kernel_estimates_half-space} $\gamma^-(s) = s + \ii \varepsilon (r + \abs{s} + \abs{\lambda}^{\nicefrac{1}{2}m})$, $s \in \R$, is a parametrisation of the curve $\Gamma^-$, and Cauchy's integral theorem as well as holomorphy of $\bb A_0$ and $\bb M$ have been employed.
 Setting $c = \nicefrac{c_1}{2}$ for $c_1 > 0$ from \eqref{eqn:Spectral_Gap_S+}, as in Proposition \ref{prop:kernel_estimates_half-space} we obtain the identity
  \begin{align*}
   &\abs{\lambda}^{1-\frac{\abs{\vec \alpha}}{2m}} \DD^{\vec \alpha} \frac{\partial}{\partial y} K^\lambda_{U^{j,k}}(\vec x',y)(\abs{\vec x'} + y)^{n+1}
    \\
    &= \frac{1}{(2\pi)^n} \int_{\S^{n-2}} \int_0^\infty \int_0^\infty F(r,\tau,\vec \varphi, \lambda, \vec x', y) r^{n-2} (\abs{\vec x'} + y)^n \ee^{-\frac{1}{2} (\varepsilon \abs{\vec x'} + cy)(r + \tau)} \, \dd r \, \dd \tau \, \dd \S^{n-2}(\vec \varphi),
  \end{align*}
 where
  \[
   F(r,\tau,\vec \varphi, \lambda, \vec x',y)
    = G(r,\tau, \vec \varphi, \lambda, y) H(\vec x', y, r, \tau, \lambda)
  \]
 with
  \begin{align*}
   G(r,\tau,\vec \varphi, \lambda, y)
    &= \ee^{\ii \rho \bb A_0(\vec b, \sigma) y} \ee^{cy(r + \tau + \abs{\lambda}^{\nicefrac{1}{2}m})} \bb A_0(\vec b,\sigma)^\gamma \bb A_0(\vec b, \sigma) \bb M(\vec b,\sigma),
    \\
   H(\vec x', y, r, \tau, \lambda)
    &= (\abs{\vec x'} + y) \ee^{-(\varepsilon \abs{\vec x'} + c y) \abs{\lambda}^{\nicefrac{1}{2}m}}  \ee^{- \frac{1}{2} (\varepsilon \abs{\vec x'} + cy)(r + \tau)} {\vec \xi'}^{\vec \beta} \rho^\gamma \frac{\lambda^{1-\frac{\abs{\vec \alpha}}{2m}}}{\rho^{2m-1}} J_{\vec \xi',\lambda}^{2m-k}
  \end{align*}
 and where $\rho, \vec b, \sigma$ depend on $r, \tau, \lambda$ and $\abs{\vec \alpha} = \abs{\vec \beta} + \gamma$.
  Setting $M = \S^{n-2} \times \R_+ \times \R_+$, this integral equals $\int_M F(r,\tau,\vec \varphi, \lambda, \vec x', y) \, \dd \mu(r,\tau, \vec \varphi)$, where \[\ \dd \mu(r,\tau,\vec \varphi) = r^{n-2} (\abs{\vec x'} + y)^n \ee^{-\frac{1}{2}(\varepsilon \abs{\vec x'} + cy)(r + \tau)} \, \dd r \, \dd \tau \, \dd \sigma(\vec \varphi).\] Here, $\mu$ is a finite measure on $M$ and such that $\int_M \, \dd \mu(r,\tau, \vec \varphi) \leq C$ for some $C > 0$ independent of $\vec x' \in \R^n$ and $y > 0$.
  By \cite[Proposition 3.8]{DeHiPr03} it follows that for $\abs{\vec \beta} + \gamma = \abs{\vec \alpha}$
   \begin{align*}
    &(\abs{\vec x'} + y)^{n+1} \Rcal ( \{ \lambda^{1-\frac{\abs{\vec \alpha}}{2m}} \DD^{\vec \alpha} \frac{\partial}{\partial y} K^\lambda_{U^{j,k}}(\vec x',y): \, \lambda \in \Sigma_{\pi-\phi}, \, \abs{\vec \alpha} \leq 2m \} )
     \\
     &\leq \Rcal ( \{ (\abs{\vec x'} + y)^{n+1} \lambda^{1-\frac{\abs{\vec \alpha}}{2m}} \DD^{\vec \alpha} \frac{\partial}{\partial y} K^\lambda_{U^{j,k}}(\vec x',y): \, \lambda \in \Sigma_{\pi-\phi}, \, \abs{\vec \alpha} \leq 2m, \, \vec x' \in \R^n, \, y > 0 \} )
     \\
     &= \Rcal \big( \big\{ \int_M F(r,\tau,\vec \varphi, \lambda, \vec x',y) \, \dd \mu(r,\tau, \vec \varphi): \, \lambda \geq 0, \, \vec x' \in \R^n, \, \vec \varphi \in \S^{n-2} \big\} \big).
   \end{align*}
 The assertion for $V^\lambda_{j,k}$ follows similarly and those for the $\lambda$-derivatives follow $\Rcal$-boundedness for the operators $U^\lambda_{j,k}$, $V^\lambda_{j,k}$ and the abstract result Proposition \ref{prop:R-boundedness_sectors}.
 \end{proof}
 
 \subsection{Spatially dependent interior and boundary symbols}
 
 As situations with spatial dependent parameters are of utmost importance, in particular for the localisation procedure by which results on the half-space are transferred to general, but sufficiently regular domains, we consider spatially dependent perturbations of the principle part next, i.e.\ the interior symbol $\Acal$ is allowed to be of the form
  \[
   \Acal(\vec x, D)
    = \sum_{\abs{\vec \alpha} = 2m} a_{\vec \alpha}(\vec x) \DD^{\vec \alpha}
  \]
 with spatially dependent, but bounded and uniformly continuous coefficients $a_{\vec \alpha}: \R_+^{n+1} \rightarrow \B(E)$, and similarly for the boundary symbols
  \[
   \Bcal_j(\vec x,\DD)
    = \sum_{k=0}^{m_j} \sum_{\abs{\vec \beta} = k} b_{j,\vec \beta}(\vec x) \DD^{\vec \beta} \Pcal_{j,k},
    \quad
    j = 1, \ldots, m
  \]
 where we demand regularity of the coefficients $b_{j,\vec \beta} \in \mathrm{BUC}^{2m-k}(\R_+^{n+1};E)$ for each multi-index $\vec \beta \in \N_0^{n+1}$ of length $\abs{\vec \beta} = k$.
 Throughout this subsection, we assume that some $\phi > \phi_\Acal^\mathrm{ellipt}$ is given, the Banach space $E$ is of class $\HTcal$ and the Lopatinskii--Shapiro condition is satisfied at every boundary point $\vec x \in \partial \R_+^{n+1}$.
 For now, let us fix $\vec x_0 \in \R_+^{n+1}$ and, as an intermediate step, assume that, for some given small $\varepsilon > 0$, the coefficients $a_{\vec \alpha}$ and $b_{j,\vec \beta}$ have uniformly small oscillation, in the sense that
  \begin{align*}
   \sup_{\vec x \in \R_+^{n+1}} \sum_{\abs{\vec \alpha} = 2m} \norm{a_{\vec \alpha}(\vec x) - a_{\vec \alpha}(\vec x_0)}_{\B(E)}
    &< \varepsilon,
    \\
   \sup_{\vec x \in \R_+^{n+1}} \sum_{\abs{\vec \beta} = k} \norm{b_{j,k,\vec \beta}(\vec x) - b_{j,k,\vec \beta}(\vec x_0)}_{\B(E)}
    &< \varepsilon
    \quad \text{for }
    j = 1, \ldots, m, \, k = 0, \ldots, m_j.
  \end{align*}
 We may then interpret $\Acal(\vec x, D)$ and $\Bcal_j(\vec x, D)$ as small perturbations from the constant coefficient case:
  \begin{align*}
   \Acal(\vec x,\DD)
    &= \Acal(\vec x_0,\DD) + \Acal^\mathrm{sm}(\vec x,\DD)
    &&\text{for } \vec x \in \R_+^{n+1}
    \\
   \Bcal_j(\vec x,\DD)
    &= \Bcal_j(\vec x_0,\DD) + \Bcal_j^\mathrm{sm}(\vec x,\DD)
    &&\text{for } \vec x \in \partial \R_+^{n+1}.
  \end{align*}
 For each $\lambda \in \Sigma_{\pi-\phi}$ with $\phi > \phi_\Acal^\mathrm{ellipt}$ and a given right-hand side $f \in \LL_p(\R_+^{n+1};E)$, the variable coefficient problem
  \begin{equation}
   \begin{cases}
    (\lambda + \Acal(\vec x,\DD)) u = f
    &\text{in } \R_+^{n+1},
    \\
    \Bcal_j(\vec x,\DD) u = 0
    &\text{on } \partial \R_+^{n+1}, \, j = 1, \ldots, m
   \end{cases}
   \label{eqn:variable_coefficient_problem}
  \end{equation}
  has a unique solution $u \in \WW_p^{2m}(\R_+^{n+1};E)$ if and only if it is the unique solution $u \in \WW_p^{2m}(\R_+^{n+1};E)$ to the problem
  \begin{equation}
   \begin{cases}
    (\lambda + \Acal(\vec x_0,\DD)) u = f - \Acal^\mathrm{sm}(\vec x, D) u
    &\text{in } \R_+^{n+1},
    \\
    \Bcal_j(\vec x_0) u(\vec x',0) = - \Bcal_j^\mathrm{sm}(\vec x',\DD) u(\vec x,0)
    &\text{on } \partial \R_+^{n+1}, \text{ for } j = 1, \ldots, m.
   \end{cases}
   \label{eqn:perturbed_EBVP}
  \end{equation}
 A procedure as in \cite[Section 7.3]{DeHiPr03}, based on the contraction mapping principle, then delivers that, for sufficiently small $\varepsilon > 0$ and $\lambda \in \Sigma_{\pi-\phi}$ with sufficiently large modulus, the system of equations \eqref{eqn:variable_coefficient_problem} (or, equivalently, equation \eqref{eqn:perturbed_EBVP}) admits a unique solution $u \in \WW_p^{2m}(\R_+^{n+1};E)$ of the form
  \[
   u
    = [\id + (\lambda + A_{B^0}^0)^{-1} \Acal^\mathrm{sm} + \sum_{j=1}^m \sum_{k=0}^{m_j} S^\lambda_{j,k} \Bcal_j^\mathrm{sm}]^{-1} (\lambda + A_{B^0}^0)^{-1} f.
  \]
 Here, we set $A_{B^0}^0 := \Acal(\vec x_0,\DD)|_{\ker (\Bcal(\vec x_0,\DD))}$ and then define
  \begin{align*}
   A_B u
    &= \Acal(\vec x,\DD) u,
    \\
   \dom(A_B)
    &= \{ u \in \WW_p^{2m}(\R_+^{n+1};E): \, \Bcal_j(\vec x_0, D) u(\vec x',0) = 0, \, \vec x' \in \R^n, \, j = 1, \ldots, m  \}.
  \end{align*}
 Thus, for $\lambda \in \Sigma_{\pi-\phi}$ with sufficiently large modulus the operator $\lambda + A_B$ is invertible and its inverse is implicitly given by the relation
  \[
   (\lambda + A_B)^{-1}
    = (\lambda + A_{B^0}^0)^{-1} - (\lambda + A_{B^0}^0)^{-1} \Acal^\mathrm{sm}(\cdot,\DD)(\lambda + A_B)^{-1} - \sum_{j=1}^m \sum_{k=0}^{m_j} S^\lambda_{j,k} \Pcal_{j,k} \Bcal_j^\mathrm{sm} (\lambda + A_B)^{-1},
  \]
 hence, solving for $(\lambda + A_B)^{-1}$ we find that
  \[
   (\lambda + A_B)^{-1}
    = [ \id + (\lambda + A_{B^0}^0)^{-1} \Acal^\mathrm{sm}(\cdot,\DD) + \sum_{j=1}^m \sum_{k=0}^{m_j} S^\lambda_{j,k} \Pcal_{j,k} \Bcal^\mathrm{sm}_{j,k}(\cdot,\DD)]^{-1} (\lambda + A_{B^0}^0)^{-1}.
  \]
 
 With this formula at hand, we may establish $\Rcal$-bounds on the set of operators $\lambda^{1 - \frac{\abs{\vec \alpha}}{2m}} \DD^{\vec \alpha} (\lambda + A_B)^{-1}$ as follows.
 
 \begin{lemma}
  There is $\lambda_0 > 0$ such that the set
   \[
    \{ \lambda^{1 - \frac{\abs{\vec \alpha}}{2m}} \DD^{\vec \alpha} (\lambda + A_B)^{-1}: \, \lambda \in \Sigma_{\pi-\phi}, \, \abs{\lambda} \geq \lambda_0, \, 0 \leq \abs{\vec \alpha} \leq 2m \} \subseteq \B(\LL_p(\R_+^{n+1};E))
   \]
  is $\Rcal$-bounded.
 \end{lemma}
 \begin{proof}
  The proof is essentially the same as in \cite[Section 7.3]{DeHiPr03}. Instead of the operators $\Bcal_j(\DD)$, we consider each of the operators $\Bcal_{j,k}(\DD) = \Bcal_j(\DD) \Pcal_{j,k}$ ($k = 0, 1, \ldots, m_j$) separately and use that $\Bcal_j(\DD) = \sum_{k=0}^{m_j} \Bcal_{j,k}(\DD)$ is the sum of these operators.
  We leave out the details here.
 \end{proof}
 
 \subsection{Variable and lower order coefficients in a half-space}

 To generalise the results of the previous subsection to more general spatially dependent coefficients, i.e.\ to get rid of the small oscillation constraint, and elliptic operators with lower order coefficients, we use the same localisation procedure as in \cite[Section 5.3]{DeHiPr03}. In comparison with \cite[Section 7.4]{DeHiPr03} we use a more subtle notion of the principal part, which allows us to handle the mixed type boundary conditions considered here. Throughout this subsection, we assume that there are projections $\Pcal_{j,k}$ (independent of the spatial variable $\vec x$), for $j = 1, \ldots, m$ and $k = 0, 1, \ldots, m_j < 2m$, such that $\Pcal_{j,k} \Pcal_{j,k'} = 0$ for each $k \neq k'$ and we allow that the interior symbol $\Acal$ and the boundary symbols $\Bcal_j$ have the spatial, but not temporal dependent form
  \[
   \Acal(\vec x,\DD)
    = \sum_{\abs{\vec \alpha} \leq 2m} a_{\vec \alpha}(\vec x) \DD^{\vec \alpha},
    \quad
   \Bcal_j(\vec x,\DD)
    = \sum_{k=0}^{m_j} \sum_{\abs{\vec \beta} \leq k} b_{j,k,\vec \beta}(\vec x) \DD^{\vec \beta} \Pcal_{j,k},
    \quad
    j = 1, \ldots, m.
  \]
 The principle parts $\Acal_\#$ of the interior symbol and $\Bcal_{j,\#}$ of the boundary symbols are defined as
  \[
   \Acal_\#(\vec x,\DD)
    = \sum_{\abs{\vec \alpha} = 2m} a_{\vec \alpha}(\vec x) \DD^{\vec \alpha},
    \quad
   \Bcal_{j,\#}(\vec x,\DD)
    = \sum_{k=0}^{m_j} \sum_{\abs{\vec \beta} = k} b_{j,k,\vec \beta}(\vec x) \DD^{\vec \beta} \Pcal_{j,k},
    \quad
    j = 1, \ldots, m.
  \]
 Moreover, we demand the following regularity of the coefficients $a_{\vec \alpha}$ and $b_{j,k,\vec \beta}$.
 
 \begin{assumption}
  \label{assmpt_Variable_and_lower_order_coefficients}
  Assume that the following conditions are met by the interior and boundary symbols and their coefficients.
  \begin{enumerate}
   \item
    Smoothness assumptions:
     \begin{enumerate}
      \item
       $a_{\vec \alpha} \in \CC_\mathrm{l}(\overline{\R_+^{n+1}};\B(E))$ for all $\vec \alpha \in \N_0^{n+1}$ of length $\abs{\vec \alpha} = 2m$;
      \item
       $a_{\vec \alpha} \in [\LL_\infty + \LL_{r_k}](\R_+^{n+1};\B(E))$ for all $\vec \alpha \in \N_0^{n+1}$ of length $\abs{\vec \alpha} = k < 2m$ where $r_k \geq p$ such that $2m - k > \tfrac{n}{r_k}$;
      \item
       $b_{j,k,\vec \beta} \in \CC^{2m-k}(\partial \R_+^{n+1}; \B(E))$ and $b_{j,k,\vec \beta}(\ran(\Pcal_{j,k})) \subseteq \ran(\Pcal_{j,k}) + \ldots + \ran(\Pcal_{j,m_j})$ for each $j = 1, \ldots, m$, $0 \leq k \leq m_j$ and $\vec \beta \in \N_0^{n+1}$ of length $\abs{\vec \beta} = k$.
     \end{enumerate}
   \item
    There is $\phi_\Acal^\mathrm{ellipt} \in [0,\pi)$ such that
     \begin{enumerate}
      \item
       \textit{Ellipticity of the interior symbol:} 
       The principal symbol $\Acal_\#(\vec x,\vec \xi)$ is parameter-elliptic with angle of ellipticity $\phi \leq \phi_\Acal^\mathrm{ellipt}$ for each $\vec x \in \overline{\R_+^{n+1}}$.
      \item
       Lopatinskii--Shapiro condition: For each $\vec x_0 \in \overline{\R_+^{n+1}}$, $\lambda \in \overline{\Sigma_{\pi-\phi_\mathcal{A}^\mathrm{ellipt}}}$ and $\vec \xi' \in \R^n$ with $(\lambda, \vec \xi') \neq (0,\vec 0)$, the initial value problem for the principle symbols
        \begin{align*}
         (\lambda + \Acal_\#(\vec x_0, \vec \xi', \DD_{n+1}))v(y)
          &= 0,
          &&y > 0,
          \\
         \Bcal_{j,\#}(\vec x_0, \vec \xi', \DD_{n+1}) v(0)
          &= h_j,
          &&j = 1, \ldots, m
        \end{align*}
       has a unique solution $u \in \CC_0(\R_+;E)$ for each $(h_1, \ldots, h_m)^\mathsf{T} \in E^m$.    
     \end{enumerate}
  \end{enumerate}
 \end{assumption}
 
 A procedure similar to the approach in the preceeding subsection then gives the following result.
 
 \begin{theorem}
  \label{thm:Lq-Lp-optimal_regularity_zero_iv}
  Let $E$ be a Banach space of class $\HTcal$, $n,m \in \N$, $p \in (1, \infty)$ and assume that for some $\phi_\Acal^\mathrm{ellipt} \in [0,\pi)$ the boundary value problem $(\Acal, (\Bcal_j)_{j=1}^m)$ satisfies Assumption \ref{assmpt_Variable_and_lower_order_coefficients}.
  Let $A_B$ be the $\LL_p(\R_+^{n+1};E)$-realisation of the homogeneous boundary value problem with domain
   \[
    \dom(A_B)
     = \{u \in \WW_p^{2m}(\R_+^{n+1};E): \, \Bcal_j(\vec x,\DD) u = 0 \, \text{on } \partial \R_+^{n+1}, \, j = 1, \ldots, m \}.
   \]
  Then, for each $\phi > \phi_\Acal^\mathrm{ellipt}$, there is $\mu_\phi \geq 0$ such that $\mu_\phi + A_B$ is $\Rcal$-sectorial with $\phi_{\mu_\phi + A_B}^{\mathcal{R}} \leq \phi$.
  In particular, if $\phi_\Acal^\mathrm{ellipt} < \tfrac{\pi}{2}$, then the parabolic initial-boundary value problem
   \begin{align*}
    \partial_t u + A_B u + \mu_\phi u
     &= f,
     &&t > 0,
     \\
    u(0)
     &= 0
   \end{align*}
  has the property of maximal regularity in the class $\LL_{q,p}(\R_+ \times \R_+^{n+1};E))$ for each $q \in (1, \infty)$, i.e.\ for every given $f \in \LL_{q,p}(\R_+ \times \R_+^{n+1};E)$ there is a unique solution $u \in \WW_{q,p}^{(1,2)}(\R_+ \times \R_+^{n+1};E)$ and
   \[
    \norm{u}_{\WW_{q,p}^{(1,2)}}
     \leq C \norm{f}_{\LL_{q,p}}
     \quad
     \text{for all }  f \in \LL_{q,p}(\R_+ \times \R_+^{n+1};E).
   \]
 \end{theorem}
 
 Since the technique of proof is very similar to parts which appear in the proof for the general domain case, at this point we skip the details.
 
 \section{Elliptic vector-valued PDE: The case of $\CC^{2m}$-domains}
 \subsection{Localisation techniques for domains}
 \label{Sec:General_domains}
 
 So far, we have considered the half-space $\Omega = \R_+^{n+1}$ as a model case.
 Via bending and localisation techniques, we may now transfer the result to more general domains which are sufficiently regular.
 Let $E$ be a Banach space of class $\HTcal$ and $m, n, m_1, \ldots, m_m \in \N$ be natural numbers such that $m_j < 2m$ for $j = 1, \ldots, m$.
 For each $j = 1, \ldots, m$ and $k = 0, \ldots, m_j$, let $\Pcal_{j,k}$ be a projection such that $\Pcal_{j,k} \Pcal_{j,k'} = 0$ whenever $k \neq k'$.
 Let $\Omega \subseteq \R^{n+1}$ be a domain with compact $\CC^{2m}$-boundary, i.e.\ we assume that $\partial \Omega \subseteq \R^{n+1}$ is compact, and for each boundary point $\vec x_0 \in \partial \Omega$ there exist local coordinates of class $\CC^{2m}$ which are obtained from the original one by rotating and shifting, and such that the positive $x_{n+1}$-axis corresponds to the direction of the inner normal vector $- \vec n(\vec x_0)$ to $\Omega$ at $\vec x_0$.
 These local coordinates may be chosen such that they depend continuously on boundary point $\vec x_0 \in \partial \Omega$.
 We now consider the differential operators with spatial dependent coefficients of the form
  \[
   \Acal(\vec x,\DD)
    = \sum_{\abs{\vec \alpha} \leq 2m} a_{\vec \alpha}(\vec x) \DD^{\vec \alpha},
    \quad
   \Bcal_j(\vec x,\DD)
    = \sum_{k=0}^{m_j} \sum_{\abs{\vec \beta} \leq k} b_{j,k,\vec \beta}(\vec x) \DD^{\vec \beta} \Pcal_{j,k},
    \quad
    j = 1, \ldots, m,
  \]
 where the $\B(E)$-valued coefficients satisfy the following regularity and ellipticity conditions.
 
 \begin{assumption}
  \label{assmpt_general_domains}
  Let the following conditions be satisfied:
  \begin{enumerate}
   \item
    Smoothness conditions:
     \begin{enumerate}
      \item
       $a_{\vec \alpha} \in \CC_\mathrm{l}(\overline{\Omega};\B(E))$ for each multi-index $\vec \alpha \in \N_0^{n+1}$ of length $\abs{\vec \alpha} = 2m$;
      \item
       $a_{\vec \alpha} \in [\LL_\infty + \LL_{r_k}](\Omega; \B(E))$ for some $r_k \geq p$ such that $2m - k > \tfrac{n}{r_k}$, for each multi-index $\vec \alpha \in \N_0^{n+1}$ of length $\abs{\vec \alpha} = k < 2m$;
      \item
       $b_{j,k,\vec \beta} \in \CC^{2m-k}(\partial \Omega;\B(E))$ and $b_{j,k,\vec \beta}(\ran(\Pcal_{j,k})) \subseteq \ran(\Pcal_{j,k}) + \ldots + \ran(\Pcal_{j,m_j})$ for each $j = 1, \ldots, m$, $0 \leq k \leq m_j$ and each multi-index $\vec \beta \in \N_0^{n+1}$ of length $\abs{\vec \beta} = k$.
     \end{enumerate}
   \item
    There is an angle $\phi_{\Acal,\Bcal} \in [0,\pi)$ such that the following assertions hold true:
     \begin{enumerate}
      \item
       \textit{ellipticity of the interior symbol:}
       The principal symbol
        \[
         \Acal_\#(\vec x, \vec \xi)
          = \sum_{\abs{\vec \alpha} = 2m} a_{\vec \alpha}(\vec x) \vec \xi^{\vec \alpha}
        \]
       is parameter elliptic with angle of ellipticity less than $\phi_{\Acal,\Bcal}$, for each $\vec x \in \overline{\Omega} \cup \{\infty\}$.
      \item
       Lopatinskii--Shapiro condition:
       At each boundary point $\vec x \in \partial \Omega$, and for all $\lambda \in \overline{\Sigma}_{\pi - \phi_{\Acal,\Bcal}}$ and $\vec \xi \in \R^{n+1}$ such that $\vec \xi \cdot \vec n(\vec x_0) = 0$, but $(\lambda, \vec \xi) \neq (0,\vec 0)$, the initial value problem
        \begin{align*}
         \Acal_\#(\vec x, \vec \xi + \ii \vec n \tfrac{\partial}{\partial y}) v(y)
          &= 0
          &&\text{for } y > 0,
          \\
         \Bcal_j(\vec x, \vec \xi + \ii \vec n \tfrac{\partial}{\partial y}) v(0)
          &= h_j
          &&\text{for } j = 1, \ldots, m,
        \end{align*}
       has a unique solution $v \in \CC_0(\R_+;E)$, for each $(h_1, \ldots, h_m)^\mathsf{T} \in E^m$.
     \end{enumerate}
  \end{enumerate}
 \end{assumption}
 
 Under these assumptions, we may now prove maximal regularity of the parabolic problem associated to the $\LL_p(\Omega;E)$-realisation of the boundary value problem.
  
  \begin{theorem}
   \label{thm:Lp-Lq_optimal_regularity_zero_iv_general_domain}
   Let $E$ be a Banach space of class $\HTcal$, $n, m \in \N$ and $p \in (1, \infty)$.
   Let $\Omega \subseteq \R^{n+1}$ be a domain with compact $\CC^{2m}$-boundary. Suppose that for some angle $\phi_{\Acal,\Bcal} \in [0,\pi)$ the boundary value problem $(\Acal(\vec x, D), (\Bcal_j(\vec x, D))_{j=1}^m)$ satisfies Assumption \ref{assmpt_general_domains}.
   Let $A_B$ denote the $\LL_p(\Omega;E)$-realisation of the boundary value problem with domain
    \[
     \dom(A_B)
      = \{u \in \WW_p^{2m}(\Omega;E): \, \Bcal_j(\vec x,\DD)u = 0 \, \text{on } \partial \Omega, \, j = 1, \ldots, m \}.
    \]
   Then, for each $\phi > \phi_{\Acal,\Bcal}$, there is a constant $\mu_\phi \geq 0$ such that $\mu_\phi + A_B$ is $\Rcal$-sectorial with $\phi^\mathcal{R}_{\mu_\phi + A_B} \leq \phi$.
   In particular, if $\phi_{\Acal,\Bcal} < \tfrac{\pi}{2}$, the parabolic initial-boundary value problem
    \begin{align*}
     \partial_t u + A_B u + \mu_\phi u
      &= f,
      &&t > 0,
      \\
     u(0)
      &= 0
    \end{align*}
   has the property of maximal regularity in the class $\LL_q(\R_+;\LL_p(\Omega;E))$ for each $q \in (1, \infty)$.
  \end{theorem}
  \begin{proof}
  The localisation procedure as presented in \cite[Section 8.2]{DeHiPr03} carries over almost literally; again we consider the operators $\Bcal_j(\DD) \Pcal_{j,k}$ for each pair $(j,k)$ separately, but the technique stays the same.
  \newline
  As a first step, one needs to find suitable (local) diffeomorphisms which can be used to \emph{flatten} the boundary. To this end, given any $\vec x_0 \in \partial \Omega$ there are local coordinates such that the inner normal vector $- \vec n(\vec x_0)$ to $\Omega$ at $\vec x_0$ corresponds to the $n+1$-st direction, and then, since $\Omega$ is a $\CC^{2m}$-domain, one finds open sets $U_1 \subseteq \R^n$ and $U_2 \subseteq \R$ such that $\vec x_0 \in U := U_1 \times U_2 \subseteq \R^{n+1}$, and there is a height function $h \in \CC^{2m}(\overline{U_1};\R)$ such that $\Omega \cap U = \{ (\vec x', y) \in U: \, y > h(\vec x') \}$ and $\partial \Omega \cap U = \{ (\vec x', y) \in U: \, y = h(\vec x') \}$ and $\Phi(\vec x',y) = (\vec x', y - h(\vec x'))^\mathsf{T}$ defines a $\CC^{2m}$-diffeomorphism between $U$ and $\tilde U := \Phi(U)$, which can be extended to a $\CC^{2m}$-diffeomorphism
   \[
    \tilde \Phi: \{ \vec x = (\vec x', y) \in \R^{n+1}: \, y > \tilde h(\vec x') \}
     \rightarrow \R_+^{n+1}
     \quad
     \text{with } \tilde \Phi|_U = \Phi.
   \]
  In the following, we often will not distinguish between $\Phi$ and its extension $\tilde \Phi$, if no confusion may arise.
  Moreover, such a diffeomorphism $\Phi$ can be chosen to be \emph{$\CC^{2m}$-admissible}, i.e.\ the normal direction to $\partial \Omega$ at $\vec x_0 \in \partial \Omega$ and the tangent space $T_{\partial \Omega;\vec x_0}$ to $\partial \Omega$ at $\vec x_0 \in \partial \Omega$ are transformed into the normal direction and the tangential space $T_{\R_+^{n+1};\Phi(\vec x_0)}$ at $\Phi(\vec x_0) \in \partial \R_+^{n+1}$ to $\R_+^{n+1}$.
  The latter property can be characterised by a block-diagonal form of the Jacobi matrix $\DD \Phi(\vec x_0)$, i.e.\
   \[
    \DD \Phi(\vec x_0)
     = \left( \begin{array}{cc} \bb H(\vec x_0) & 0 \\ 0 & d(\vec x_0) \end{array} \right),
   \]
  where $\bb H(\vec x_0) \in \R^{n \times n}$ is invertible and $d(\vec x_0) > 0$ is a positive scalar.
  Here, replacing $\Phi(\vec x)$ by $(\DD \Phi(\vec x_0))^{-1} \Phi(\vec x)$, we may w.l.o.g.\ assume that $(\DD \Phi)(\vec x_0) = \bb I_{n+1}$ is the identity matrix.
  Using the admissible $\CC^{2m}$-diffeomorphisms, we may now \emph{push forward} any differential equation on $\Omega$ by defining
   \[
    \Gcal^\Phi: u \mapsto \Gcal^\Phi u := v,
    \quad
    \Gcal^\Phi u(\vec y)
     = v(\vec y)
     := u(\Phi^{-1}(\vec y)),
     \quad
     \vec y \in \tilde U \cap \R_+^{n+1}.
   \]
  Given a multi-index $\vec \alpha \in \N_0^{n+1} \setminus \{\vec 0\}$, by the Leibniz formula for derivatives we then obtain that
   \[
    \DD^{\vec \alpha} (\Gcal^\Phi u)(\vec y)
     = \sum_{1 \leq \abs{\vec \gamma} \leq \abs{\vec \alpha}} q_{\vec \alpha, \vec \gamma}(\Phi;\vec y) (\DD^{\vec \gamma} u)(\Phi^{-1}(\vec y)),
   \]
  where each $q_{\vec \alpha,\vec \gamma}$ is a homogeneous polynomial of degree $\abs{\vec \gamma}$ in the derivatives $\DD^{\vec \beta} (\Phi^{-1})$ of order $\vec \beta \in \N_0^{n+1}$ such that $1 \leq \abs{\vec \beta} \leq \abs{\vec \alpha} - \abs{\vec \gamma} - 1$.
  Therefore, in this way the differential operators $\Acal(\vec x, \bb D) = \sum_{\abs{\vec \alpha} \leq 2m} a_{\vec  \alpha}(\vec x) \DD^{\vec \alpha}$ and $\Bcal_j(\vec x, \DD) = \sum_{k=0}^{m_j} \sum_{\abs{\vec \beta} \leq k} b_{j,k,\vec \beta}(\vec x) \DD^{\vec \beta} \Pcal_{j,k}$ will be pushed forward to differential operators of the same order
   \begin{align*}
    \Acal^\Phi(\vec y, \DD)
     &= \sum_{\abs{\vec \alpha} \leq 2m} a^\Phi_{\vec \alpha}(\vec y) \DD^{\vec \alpha},
     \\
    \Bcal_j^\Phi(\vec y, \DD)
     &= \sum_{k=0}^{m_j} \sum_{\abs{\vec \beta} \leq k} b^\Phi_{j,k,\vec \beta}(\vec y) \DD^{\vec \beta} \Pcal_{j,k},
     \quad
     j = 1, \ldots, m.
   \end{align*}
  Alternatively, we may write $\Acal^\Phi(\vec y, \DD) = \Gcal \Acal(\vec x, \DD) \Gcal^{-1}$ for $y \in \tilde U \cap \R_+^{n+1}$ and $\Bcal^\Phi_j(\vec y, \DD) = \Gcal \Bcal^\Phi_j(\vec x, \DD) \Gcal^{-1}$ for $y \in \partial \R_+^{n+1} \cap \tilde U$.
  Since
   \begin{align*}
    &\Acal^\Phi(\vec y, \DD) (\Gcal u)(\vec y)
     \\
     &= \big( \Gcal \sum_{\abs{\vec \alpha} \leq 2m} a_{\vec \alpha}(\vec x) \DD_{\vec x}^{\vec \alpha} \Gcal^{-1} \Gcal v \big)(\vec y)
     = \Gcal \sum_{\abs{\vec \alpha} \leq 2m} a_{\vec \alpha}(\Phi^{-1}(\vec y)) \DD_{\vec x}^{\vec \alpha}|_{\vec x = \Phi^{-1}(\vec y)} u(\Phi^{-1}(\vec y))
     \\
     &= \sum_{\abs{\vec \alpha} \leq 2m} a_{\vec \alpha}(\Phi^{-1}(\vec y)) \big( \DD \Phi (\Phi^{-1} y))^\mathsf{T} \DD_{\vec y} \big)^{\vec \alpha} (\Gcal u)(\vec y),  
   \end{align*}
  the principal part of the interior differential operator is given by
   \[
    \Acal^\Phi_\#(\vec y, \DD)
     = \sum_{\abs{\vec \alpha} = 2m} a_{\vec \alpha}(\Phi^{-1}(\vec y)) \big( \DD \Phi(\Phi^{-1}(\vec y))^\mathsf{T} \DD \big)^{\vec \alpha}
     = \Acal_\#(\Phi^{-1}(\vec y), \DD \Phi(\Phi^{-1}(\vec y))^\mathsf{T} \DD).
   \]
  Similarly, for the boundary operators we obtain that
   \[
    \Bcal^\Phi_j(\vec y, \DD)
     = \Gcal \big( \sum_{k=0}^{m_j} \sum_{\abs{\vec \beta} \leq k} b_{j,k,\vec \beta}(\vec x) \DD_{\vec x}^{\vec \beta} \Pcal_{j,k} \big) \Gcal^{-1}
     = \sum_{k=0}^{m_j} \sum_{\abs{\vec \beta} \leq k} b_{j,k,\vec \beta}(\Phi^{-1}(\vec y)) \big[ \DD \Phi(\Phi^{-1}(\vec y))^\mathsf{T} \DD \big]^{\vec \beta}
   \]
  and from this their principle parts
   \[
    \Bcal^\Phi_{j,\#}(\vec y, \vec \xi)
     = \sum_{k=0}^{m_j} \sum_{\abs{\vec \beta} = k} b_{j,k,\vec \beta}(\Phi^{-1}(\vec y)) \big[ \DD \Phi(\Phi^{-1}(\vec y))^\mathsf{T} \vec \xi \big]^{\vec \beta}
     = \Bcal_{j,\#}(\Phi^{-1}(\vec y), \DD \Phi(\Phi^{-1}(\vec y))^\mathsf{T} \vec \xi).
   \]
  With these relations at hand, we may check that the Lopatinskii--Shapiro condition is unaffected under $\CC^{2m}$-admissible transformations.
  By this, we mean the following.
  The Lopatinskii--Shapiro condition for the original problem reads:
  For every $\lambda \in \Sigma_{\pi-\phi}$, every point $\vec x \in \partial \Omega$ and every $\vec \xi \in \R^n$ such that $\vec \xi \cdot \vec n(\vec x) = 0$ for the outer normal vector $\vec n(\vec x_0)$ at $\vec x_0$, and every $\vec w = (w_1, \ldots, w_m)^\mathsf{T} \in E^m$, the initial boundary value problem
   \[
    \begin{cases}
     (\lambda + \Acal_\#(\vec x, \vec \xi + \ii \vec n \frac{\partial}{\partial \eta}) v(\eta)
      = 0,
      &\eta > 0,
      \\
     \Bcal_{j,\#}(\vec x, \vec \xi + \ii \vec n \frac{\partial}{\partial \eta}) v(0)
      = w_j,
      &j = 1, \ldots, m
    \end{cases}
   \]
  has a unique solution $v = v_{\vec x, \lambda, \vec \xi}(\vec w)$ in the class $v \in \CC_0(\R_+;E)$.
  The Lopatinskii--Shapiro condition for the transformed problem involves the initial value problem
   \[
    \begin{cases}
     (\lambda + \Acal^\Phi_\#(\vec y, \vec \xi', - \ii \frac{\partial}{\partial \sigma}) v^\Phi(\sigma)
      = 0,
      &\sigma > 0,
      \\
     \Bcal^\Phi_{j,\#}(\vec y, \vec \xi', - \ii \frac{\partial}{\partial \sigma}) v^\Phi(\sigma)
      = w^\Phi_j,
      &j = 1, \ldots, m.
    \end{cases}
   \]
  From the previous considerations, this means that
   \[
    \begin{cases}
     (\lambda + \Acal_\#(\Phi^{-1}(\vec y), \DD \Phi(\Phi^{-1}(\vec y))^\mathsf{T} (\vec \xi', \frac{\partial}{\partial \sigma})^\mathsf{T}) v^\Phi(\sigma)
      = 0,
      &\sigma > 0,
      \\
     \Bcal_{j,\#}(\Phi^{-1}(\vec y), \DD \Phi(\Phi^{-1}(\vec y))^\mathsf{T} (\vec \xi', \frac{\partial}{\partial \sigma})^\mathsf{T}) v^\Phi(\sigma)
      = w^\Phi_j,
      &j = 1, \ldots, m,
    \end{cases}
   \]
  where we may write $\DD \Phi(\Phi^{-1}(\vec y))^\mathsf{T} (\vec \xi', - \ii \frac{\partial}{\partial \sigma})^\mathsf{T} = \kappa(\vec y)  \big( \tilde {\vec \xi}(\Phi^{-1}(\vec y), \vec \xi') + \ii \vec n(\Phi^{-1}(\vec y)) \big)$ for some $\kappa (\vec y) > 0$ and $\tilde {\vec \xi} \in \R^n$.
  Since $\Phi$ is an admissible $\CC^{2m}$-diffeomorphism it holds that $\tilde {\vec \xi} \cdot \vec n = 0$.
  Therefore, the solutions to the original (at $\vec x$) and the transformed (at $\vec y = \Phi(\vec x)$) initial value problem are related via
   \[
    w^\Phi_{\vec y, \lambda, \vec \xi'}(\sigma)
     = w_{\Phi^{-1}(\vec y), \lambda, \tilde {\vec \xi}}(\frac{\sigma}{\kappa}),
     \quad \text{that is} \quad
    w_{\vec x, \lambda, \vec \xi}(\eta)
     = w^\Phi_{\Phi(\vec y), \lambda, \vec \xi'}(\kappa \eta),
   \]
  where $\kappa > 0$ is the norm of $\DD \Phi(\Phi^{-1}(\vec y))^\mathsf{T} \vec n = \DD \Phi(\vec x)^\mathsf{T} \vec n$.
  This shows, in particular, that the validity of the Lopatinskii--Shapiro condition is not affected by admissible $\CC^{2m}$-diffeomorphisms.

  Now, we may find such an admissible $\CC^{2m}$-diffeomorphism $\Phi_{\vec x_0}: U_{\vec x_0} \rightarrow \tilde U_{\vec x_0}$ for each boundary point $\vec x_0 \in \partial \Omega$, and moreover, since $U_{\vec x_0}$ is an open neighbourhood of $\vec x_0$, we find radii $r(\vec x_0) > 0$ such that $\Phi_{\vec x_0}^{-1} \big( B_{2r(\vec x_0)}(\vec y_0) \big) \subseteq U_{\vec x_0}$ for each $\vec y_0 = \Phi(\vec x_0)$ and, for some $\varepsilon_0 > 0$ which is given a priori, the coefficients for the principle parts of the pushed-forward operators have locally small oscillations in the sense that
  \begin{align*}
   \sum_{\abs{\vec \alpha}} \norm{a^\Phi_{\vec \alpha}(\vec y) - a^\Phi_{\vec \alpha}(\vec y_0)}
    &< \varepsilon_0,
    \\
   \sum_{k=0}^{m_j} \sum_{\abs{\vec \beta} = k} \norm{b^\Phi_{j,k,\vec \beta}(\vec y) - b^\Phi_{j,k,\vec \beta}(\vec y_0)}
    &< \varepsilon_0
    \quad
    \text{for all } \vec y \in B_{2r(\vec x_0)}(\vec y_0).
  \end{align*}
 Then, trivially
   \[
    \bigcup_{\vec x_0 \in \partial \Omega} \Phi_{\vec x_0}^{-1} \big( B_{r(\vec x_0)}(\vec x_0) \big) \supset \partial \Omega
   \]
  is an open cover for the compact boundary $\partial \Omega$, so that we may choose a finite subcover $\bigcup_{\nu=1}^N U_\nu$ with $U_\nu = \Phi_{x_\nu}^{-1}(B_{r(\vec x_\nu)}(\vec y_\nu))$ for some $\{ \vec x_1, \ldots, x_N \} \subseteq \partial \Omega$ and $\vec y_\nu := \Phi(\vec x_\nu)$, $\nu = 1, \ldots, N$.
  \newline
  For each of these $\nu = 1, \ldots, N$, we may then extend the top-order coefficients $a^\Phi_{\vec \alpha}(\vec y)$ on $B_{r(\vec x_\nu)}(\vec y_\nu) \cap \R_+^{n+1}$ to the half-space $\R_+^{n+1}$ by reflection at the surface $\partial B_{r(\vec x_\nu)}(\vec y_\nu)$, viz.\
   \begin{align*}
    a^\nu_{\vec \alpha}(\vec y)
     := \begin{cases}
      a^{\Phi_{\vec x_\nu}}_{\vec \alpha}(\vec y),
      &\vec y \in \overline{B_{r(\vec x_\nu)}(\vec y_\nu)} \cap \overline{\R_+^{n+1}},
      \\
      a^{\Phi_{\vec x_\nu}}_{\vec \alpha}(\vec y_\nu + \frac{r(\vec x_\nu)^2}{\abs{\vec y - \vec y_\nu}^2} (\vec y - \vec y_\nu)),
      &\vec y \in \overline{\R_+^{n+1}} \setminus \overline{B_{r(\vec x_\nu)}(\vec y_\nu)}
     \end{cases}
     \quad
     \text{for } \vec \alpha \in \N_0^{n+1} \text{ with } \abs{\vec \alpha} = 2m,
  \end{align*}
   and also extend the coefficients of the boundary symbols for all $j = 1, \ldots, m$ and $k = 0, \ldots, m_j$ according to
  \begin{align*}
    b^\nu_{j,k,\vec \beta}(\vec y)
     &= b^{\Phi_{\vec x_\nu}}(\vec y_\nu + \chi(\frac{\vec y - \vec y_\nu}{r(\vec x_\nu)}) (\vec y - \vec y_\nu)),
     \quad
     \vec y \in \R_+^{n+1}
   \end{align*}
  where $\chi \in \CC_0^\infty(\R^n)$ has compact support which is contained in $B_2(\vec 0)$, and is identically $1$ on $\overline{B_1(\vec 0)}$.
  Note that these definitions lead to top order coefficients on $\overline{\R_+^{n+1}}$ which still have small oscillations, i.e.\
   \[
    \sum_{\abs{\vec \alpha} = 2m} \norm{a^\nu_{\vec \alpha}(\vec y) - a^\nu_{\vec \alpha}(\vec y_\nu)},
    \quad
    \sum_{k=0}^{m_j} \sum_{\abs{\vec \beta} = k} \norm{b^\nu_{j,k,\vec \beta}(\vec y) - b^\nu_{j,k,\vec \beta}(\vec y_\nu)} < \varepsilon_0,
    \quad
    j = 1, \ldots, m, \, \vec y \in \overline{\R_+^{n+1}}.
   \]
  For any given $\lambda \in \Sigma_{\pi-\phi}$ and $f \in \LL_p(\Omega;E)$, let us consider the boundary value problem
   \[
    \label{eqn:ast}
    \begin{cases}
     \lambda u + \Acal(\vec x, \DD) u  = f,
      &\text{in } \Omega,
      \\
     \Bcal_j(\vec x, \DD) u = 0,
      &\text{on } \partial \Omega, \, j = 1, \ldots, m.
    \end{cases}
   \]
  We choose a partition of unity $(\varphi_\nu)_{\nu=1,\ldots,M} \subseteq \CC^\infty(\R^n)$ with $M \geq N$ and such that $0 \leq \varphi_\nu \leq 1$ for all $\nu = 1, \ldots, M$, and $\supp \varphi_\nu \subseteq U_\nu$ for $\nu = 1, \ldots, N$ whereas $\supp \varphi_\nu \subseteq \R^n \setminus \partial \Omega$ for $\nu = N+1, \ldots, M$.
  Then a function $u \in \WW_p^{2m}(\Omega;E)$ solves the boundary value problem \eqref{eqn:ast} if and only if for $\nu = 1, \ldots, N$
   \[
    \begin{cases}
     \lambda (\varphi_\nu u) + \varphi_\nu \Acal(\vec x, \DD) u
      = \varphi_\nu f
      &\text{in } \Omega \cap U_\nu,
      \\
     \varphi_\nu \Bcal_j(\vec x, \DD) u
      = 0
      &\text{on } \partial \Omega \cap U_\nu, \, j = 1, \ldots, m
    \end{cases}
   \]
  and for $\nu = N+1, \ldots, M$
   \[
    \lambda (\varphi_\nu u) + \varphi_\nu \Acal(\vec x, \DD) u
     = \varphi_\nu f
     \quad
     \text{on } \Omega.
   \]
  Note that in the latter case no boundary conditions are involved since $\varphi_\nu \equiv 0$ in a neighbourhood of $\partial \Omega$.
  For $\nu = N+1, \ldots, M$, we may write the equation equivalently as
   \[
    \lambda (\varphi_\nu u) + \Acal_\#(\vec x, \DD) (\varphi_\nu u)
     = \varphi_\nu f + \big[ \Acal_\#, \varphi_\nu \big] u - \varphi_\nu (\Acal(\vec x, \DD) - \Acal_\#(\vec x, \DD) ) u
     =: \varphi_\nu f + \Ccal_\nu(\vec x, \DD) u,
   \]
  where the latter terms on the right-hand side include the commutator $[\Acal_\#(\vec x, \DD), \varphi_\nu]$ between the principle part $\Acal_\#(\vec x, \DD)$ of the elliptic differential operator $\Acal(\vec x, \DD)$ and the multiplication operator for multiplication with $\varphi_\nu$, and $\Acal(\vec x, \DD) - \Acal_\#(\vec x, \DD)$ are the lower-order terms of $\Acal(\vec x, \DD)$.
  We may, therefore, for $\nu = N+1, \ldots, M$, denote by $A_\nu$ the $\LL_p(\R_+^{n+1};E)$-realisation of the operator
   \[
    \Acal^\nu(\vec x, \DD)
     = \sum_{\abs{\vec \alpha} = 2m} a^\nu_{\vec \alpha}(\vec x) \DD^\alpha
   \]
  and set $R^\lambda_\nu := (\lambda + A_\nu)^{-1}$ for the resolvent operator, if it exists.
  Then, $u$ has to respect the identity
   \[
    \varphi_\nu u
     = R^\lambda_\nu (\varphi_\nu f) + R^\lambda_\nu \Ccal_\nu(\vec x, \DD) u,
     \quad
     \nu = N+1, \ldots, M.
   \]
  On the other hand, for $\nu = 1, \ldots, N$, we use the local coordinates and the $\CC^{2m}$-diffeomorphisms introduced above and note that the pull-back operators $\Gcal_\nu$ corresponding to the diffeomorphisms $\Phi_\nu = \Phi_{\vec x_\nu}$, $\nu = 1, \ldots, N$, act as continuous operators between the Sobolev spaces $\WW_p^l(U_k;E)$ and $\WW_p^l(\R_+^{n+1};E)$, viz.\
   \[
    \Gcal_\nu:
     \quad
     u \mapsto u^{\Phi_\nu},
     \,
     u^{\Phi_\nu}(\vec y) = u(\Phi_\nu^{-1}(\vec y)),
     \quad
     \WW_p^l(U_\nu;E) \mapsto \WW_p^l(\R_+^{n+1};E),
     \quad
     \nu = 1, \ldots, N,
      \,
     l = 0, 1, \ldots, 2m.
   \]
  We set
   \begin{align*}
    \Acal^\nu(\vec y, \DD)
     &:= \sum_{\abs{\vec \alpha} = 2m} a^\nu_{\vec \alpha}(\vec y) \DD^{\vec \alpha},
     \\
    \Bcal^\nu_j(\vec y, \DD)
     &:= \sum_{k=0}^{m_j} b^\nu_{j,k,\vec \beta}(\vec y) \DD^{\vec \beta} \Pcal_{j,k},
     \quad
     j = 1, \ldots, m,
   \end{align*}
  and rewrite the corresponding boundary value problems for $\varphi_\nu u$, $\nu = 1, \ldots, N$, as
   \[
    \begin{cases}
     \lambda \Gcal_\nu (\varphi_\nu u) + \Acal^\nu(\vec y, \DD) \Gcal_\nu (\varphi_\nu u)
      &= \Gcal_\nu (\varphi_\nu f) + \Gcal_\nu \big( \big[ \Acal(\vec x, \DD), \varphi_\nu \big] - \Gcal_k^{-1} (\Acal^{\Phi_\nu}(\vec y, \DD) - \Acal^\nu(\vec y, \DD)) \Gcal_k \varphi_\nu \big) u,
      \\
      \Bcal^\nu_j(\vec y, \DD) \Gcal_\nu(\varphi_\nu u)
       &= \Gcal_\nu \big( \big[ \Bcal_j(\vec x, \DD), \varphi_\nu \big] - \Gcal_\nu^{-1} (\Bcal^{\Phi_\nu}_j(\vec y, \DD) - \Bcal^\nu_j(\vec y, \DD)) \Gcal_\nu \varphi_\nu \big) u.
    \end{cases}
   \]
  Since we imposed the condition that $b_{j,k,\vec \beta}(\ran \Pcal_{j,k}) \subseteq \ran(\Pcal_{j,k}) + \ldots + \ran(\Pcal_{m_j,k})$, on the right-hand side only lower order differential operators appear, i.e.\ the right-hand side includes only derivatives up to order $2m-1$ resp.\ up to order $k-1$ on every subspace $\ran \Pcal_{j,k}$, $k = 1, \ldots, m_j$.
  Hence, these terms can be considered as a small perturbation to the operators $\Acal^\nu(\vec y, \DD)$ and $\Bcal^\nu_j(\vec y, \DD)$, $j = 1, \ldots, m$, respectively. The solution then will satisfy the implicit relation
   \[
    \Gcal_\nu (\varphi_\nu u)
     = \tilde R^{\lambda,\nu}_0 \big( \Gcal_\nu (\varphi \nu u) + \Gcal_\nu \Ccal^\nu(\vec x, \DD) u \big)
      + \sum_{j=0}^m \sum_{k=0}^{m_j} \tilde S^{\lambda,\nu}_{j,k} \Gcal_\nu \Dcal^\nu_{j,k}(\vec x, \DD) u,
   \] 
  where we write
   \begin{align*}
    \Ccal^\nu(\vec x, \DD) u
     &:= \big[ \Acal(\vec x, \DD), \varphi_\nu \big]  u
      - \Gcal_\nu^{-1} (\Acal^{\Phi_\nu}(\vec y, \DD) - \Acal^\nu(\vec y, \DD)) \Gcal_\nu u,
      \\
     \Dcal^\nu_{j,k} u
      &:= \big[ \Bcal_{j,k}(\vec x, \DD), \varphi_\nu \big]
       - \Gcal_\nu^{-1} (\Bcal_{j,k}^{\Phi_\nu}(\vec y, \DD) - \Bcal^\nu_{j,k}(\vec y, \DD)) \Gcal_\nu u.
   \end{align*}
  From the solution and maximal regularity theory on the half-space we know that we may write
   \begin{align*}
    \tilde S^{\lambda,\nu}_{j,k} \Dcal^\nu_{j,k}(\vec x, \DD)
     &= \big( \tilde S^{\lambda, \nu, \mathrm{I}}_{j,k} + \tilde S^{\lambda, \nu, \mathrm{II}}_{j,k} \big) \Dcal^\nu_{j,k}(\vec x,\DD)
     \\
     &= \tilde T^{\lambda,\nu}_{j,k} (\abs{\lambda} + (- \Delta)^{2m})^{(2m-k)/2m} \Dcal^\nu_{j,k}(\vec x,\DD)
      \\ &\quad
      + \tilde R^{\lambda,\nu}_{j,k} (\abs{\lambda} + (- \Delta)^{2m})^{(2m-k-1)/2m} \partial_{n+1} \Dcal^\nu_{j,k}(\vec x,\DD)
   \end{align*}
  where the operators $\tilde T^{\lambda,\nu}_{j,k}$ and $R^{\lambda,\nu}_{j,k}$ satisfy $\Rcal$-bounds
   \begin{align*}
    \Rcal ( \{ \abs{\lambda}^{1 - \frac{\abs{\vec \alpha}}{2m}} \DD^{\vec \alpha} \tilde T^{\lambda, \nu}_{j,k}: \, \lambda \in \Sigma_{\pi-\phi}, \, \abs{\lambda} \geq \lambda_0, \, \abs{\vec \alpha} \leq 2m \} )
     &< \infty,
     \\
    \Rcal ( \{ \abs{\lambda}^{1 - \frac{\abs{\vec \alpha}}{2m}} \DD^{\vec \alpha} \tilde R^{\lambda, \nu}_{j,k}: \, \lambda \in \Sigma_{\pi-\phi}, \, \abs{\lambda} \geq \lambda_0, \, \abs{\vec \alpha} \leq 2m \} )
     &< \infty,
   \end{align*}
  for each fixed $j = 1, \ldots, m$, $k = 0, 1, \ldots, m_j$ and $\nu = 1, \ldots, N$.
  However, the set of admissible indices $(j,k,\nu)$ is finite, which implies that
   \begin{align*}
    \Rcal ( \{ \abs{\lambda}^{1 - \frac{\abs{\vec \alpha}}{2m}} \DD^{\vec \alpha} \tilde T^{\lambda, \nu}_{j,k}: \, \lambda \in \Sigma_{\pi-\phi}, \, \abs{\lambda} \geq \lambda_0, \, \abs{\vec \alpha} \leq 2m, \, 1 \leq j \leq m, \, 0 \leq k \leq m_j, \, 1 \leq \nu \leq N \} )
     &< \infty,
     \\
    \Rcal ( \{ \abs{\lambda}^{1 - \frac{\abs{\vec \alpha}}{2m}} \DD^{\vec \alpha} \tilde R^{\lambda, \nu}_{j,k}: \, \lambda \in \Sigma_{\pi-\phi}, \, \abs{\lambda} \geq \lambda_0, \, \abs{\vec \alpha} \leq 2m, \, 1 \leq j \leq m, \, 0 \leq k \leq m_j, \, 1 \leq \nu \leq N \} )
     &< \infty.
   \end{align*}
  Introducing
   \[
    T^{\lambda,\nu}_{j,k} := \Gcal_\nu^{-1} \tilde T^{\lambda,\nu}_{j,k} \Gcal_\nu
     \quad \text{and} \quad
    R^{\lambda,\nu}_{j,k} := \Gcal_\nu^{-1} \tilde R^{\lambda,\nu}_{j,k} \Gcal_\nu
   \]
  and
   \begin{align*}
    \Dcal^{\nu,I}_{j,k}(\vec x, \DD)
     &:= \Gcal_\nu^{-1} \big( (-\Delta)^{2m} + \abs{\lambda} \big)^{(2m-k)/2m} \Gcal_\nu \Dcal^\nu_{j,k}(\vec x, \DD),
     \\
    \Dcal^{\nu,II}_{j,k}(\vec x, \DD)
     &:= \Gcal_\nu^{-1} \big( (-\Delta) + \abs{\lambda} \big)^{(2m-k-1)/2m} \partial_{n+1} \Gcal_\nu \Dcal^\nu_{j,k}(\vec x, \DD)
   \end{align*}
  by pull-back we get (for $\nu = 1, \ldots, N$) the identities
   \begin{align*}
    \varphi_\nu u
     &= R^{\lambda,\nu}_0 \big( \varphi \nu u + \Ccal^\nu(\vec x, \DD) u \big)
      + \sum_{j=0}^m \sum_{k=0}^{m_j} \tilde S^{\lambda,\nu}_{j,k} \Dcal^\nu_{j,k}(\vec x, \DD) u
     \\
     &= R^{\lambda,\nu}_0 \big( \varphi \nu u + \Ccal^\nu(\vec x, \DD) u \big)
      + \sum_{j=0}^m \sum_{k=0}^{m_j} T^{\lambda,\nu}_{j,k} \Dcal^{\nu,I}_{j,k}(\vec x, \DD) u + R^{\lambda,\nu}_{j,k} \Dcal^{\nu,II}_{j,k}(\vec x, \DD) u.
   \end{align*}
  Summing over $\nu = 1, \ldots, M$, we arrive at
   \begin{align*}
    (\lambda + A_B)^{-1} f
     &= u
     = \sum_{\nu=1}^M \phi_\nu u
     \\
     &= \sum_{\nu=1}^M R^{\lambda,\nu} (\phi_\nu f + \Ccal^\nu(\vec x, \DD) u)
      + \sum_{\nu=1}^N \sum_{j=1}^m \sum_{k=0}^{m_j} (T^{\lambda,\nu}_{j,k} \Dcal^{\nu,I}_{j,k}(\vec x, \DD) + R^{\lambda,\nu}_{j,k} \Dcal^{\nu,II}_{j,k}(\vec x, \DD)) u
      \\
     &= R(\lambda) f + S(\lambda) u,
   \end{align*}
  where we set
   \[
    R(\lambda)
     = \sum_{\nu = 1}^M R^{\lambda,\nu}_0 \phi_\nu,
     \quad
    S(\lambda)
     = \sum_{\nu = 1}^N \sum_{j=1}^m \sum_{k=0}^{m_j} (T^{\lambda,\nu}_{j,k} \Dcal^{\lambda,\nu}_{j,k} + R^{\lambda,\nu} \Dcal^{\lambda,\nu}_{j,k}).
   \]
  Thus, provided $\norm{S(\lambda)} < 1$, we may solve for the resolvent operator by employing the Neumann series
   \[
    (\lambda + A_B)^{-1}
     = (I - S(\lambda))^{-1} R(\lambda) f
     = \sum_{l = 0}^\infty S(\lambda)^l R(\lambda) f.
   \]
  Here, each summand $S(\lambda)^l$, $l \geq 0$, is a sum of products of the form
   \[
    U_\lambda^0 \prod_{\eta = 1}^l V_\lambda^\eta U_\lambda^\eta,
   \]
  where $U_\lambda^\eta \in \{ R^{\lambda,\nu}, R^{\lambda,\nu}_{j,k}, T^{\lambda,\nu}_{j,k}: \, j = 1, \ldots, m, \, k = 0, \ldots, m_j \}$ is one of the resolvent or solution operators for the BVP, and $V_\lambda^\eta \in \{ \Ccal^\nu(\vec x, \DD), \Dcal^{\nu,I}_{j,k}, \Dcal^{\nu,II}_{j,k}: \, j = 1, \ldots, m, \, k = 0, \ldots, m_j \}$ are all lower-order operators, such that any product $U_\lambda^\eta V_\lambda^\eta$ is, in fact, a compact operator on $X$.
  Finally, we may employ Remark \ref{rem:Rcal-boundedness-properties} to conclude that for sufficiently small $\varepsilon > 0$ and sufficiently large $\lambda_0 \geq 0$, the family of operators given by such Neumann series $\{ \sum_{l=0}^\infty S(\lambda)^l R(\lambda): \, \lambda \in \Sigma_{\pi-\phi}, \, \abs{\lambda} \geq \lambda_0 \}$ not only is well-defined, but actually $\Rcal$-bounded.
  By Theorem \ref{thm:characterisation_Lp-max}, this implies that the abstract Cauchy problem
   \[
    \begin{cases}
     \frac{\dd}{\dd t} u(t) + A_B u(t)
      = f(t),
      &t \geq 0,
      \\
     u(0)
      = 0
    \end{cases}
   \]
  has $\LL_q$-maximal regularity for all $q \in (1, \infty)$.
 \end{proof}
  
 \section{Inhomogeneous initial data and anisotropic Sobolev spaces}
 \label{Sec:Inhomogeneous_IV}
 
 In \cite{DeHiPr07}, the authors extended their results in \cite{DeHiPr03} to the case of inhomogeneous initial data and anisotropic Sobolev spaces, i.e.\ distinct integrability parameters $p, q \in (1, \infty)$ for the time and spatial variable, resp., by identifying the optimal spaces for the data to obtain optimal $\LL_p$-$\LL_q$-estimates.
 These results again heavily rely on the results and techniques as in \cite{DeHiPr03}, especially vector-valued multiplier theorems on Banach spaces with property $\HTcal$ (i.e.\ UMD-spaces), Gagliardo-Nirenberg type estimates as well as optimal embedding results for Sobolev, Besov and Triebel-Lizorkin spaces. The methods used in \cite{DeHiPr07} translate more or less directly to the situation considered here: Obviously, the optimal regularity spaces have to be adjusted since the boundary operators have different orders $k$ in the subspaces $\ran (\Pcal_{j,k})$ corresponding to the projections $\Pcal_{j,k}$.
 We introduce the following assumptions for this section.
   \begin{enumerate}
    \item[{\bf (D)}]
     Assumptions on the data:
      \begin{enumerate}
       \item[(i)]
        $f \in \LL_p(J \times \Omega; E)$, where $J = [0,T]$ for some $T > 0$;
       \item[(ii)]
        $g_{j,k} \in \WW_p^{(1,2m) \cdot \kappa_k} (J \times \partial \Omega;\ran \Pcal_{j,k})$, where $\kappa_k = \frac{2m - k - \nicefrac{1}{p}}{2m}$, and we then set $g_j = \sum_{k=0}^{m_j} g_{j,k}$;
       \item[(iii)]
        $u_0 \in \BB_{p,p}^{2m(1-\nicefrac{1}{p})}(\Omega;E)$;
       \item[(iv)]
        if $\kappa_k > \nicefrac{1}{p}$, then $\Bcal_j(0,\vec x) \Pcal_{j,k} u_0(\vec x) = g_{j,k}(0,\vec x)$ for $\vec x \in \partial \Omega$.
      \end{enumerate}
    \item[{\bf (SD)}]
     Regularity assumptions on the interior symbol:
     \newline
     There are $r_l, s_l \geq p$ with $\frac{1}{s_l} + \frac{n}{2m r_l} < 1 - \frac{l}{2m}$ such that
      \begin{align*}
       a_{\vec \alpha}
        &\in \LL_{s_l}(J; [\LL_{r_l} + \LL_\infty](\Omega;\B(E))),
        &&\abs{\vec \alpha} = l < 2m,
        \\
       a_{\vec \alpha}
        &\in \CC_\mathrm{l}(J \times \overline{\Omega}; \B(E)),
        &&\abs{\vec \alpha} = 2m.
      \end{align*}
    \item[{\bf (SB)}]
     Regularity assumptions on the boundary symbols:
     \newline
     There are $s_{j,k,l}, r_{j,k,l} \geq p$ such that $\frac{1}{s_{j,k,l}} + \frac{n-1}{2m r_{j,k,l}} < \kappa_k + \frac{l-k}{2m}$ and
      \[
       b_{j,k,\vec \beta} \in \WW_{s_{j,k,l},r_{j,k,l}}^{(1,2m) \cdot \kappa_k}(J \times \partial \Omega; \B(E)),
       \quad
       \text{for each } \vec \beta \in \N_0^{n+1} \text{ of length }
       \abs{\vec \beta} = l \leq k \leq m_j.
      \]
     Moreover, $b_{j,k,\vec \beta}(t, \vec x)(\ran \Pcal_{j,k}) \subseteq \ran(\Pcal_{j,k}) + \ldots + \ran(\Pcal_{j,m_j})$ for a.e.\ $(t, \vec x) \in J \times \partial \Omega$ and $j = 1, \ldots, m$, $k = 0, 1, \ldots, m_j$ and $\vec \beta \in \N_0^{n+1}$ such that $\abs{\vec \beta} \leq k$.
    \item[{\bf (E)}]
     Ellipticity of the interior symbol:
     For all $t \in J$, $\vec x \in \overline{\Omega}$ and $\vec \xi \in \S^{n-1}$, the spectrum of the principle part of the interior symbol lies in the open right-half plane,
      \[
       \sigma(\Acal_\#(t,\vec x, \vec \xi)) \subseteq \C_0^+,
      \]
     i.e.\ $\Acal(t,\vec x,\DD)$ is \emph{normally elliptic}.
     If $\Omega$ is unbounded, the same condition is imposed at $\vec x = \infty$ for $\Acal_\#(t,\infty,\vec \xi) := \lim_{\abs{\vec x} \rightarrow \infty} \sum_{\abs{\vec \alpha} = 2m} a_{\vec \alpha}(t,\vec x, \vec \xi) \vec \xi^{\vec \alpha}$.
    \item[{\bf (LS)}]
     Lopatinskii-Shapiro condition:
     For all $t \in J$, $\vec x \in \partial \Omega$ and all $\vec \xi \in \R^n$ with $\vec \xi \cdot \vec n(\vec x) = 0$, and all $\lambda \in \overline{\C_0^+}$ such that $(\lambda, \vec \xi) \neq (0, \vec 0)$, the initial value problem
      \begin{align*}
       \lambda v(y) + \Acal_\#(t,\vec x,\vec \xi + \ii \vec n(\vec x) \tfrac{\partial}{\partial y})v(y)
        &= 0,
        &&y > 0,
        \\
       \Bcal_{j,\#}(t,\vec x, \vec \xi + \ii \vec n(\vec x) \frac{\partial}{\partial y})v(0)
       &= h_j,
       &&j = 1, \ldots, m,
      \end{align*}
     has a unique solution in the class $v \in \CC_0(\R_+;E)$.
   \end{enumerate}
 The theorem on $\LL_p$-$\LL_p$-optimal regularity then reads as follows.

 \begin{theorem}[$\LL_p$--$\LL_p$-optimal regularity]
  \label{thm:Lp-Lp-Optimal_Regularity}
  Let $E$ be a Banach space of class $\HTcal$ and $\Omega \subseteq \R^n$ be a domain with compact boundary of class $\partial \Omega \in \CC^{2m}$.
  Let $p \in (1, \infty)$ and suppose that assumptions {\bf (E)}, {\bf (LS)}, {\bf (SD)} and {\bf (SB)} hold true. Then the inhomogeneous parabolic initial-boundary value problem 
   \begin{align}
    \partial_t u + \Acal(t,\vec x,\DD)u
     &= f(t,\vec x),
     &&t \in J, \, \vec x \in \Omega,
     \nonumber \\
    \Bcal_j(t,\vec x,\DD)u
     &= g_j(t,\vec x),
     &&t \in J, \, \vec x \in \partial \Omega, \, j = 1, \ldots, m,
     \tag{\textrm{iIBVP}}
     \label{eqn:iIBVP}
     \\
    u(0,\vec x)
     &= u_0(\vec x),
     &&\vec x \in \Omega
     \nonumber
   \end{align}
  has a unique solution in the class
   \[
    u \in \WW_p^{(1,2)}(J \times \Omega; E)
   \]
  if and only if the data $f$, $\vec g$ and $u_0$ are subject to condition {\bf (D)}.
 \end{theorem}
 
 We refrain from giving a proof, but only refer to the proof of \cite[Theorem 2.1]{DeHiPr07} which can be modified in a similar manner as the proof of Theorem \ref{thm:Lp-Lq-optimality} below.
 \\
 For $\LL_p$--$\LL_q$-estimates with $p \neq q$, where $p \in (1, \infty)$ corresponds to the integrability parameter w.r.t.\ time $t \geq 0$ and $q \in (1, \infty)$ reflects integrability w.r.t.\ the spatial variable $\vec x \in \Omega$, a similar result holds true, but the smoothness conditions on the coefficients and on the data have to be slightly modified, cf.\ the corresponding result \cite[Theorem 2.3]{DeHiPr07} by Denk, Hieber and Pr\"uss.
  \begin{enumerate}
   \item[{\bf (D1)}]
    Assumptions on the data in case the $p \neq q$:
     \item[(i)]
      $f \in \LL_p(J;\LL_q(\Omega;E))$,
     \item[(ii)]
      $g_{j,k} \in \FF_{p,q}^{\kappa_k}(J;\LL_q(\partial \Omega;\ran \Pcal_{j,k})) \cap \LL_p(J;\BB_{q,q}^{2m\kappa_k}(\partial \Omega;\ran \Pcal_{j,k}))$ with $\kappa_k = \frac{2m-k-1/q}{2m}$, and we then set $g_j = \sum_{k=0}^{m_j} g_{j,k}$,
     \item[(iii)]
      $u_0 \in \BB_{q,p}^{2m(1-\nicefrac{1}{p})}(\Omega;E)$,
     \item[(iv)]
      if $\kappa_k > 1/q$, then $\Bcal_j(0,\vec x) \Pcal_{j,k}u_0(\vec x) = g_{j,k}(0,\vec x)$ for $\vec x \in \partial \Omega$, for each $j = 1, \ldots, m$ and $k = 0, 1, \ldots, m_j$.
   \item[{\bf (SD1)}]
    There are $s_l \geq p$ and $r_l \geq q$ with $\frac{1}{s_l} + \frac{n}{2m r_l} < 1 - \frac{l}{2m}$ such that
     \begin{align*}
      a_{\vec \alpha}
       &\in \LL_{s_l}(J; (\LL_{r_l} + \LL_\infty)(\Omega;\B(E)),
       &&\abs{\vec \alpha} = l < 2m,
       \\
      a_{\vec \alpha}
       &\in \CC_\mathrm{l}(J \times \overline{\Omega}; \B(E)),
       &&\abs{\vec \alpha} = 2m.
     \end{align*}
   \item[{\bf (SB1)}]
    There are $s_{j,k,l} \geq p$ and $r_{j,k,l} \geq q$ with $\frac{1}{s_{j,k,l}} + \frac{n-1}{2m r_{j,k,l}} < \kappa_k + \frac{k-l}{2m}$ such that
     \[
      b_{j,k,\vec \beta}
       \in \WW_{s_{j,k,l},r_{j,k,l}}^{(1,2m) \cdot \kappa_k}(J \times \partial \Omega; \B(E)))
       \quad
       \text{for each } \vec \beta \in \N_0^{n+1} \text{ of length }
       \abs{\vec \beta} = l \leq k \leq m_j.
     \]
    Moreover, $b_{j,k,\vec \beta}(t, \vec x)(\ran \Pcal_{j,k}) \subseteq \ran(\Pcal_{j,k}) + \ldots + \ran(\Pcal_{j,m_j})$ for a.e.\ $(t, \vec x) \in J \times \partial \Omega$ and $j = 1, \ldots, m$, $k = 0, 1, \ldots, m_j$ and $\vec \beta \in \N_0^n$ such that $\abs{\vec \beta} \leq k$.
  \end{enumerate}
  
  \begin{theorem}[$\LL_p$--$\LL_q$-optimal regularity]
  \label{thm:Lp-Lq-optimality}
  Let $E$ be a Banach space of class $\HTcal$ and $\Omega \subseteq \R^n$ be a domain with compact boundary of class $\partial \Omega \in \CC^{2m}$.
  Let $p,q \in (1, \infty)$ and suppose that assumptions {\bf (E)}, {\bf (LS)}, {\bf (D1)}, {\bf (SD1)} and {\bf (SB1)} hold true. Then the problem \eqref{eqn:iIBVP} has a unique solution in the class
   \[
    u \in \WW_{p,q}^{(1,2)}(J \times \Omega; E)
   \]
  if and only if the data $f$, $\vec g$ and $u_0$ are subject to conditions {\bf (D1)}.
  \end{theorem}
  
  \textbf{Proof.}
  Let us sketch the strategy which has been employed in \cite{DeHiPr07} and comment on adjustments that are necessary for the proof to carry over to the situation considered here.
  As for the case of homogeneous initial data, i.e.\ $\vec u_0 = \vec 0$, the basic strategy is to consider step by step the full-space, the half-space and finally via localisation the general domain problem.
  As the full-space problem obviously does not involve any boundary conditions, by \cite[Proposition 6.1]{DeHiPr07}, we know that the full-space problem with constant coefficients admits a unique solution of \ref{eqn:iIBVP} in the class $\WW^{(1,2m)}_{p,q}(\R_+ \times \R^n;E)$ if and only if $f \in \LL_{p,q}(\R_+ \times \R^n;E)$ and $u_0 \in \BB_{q,p}^{2m (1 - \frac{1}{p})}(\R^n;E)$.
  For the next step, the initial-boundary value problem on the half-space, \cite[Proposition 6.4]{DeHiPr07} and its proof can easily be modified as follows:
 \begin{proposition}
 \label{prop:inhomogeneous-pq-half-space}
   Assume that the BVP associated to the operators $(\Acal(\DD), \Bcal_1(\DD), \ldots, \Bcal_m(\DD))$ satisfy conditions {\bf (E)} and {\bf (LS)} and let $\mu > 0$.
   Then the parabolic initial-boundary value problem
    \begin{align}
     (\mu + \partial_t + \Acal(\DD)) u
      &= f
      &&\text{in } \R_+ \times \R_+^n,
      \nonumber \\
     \Bcal_j(\DD) u
      &= g_j,
      &&\text{in } \R_+ \times \R^{n-1},
      \label{eqn:IBVP-half-space}
      \\
     u(t,\cdot)
      &= u_0
      &&\text{in } \R_+^n
    \end{align}
   has a (unique) solution in the class $\WW_{p,q}^{(1,2m)}(\R_+ \times \R_+^n;E)$ if and only if $f \in \LL_{p,q}(\R_+ \times \R_+^n;E)$, $u_0 \in \BB_{q,p}^{2m(1-\frac{1}{p})}(\R_+^n;E)$, $g_{j,k} = \Pcal_{j,k} g_j \in \FF_{p,q}^{1-\frac{1}{p}-\frac{k}{2m}}(\R_+; \LL_q(\R^{n-1};E)) \cap \LL_p(\R_+;\BB_{q,q}^{2m(1-\frac{1}{p})-k}(\R^{n-1};E))$ and satisfy the compatibility conditions
    \[
     \Bcal_{jk}(\DD) u_0|_{y=0}
      = g_{j,k}|_{t=0}
      \quad
      \text{whenever }
      (1-\frac{1}{q}) - \frac{k}{2m} > \frac{1}{q}.
    \]
 \end{proposition}
 \textbf{Proof.}
 \textit{Necessity of the regularity of the data:}
 The necessity part can again be splitted into necessary spatial regularity and necessary of the time regularity.
 For the spatial regularity, we may follow the lines of the proof for \cite[Proposition 6.4]{DeHiPr07} where by means of \cite[Proposition 6.1]{DeHiPr07} and taking boundary resp.\ time trace of a presupposed solution $u \in \WW_{p,q}^{(1,2m)}(\R_+ \times \R_+^n;E)$, we may conclude that $u_0 \in \BB_{q,p}^{2m(1- \frac{1}{p})}(\R_+^n;E)$ and $u|_{y=0} \in \LL_p(\R_+;\BB_{q,q}^{2m-\frac{1}{q}}(\partial \R_+^n;E))$, so that each $g_{j,k} = \Bcal_{jk}(\DD)u_0|_{y=0} \in \LL_p(\R_+; \BB_{q,q}^{2m-k-\frac{1}{q}}(\R^{n-1};E))$, $j = 1, \ldots, m$ and $k = 0, 1, \ldots, m_j$.
 \newline
 The necessity of the time regularity can be observed based on the important fact that since $E$ is a Banach space of class $\HTcal$ and $p \in (1, \infty)$ lies in the reflexive range, then the auxiliary space $\tilde E = \LL_q(\R^{n-1};E)$ is of class $\HTcal$ as well, see Lemma \ref{lem:HTcal-L_p-HTcal}.
 Moreover, since $q \in (1, \infty)$ we may identify $\LL_q(\R_+^n;E)$ with $\LL_q(\R_+;\tilde E)$ and $\WW_q^{2m}(\R_+^n;E)$ with a subspace of $\WW_q^{2m}(\R_+;\tilde E)$ via the relation $u(\vec x', y) = (u(\cdot, y))(\vec x')$.
 In this sense, $u \in \WW_{p,q}^{(1,2m)}(\R_+ \times \R_+^n;E)$ implies that $u \in \WW_p^1(\R_+;\tilde E)$.
 Considering the (resolvent-commutative) operators $A = \partial_y$ with $\dom(A) = \LL_p(\R_+;\WW_q^1(\R_+;\tilde E))$ and $B = (\partial_t)^{1/2m}$ on $\dom(B) = \mathring \WW_p^{1/2m}(\R_+;\LL_q(\R_+;\tilde E))$, after reducing the problem to the special case $u_0 = 0$, as in \cite{DeHiPr07} it follows for the functions $u_{\vec \alpha} := B^{2m-\abs{\vec \alpha}-1} \DD^{\vec \alpha} u$ that $u_{\vec \alpha}|_{y=0} \in \mathring F_{p,q}^{\frac{1}{2m}(1-\frac{1}{q})}(\R_+;\tilde E)$, thus $\DD^{\vec \alpha} u|_{y=0} \in \mathring \FF_{p,q}^{1-\frac{\abs{\vec \alpha}}{2m} - \frac{1}{2mq}}(\R_+;\tilde E) = \{ u \in \FF_{p,q}^{1-\frac{\abs{\vec \alpha}-\frac{1}{2mq}}{2m}}(\R_+;\tilde E): \, u(0) = 0 \}$ for each multi-index $\vec \alpha \in \N_0^n$ of length $\abs{\vec \alpha} \leq 2m-1$, hence, we find that $g_{j,k} = \Bcal_{jk}(\DD)u|_{y=0} \in \mathring \FF_{p,q}^{1-\frac{k}{2m}-\frac{1}{2mq}}(\R_+;\tilde E)$.
 \newline
 \textit{Sufficiency of the regularity of the data:}
 Similar as we might have done it for the necessity of the time regularity, we demonstrate how to reduce the problem to the zero-initial value problem: Consider the zero-extension operator in space \[ \Ecal_0 \in \B(\LL_p(\R_+;\LL_q(\R_+^n;E)),\LL_p(\R_+;\LL_q(\R^n;E))\] and any continuous extension operator $\Ecal \in \B(\BB_{q,p}^{2m(1-\frac{1}{p})}(\R_+^n;E),\BB_{q,p}^{2m(1-\frac{1}{p})}(\R^n;E))$, e.g.\ defined via higher order reflections.
 By \cite[Proposition 6.1]{DeHiPr07}, for any given data $f \in \LL_p(\R_+;\LL_q(\R_+^n))$ and $u_0 \in \BB_{q,p}^{2m(1-\frac{1}{p})}(\R_+^n;E)$, there is a (unique) solution $u_1 \in \WW_{p,q}^{(1,2m)}(\R_+ \times \R^n;E)$ to the full-space initial-value problem
  \begin{align*}
   (\mu + \partial_t + \Acal(\DD)) u_1
    &= \Ecal_0 f
    &&\text{in } \R_+ \times \R^n,
    \\
   u_1(0,\cdot)
    &= \Ecal u_0
    &&\text{in } \R^n.
  \end{align*}
 We now set $u_2 = u - u_1$, which reduces the problem to finding a function $u_2 \in \WW_{p,q}^{(1,2m)}(\R_+ \times \R_+^n)$ such that
  \begin{align}
   (\mu + \partial_t + \Acal(\DD)) u_2
    &= 0
    &&\text{in } \R_+ \times \R_+^n,
    \label{eqn:aux-1}
    \\
   \Bcal_j(\DD) u_2
    &= \tilde g_j
    &&\text{in } \R_+ \times \R_+^{n-1}, \, j = 1, \ldots, m,
    \\
   u_2(0,\cdot)
    &= 0
    &&\text{in } \R_+^n,
    \label{eqn:aux-2}
  \end{align}
 where we have set $\tilde g_j = \sum_{k=0}^{m_j} \Pcal_{j,k} \tilde g_j = \sum_{k=0}^{m_j} \tilde g_{j,k}$ for adjusted boundary data defined by
  \[
   \tilde g_{j,k}
    = g_{j,k} - \Bcal_j(\DD) u_1
    \in \FF_{p,q}^{(1-\frac{1}{q})-\frac{k}{2m}}(\R_+; \LL_q(\R^{n-1};E)) \cap \LL_p(\R_+; \BB_{q,q}^{2m(1-\frac{1}{q})-k}(\R^{n-1};E)).
  \]
 The compatibility conditions on the boundary and initial data tell us that $\tilde g_{j,k}(0,\cdot) = 0$, whenever the time trace exists, and, therefore, we may simply extend $\tilde g_{j,k}$ by $0$ (for time $t < 0$) to a function (still denoted by $\tilde g_{j,k}$) in the class $\LL_p(\R_+;\BB_{q,q}^{2m(1-\frac{1}{q})-k}(\R^{n-1};E))$.
 Similar to the procedure in \cite{DeHiPr07}, we next set
  \[
   h_{j,k}(t,\vec x',y)
    := L^{\frac{2m-k}{2m}} \ee^{- y L^{1/2m}} \tilde g_{j,k}(t,\vec x')
  \]
 for the parabolic operator
  \begin{equation}
   L
    = \mu + \partial_t + (- \Delta)^m
    \text{ defined on }
   \dom(L)
    = \mathring \WW_p^1(\R_+;\LL_q(\R^{n-1};E)) \cap \LL_p(\R_+;\WW_q^{2m}(\R^{n-1};E)).
    \label{eqn:def-L}
  \end{equation}
 Using the solution operators $T_{j,k} = \Lcal^{-1} T^\lambda_{j,k} \Lcal \in \B(\LL_p(\R_+;\LL_q(\R^n;E)), \WW_{p,q}^{(1,2)}(\R_+ \times \R_+^n;E))$ for the corresponding parabolic IBVP, we find the solution to the system \eqref{eqn:aux-1}--\eqref{eqn:aux-2} to be
  \[
   u_2
    = \sum_{j=1}^m \sum_{k=0}^{m_j} T_{j,k} h_{j,k}
    \in \WW_{p,q}^{(1,2m)}(\R_+ \times \R_+^n;E),
  \]
 where, by the proof of \cite[Proposition 6.4]{DeHiPr07}, the pseudo-differential operator $L^{\frac{2m-k}{2m}}$ maps continuously from $\mathring W_p^{1-\frac{k}{2m}}(\R_+;\LL_q(\R_+^n;E)) \cap \LL_p(\R_+;\WW_q^{2m-k}(\R_+;E))$ into $X := \LL_{p,q}(\R_+ \times \R_+^n;E)$.
 We are, thus, done, provided we can establish the following adjusted version of \cite[Proposition 4.5]{DeHiPr07}.
 \begin{proposition}
  Let $p, q \in (1, \infty)$, $\mu > 0$ and $E$ be a Banach space of class $\HTcal$.
  Assume that the differential operators $(\Acal(\DD), \Bcal_1(\DD), \ldots, \Bcal_m(\DD))$ satisfy conditions {\bf (E)}, {\bf (LS)}, {\bf (D1)}, {\bf (SB1)} and {\bf (SD1)}.
  Let \[ g_{j,k} \in \WW_{p,q}^{(1,2m) \cdot (1-\frac{1}{q})\frac{k}{2m}}(\R_+ \times \R^{n-1}; \ran(\Pcal_{j,k})) \] for $j = 1,\ldots, m$ and such that $g_{j,k}(0,\cdot) = 0$ whenever $(1-\frac{1}{q})\frac{k}{2m} > \frac{1}{p}$.
  Then the initial-boundary value problem
   \begin{align*}
    (\mu + \partial_t + \Acal(\DD)) u
     &= 0
     &&\text{in } \R_+ \times \R_+^n,
     \\
    \Bcal_j(\DD) u
     &= g_j
     = \sum_{k=0}^{m_j} g_{j,k}
     &&\text{on } \R_+ \times \partial \R_+^n,
     \tag{hIBVP}
     \label{eqn:hIBVP}
     \\
    u(0,\cdot)
     &= 0
     &&\text{on } \R_+^n
   \end{align*}
  admits a unique solution in the class $u \in \WW_{p,q}^{(1,2m)}(\R_+ \times \R_+^n; E)$, which continuously depends on $\vec g$.
 \end{proposition}
\begin{proof}
 Since by assumption $g_{j,k}(0,\cdot) = 0$ whenever $\kappa_k := (1-\frac{1}{q})\frac{k}{2m} > \frac{1}{p}$, we may extend each function $g_{j,k}$ for $t < 0$ by zero (still denoted by $g_{j,k}$) to get $g_{j,k} \in \WW_{p,q}^{(\kappa_k, 2m \kappa_k)}(\R \times \R^{n-1};E)$.
 For $L = \mu + \partial_t + (-\Delta_{\R^{n-1}})^m$ as in \eqref{eqn:def-L}, we then find that $h_{j,k}(t, \vec x', y) := L^{\frac{2m-k}{2m}} \ee^{-y L^{1/2m}} g_{j,k}(t, \vec x')$ defines a function $h_{j,k} \in \LL_{p,q}(\R \times \R_+^n; \ran(\Pcal_{j,k}))$, so that by Proposition \ref{prop:Estimates_Solution_Operator} (an adjusted version of \cite[Lemma 4.3]{DeHiPr07}), the unique solution is given by $u = \sum_{j=1}^m \sum_{k=0}^{m_j} \Lcal^{-1} T^\lambda_{j,k} \Lcal h_{j,k}$ for
  \[
   (T^\lambda_{j,k} \tilde h_{j,k})(\vec x', y)
    := \int_0^\infty \int_{\R^{n-1}} T_{j,k}(\lambda, \vec x' - \vec {\tilde x}', y + \tilde y) \tilde h_{j,k}(\vec {\tilde x}', \tilde y) \, \dd \vec {\tilde x}' \, \dd \tilde y,
  \]
 where we recall that $T_{j,k}(\lambda, \vec x', y) = \big( \frac{\partial}{\partial y} K_{j,k}^\lambda - K_{j,k}^\lambda L_\lambda^{1/2m} \big)(\vec x', y)$.
 It remains to show that $u$ lies in the class $\WW_{p,q}^{(1,2m)}(\R_+ \times \R_+^n;E)$ provided $h_{j,k} = L^{\frac{2m-k}{2m}} \ee^{-y L^{1/2m}} g_{j,k} \in \mathring F_{p,q}^{\frac{1}{2m} (1 - \frac{1}{q})}(\R_+;\LL_q(\R_+^n;E)) \cap \LL_p(\R_+;\BB_{q,q}^{1-\frac{1}{q}}(\R^n;E))$.
 By $\LL_p$-maximal regularity of the operator $\partial_y + L^{1/2m}$, it suffices to prove that the map given by $\big( (t, (\vec x',y)) \mapsto L^{-k/2m} \ee^{-y L^{1/2m}} g(t,\vec x') \big)$ lies in $\LL_{p,q}(\R_+ \times \R_+^n)$, if \[ g \in \mathring F_{p,q}^{\frac{1}{2m} (1 - \frac{1}{q})}(\R_+;\LL_q(\R^{n-1};E)) \cap \LL_p(\R_+;\BB_{q,q}^{1-\frac{1}{q}}(\R^{n-1};E)).\]
 To show this property, we employ \cite[Lemma 6.2]{DeHiPr07} which shows that for such $g$, the functions defined by $\big( (t, (\vec x', y)) \mapsto B \ee^{- y B} g(t, \vec x') \big)$ and $\big( (t, (\vec x', y)) \mapsto C \ee^{- y C} g(t, \vec x') \big)$ belong to the space $\LL_{p,q}(\R_+ \times \R_+^{n+1};E)$, where $B = (\partial_t)^{1/2m}$ denotes the fractional time derivative with domain $\dom(B) = \WW_p^{\frac{1}{2m}}(\R_+; \LL_q(\R^{n-1};E))$ and $C = (- \Delta)^{1/2}$ denote the square of (minus) the (spatial) Laplace operator with domain $\dom(C) = \LL_p(\R_+; \BB_{q,q}^1(\R^{n-1};E))$.
 These operators commute on $\dom(L^{1/2m}) = \dom(B) \cap \dom(C)$.
 Moreover, since $B + C + 1$ is relatively bounded w.r.t.\ $L^{1/2m}$, by the perturbation theorem for $\Rcal$-sectorial operators Proposition \ref{prop:Perturbation_Theorem}, for sufficiently small $\eta > 0$, the perturbed operator $L^{1/2m} - \eta (B + C + 1)$ generates a bounded holomorphic semigroup on $\LL_{p,q}(\R_+ \times \R^{n-1};E)$.
 Also, since the operators $\partial_y$ (with $\dom(\partial_y) = \{ u = u(t, (\vec x', y)) \in \LL_{p,q}(\R_+ \times \R_+^n;E): \, \partial_y u \in \LL_{p,q}(\R_+ \times \R_+^n;E) \}$) and $L^{1/2m}$ are resolvent-commutative, by the Dore--Venni Theorem, cf.\ \cite[Corollary 1.6]{DeHiPr03}, the abstract Cauchy problem
  \[
   \begin{cases}
    \partial_y w + L^{1/2m} w
     = f,
     &y > 0,
     \\
    w(0,\cdot)
     = 0
   \end{cases}
  \]
 has $\LL_q$-maximal regularity as well.
 (Note that for the function $w$ we switched the order of the variables to $(y, t, \vec x')$.)
 Thus, for $f(y, (t, \vec x')) := (B + C + 1) \ee^{- \eta (B + C + 1) y} g(t, \vec x')$ we obtain a solution $w = w(y,(t,\vec x')) \in \WW_q^1(\R_+; \LL_{p,q}(\R_+ \times \R^{n-1}; E)) \cap \LL_q(\R_+; \WW_{p,q}^{(1,2)}(\R_+ \times \R^{n-1};E))$ for which, in particular, the Duhamel formula is valid:
  \[
   w(y,\cdot)
    = \int_0^y \ee^{- L^{1/2m} (y - \tilde y)} (B + C + 1) \ee^{- \eta (B + C + 1) \tilde y} g \dd \tilde y.
  \]
 To employ integration by parts in $y$, observe that
  \begin{align*}
   (B + C + 1) \ee^{- \eta (B + C + 1) \tilde y} g
    &= \ee^{- \eta (C + 1) \tilde y} B \ee^{-\eta B \tilde y} g
     + \ee^{- \eta (B + 1) \tilde y} C \ee^{-\eta C \tilde y} g
     + \ee^{- (B + C + 1) \tilde y} g
     \\
    &= \ee^{- \eta (C + 1) \tilde y} \frac{\partial}{\partial \tilde y} \ee^{-\eta B \tilde y} g
     + \ee^{- \eta (B + 1) \tilde y} \frac{\partial}{\partial \tilde y} \ee^{-\eta C \tilde y} g
     + \ee^{- (B + C + 1) \tilde y} g,
  \end{align*}
 and note that $B$ and $C$ commute on $\dom(B) \cap \dom(C)$, $(\ee^ {- \eta (B + \frac{1}{2}) \tilde y})_{\tilde y \geq 0}$ and $(\ee^ {- \eta (C + \frac{1}{2}) \tilde y})_{\tilde y \geq 0}$ are bounded semigroups on $\LL_{p,q}(\R_+ \times \R^{n-1};E)$.
 Therefore, each summand in the above representation of the function $(B + C + 1) \ee^{\eta (B + C + 1) \tilde y} g$ lies in $\LL_q(\R_+; \LL_{p,q}(\R_+ \times \R^{n-1};E)) \cap \CC(\R_+; \LL_{p,q}(\R_+ \times \R^{n-1};E)$ and, therefore, is a measurable function on $\R \times \R_+$ with values in $\LL_q(\R^{n-1};E)$.
Invoking the vector-valued version of Fubini's Theorem \cite[Theorem 6.16]{AE3} and the vector-valued characterisation of the spaces $\LL_q(\R; \LL_q(\R^{n-1};E))$ and $\LL_q(\R^n;E)$, see \cite[Bemerkung 6.19]{AE3}, we may interpret this as an element of $\LL_{p,q}(\R_+ \times \R^n; E)$ and via integration by parts we may conclude that
  \begin{align*}
   w(y,\cdot)
    &= - \frac{1}{\eta} \int_0^y L^{1/2m} \ee^{- L^{1/2m}(y - \tilde y)} \ee^{- \eta (B + C + 1) \tilde y} g \, \dd \tilde y
     - \frac{1}{\eta} \big[ \ee^{- L^{1/2m}(y - \tilde y)} \ee^{- \eta (B + C + 1) \tilde y} g \big]_{\tilde y = 0}^{\tilde y = y}
     \\
    &= \frac{1}{\eta} L^{1/2m} (B + C + 1)^{-1} w(y,\cdot)
     + \frac{1}{\eta} \big( \ee^{- L^{1/2m} y} - \ee^{- \eta (B + C + 1) y} \big) g,
  \end{align*}
 and, therefore, after rearranging the terms and multiplying by $L^{1/2m}$ from the left,
  \begin{align*}
   L^{1/2m} \ee^{- L^{1/2m} y} g
    &= L^{1/2m} \ee^{- \eta (B + C + 1) y} g
     + L^{1/2m} \big( \eta - L^{1/2m} (B + C + 1)^{-1}) w(y, \cdot)
     \\
    &= L^{1/2m} \ee^{- \eta (B + C + 1) y} g
     + \big( \eta - L^{1/2m} (B + C + 1)^{-1}) L^{1/2m} w(y, \cdot),
  \end{align*}
 also employing that $L_\lambda$ commutes with $B$ and $C$.
 Here, on the right-hand side $\ee^{- \eta (B + C + 1) y} g$ and $w \in \WW_{p,q}(\R_+ \times \R_+^{n+1};E)$, $L^{1/2m}$ maps $\LL_q(\R_+; \WW_{p,q}(\R_+ \times \R^{n-1};E))$ into $\LL_q(\R_+;\LL_{p,q}(\R_+ \times \R^{n-1};E))$ and $L^{1/2m} (B + C + 1)^{-1}$ is a bounded linear operator on $\LL_q(\R_+; \LL_{p,q}(\R_+ \times \R^{n-1};E))$ (which is isometrically isomorphic to $\LL_{p,q}(\R_+ \times \R_+^{n+1};E)$ via $\tilde w(t, (\vec x', y)) := w(y, (t, \vec x'))$, hence, we obtain that $\big( (t, (\vec x', y)) \mapsto L^{1/2m} \ee^{- L^{1/2m} y} g(t, \vec x') \big) \in \LL_{p,q}(\R_+ \times \R_+^n;E)$ as claimed.
\end{proof}

 This also finishes the proof of Theorem \ref{thm:Lp-Lq-optimality}, up to perturbation and localisation.
 We leave the latter, rather technical steps to the interested reader.

\section*{Acknowledgement}

I would like to thank Dieter Bothe for bringing up the interesting fast surface chemistry and fast sorption model considered in \cite{AugBot21} and \cite{AugBot21a} and for encouraging me to extend the appendix of a preprint version of the latter to this full manuscript.
Moreover, I would like to thank Robert Denk for his kind advice.
I would also like to thank the anonymous referee for their valuable suggestions that led to an improvement of the manuscript.

\section{Appendix}

\begin{proposition}
\label{prop:aux}
 Let $\lambda > 0$, $E$ be any Banach space and $p \in (1,\infty)$.
 Then, the operator $- A_\lambda$ defined as the $\LL_p(\R^n;E)$-realisation of the Fourier symbol $- (\lambda + \abs{\vec \xi}^{2m})^{1/2m}$ generates a bounded $\CC_0$-semigroup on $\LL_p(\R^n;E)$.
\end{proposition}

\begin{proof}
 In view of the Hille-Yosida Theorem for semigroups of class $(M,\omega) = (M,0)$ \cite[Theorem II.3.8]{EnNa00}, we need to show that $\mu + A_\lambda$ is continuously invertible on $\LL_p(\R^n;E)$ for every $\mu > 0$ and there is a constant
  \begin{equation}
   \norm{ (\mu (\mu + A_\lambda)^{-1})^k }_{\B(\LL_p(\R^n;E))}
    \leq C
    \label{eqn:Aux-1}
  \end{equation}
 for some constant $C > 0$ independent of $\mu > 0$ and $k \in \N$.
 To this end, we note that the Fourier symbol of $\mu + A_\lambda$ is $\mu + (\lambda + \abs{\vec \xi}^{2m})^{1/2m}$, hence, it suffices to show that $\frac{\mu}{\mu + (\lambda + \abs{\vec \xi}^{2m})^{1/2m}}$ is a Fourier-multiplier on $\LL_p(\R^n;E)$ and the norm estimates \eqref{eqn:Aux-1} are valid.
 In view of the Mikhlin multiplier theorem, we will show that
  \begin{equation}
   \abs{ \vec \xi^{\vec \alpha} \partial_{\vec \xi}^{\vec \alpha} \big( \frac{\mu^k}{(\mu + (\lambda + \abs{\vec \xi}^{2m})^{1/2m})^k} \big) }
    \leq C_{\vec \alpha}
    \quad
    \text{for } \abs{\vec \alpha} \leq \lfloor \frac{n}{2} \rfloor + 1, \, k \in \N, \, \vec \xi \in \R^n, \, \mu > 0,
    \label{eqn:Aux-2}
  \end{equation}
 for constants $C_{\vec \alpha} > 0$ independent of $k \in \N$ and $\mu > 0$.
 In that case $(0, \infty) \subseteq \rho(- A_\lambda)$ and \eqref{eqn:Aux-1} is valid.
 Estimate \eqref{eqn:Aux-2} will follow by combining Lemma \ref{lem:aux-1} and Lemma \ref{lem:aux-2} below with the fact that any polynomial $p$ of order $k$ can be estimated by $c_n^p \binom{n+k}{k}$ for any $n \in \N$ and a suitable constant $c_n^p > 0$.
 Since the operator symbol is scalar, the classical, scalar version of the Mikhlin multiplier theorem can be invoked to show that $\norm{(\mu (\mu + A_\lambda)^{-1})^k}_{\B(\LL_p(\R^n)} \leq C$ for all $\mu > 0$ and $k \in \N_0$, as the operator norm does not depend on the choice of the Banach space $E \neq \{0\}$, i.e.\ it coincides with the operator norm for the corresponding scalar-valued operator, where $E$ is replaced by $\K$.
 Consequently, $- A_\lambda$ generates a bounded $\CC_0$-semigroup on $\LL_p(\R^n;E)$.
\end{proof}

\begin{lemma}
\label{lem:aux-1}
 For every $N \in \N$, $\Ical = (i_1, \ldots, i_N) \in \{1, \ldots, n \}^N$, there are $J_\Ical \in \N$, polynomials $p_{\Ical,j}$ of order $\leq N$, polynomials $q_{\Ical,j}$ of order $\leq N-1$ and exponents $\alpha_{\Ical,j}^{(l)} \in \N_0$, $l = 0, 1, \ldots, N$, $\beta_{\Ical,j} \in \N_0$ with
  \[
   \beta_{\Ical,j}
    = \sum_{l=0}^n \alpha_{\Ical,j}^{(l)} + \operatorname{deg}(q_{\Ical,j})
    \quad \text{and} \quad
   \beta_{\Ical,j} \leq 2m N
  \]
 such that
  \begin{align}
   &\big( \frac{\partial}{\partial \xi_{i_1}} \cdots \frac{\partial}{\partial \xi_{i_N}} \big) \frac{\mu^k}{(\mu + (\lambda + \abs{\vec \xi}^{2m})^{1/2m})^k}
    \nonumber \\
    &= \frac{\mu^k}{(\mu + (\lambda + \abs{\vec \xi}^{2m})^{1/2m})^{k+n}} \sum_{j=1}^{J_{\Ical}} p_{\Ical,j}(k) q_{\Ical,j}(\mu) \cdot \prod_{l=1}^d \xi_l^{\alpha_{\Ical,j}^{(l)}} \cdot \abs{\vec \xi}^{\alpha_{\Ical,j}^{(0)}} \cdot (\lambda + \abs{\vec \xi}^{2m})^{- \frac{\beta_{\Ical,j}}{2m}}.
    \label{eqn:Aux-3}
  \end{align}
\end{lemma}

\begin{proof}
 We proceed by induction over $N \in \N$.
 For $N = 1$, we have
  \begin{align*}
   \frac{\partial}{\partial \xi_{i_1}} \big( \frac{\mu^k}{(\mu + (\lambda + \abs{\vec \xi}^{2m})^{1/2m})^k} \big)
    = \frac{\mu^k}{(\mu + (\lambda + \abs{\vec \xi}^{2m})^{1/2m})^k} \big( - k \xi_{i_1} \abs{\vec \xi}^{2m-1} (\lambda + \abs{\vec \xi}^{2m})^{- \frac{2m-1}{2m}} \big),
  \end{align*}
 so that we may choose $\Ical = \{i_1\}$, $J_\Ical = 1$, $p_{\Ical,1}(k) = k$, $q_{\Ical,1}(\mu) = 1$, $\alpha_{\Ical,1}^{(l)} = \delta_{l,i_1}$ for $l = 1, \ldots, N$, $\alpha_{\Ical,1}^{(0)} = 2m - 1$ and $\beta_{\Ical,1} = 2m - 1$.
 \newline
 Now, provided the claim is valid for $\Ical = (i_1, \ldots, i_N) \in \{ 1, \ldots, n \}^N$, it follows for $(i_1, \ldots, i_N, i_{N+1})$ that
  \begin{align*}
   &\big( \frac{\partial}{\partial \xi_{i_1}} \cdots \frac{\partial}{\partial \xi_{i_{N+1}}} \big) \frac{\mu^k}{(\mu + (\lambda + \abs{\vec \xi}^{2m})^{1/2m})^k}
   \\
   &= \frac{\partial}{\partial \xi_{i_{N+1}}} \big( \frac{\mu^k}{(\mu + (\lambda + \abs{\vec \xi}^{2m})^{1/2m})^{k+n}} \sum_{j=1}^{J_{\Ical}} p_{\Ical,j}(k) q_{\Ical,j}(\mu) \cdot \prod_{l=1}^d \xi_l^{\alpha_{\Ical,j}^{(l)}} \cdot \abs{\vec \xi}^{\alpha_{\Ical,j}^{(0)}} \cdot (\lambda + \abs{\vec \xi}^{2m})^{- \frac{\beta_{\Ical,j}}{2m}} \big)
   \\
   &= \frac{\mu^k}{(\mu + (\lambda + \abs{\vec \xi}^{2m})^{1/2m})^{k+n+1}} (- (k + N)) \xi_{i_{N+1}} \abs{\vec \xi}^{2m-1} (\lambda + \abs{\vec \xi}^{2m})^{- \frac{2m-1}{2m}}
   \\ &\qquad
   \cdot \sum_{j=1}^{J_{\Ical}} p_{\Ical,j}(k) q_{\Ical,j}(\mu) \cdot \prod_{l=1}^d \xi_l^{\alpha_{\Ical,j}^{(l)}} \cdot \abs{\vec \xi}^{\alpha_{\Ical,j}^{(0)}} \cdot (\lambda + \abs{\vec \xi}^{2m})^{- \frac{\beta_{\Ical,j}}{2m}}
   \\
   &\quad
   + \frac{\mu^k}{(\mu + (\lambda + \abs{\vec \xi}^{2m})^{1/2m})^{k+n}} \sum_{j=1}^{J_{\Ical}} p_{\Ical,j}(k) q_{\Ical,j}(\mu) \cdot \big( \alpha_{\Ical,j}^{(i_{N+1})} \xi_{i_{N+1}}^{\alpha_{\Ical,j}^{(i_{N+1})} - 1} \prod_{l \neq i_{N+1}} \xi_l^{\alpha_{\Ical,j}^{(l)}} \cdot
   \\ &\qquad
   \abs{\vec \xi}^{\alpha_{\Ical,j}^{(0)}} \cdot (\lambda + \abs{\vec \xi}^{2m})^{- \frac{\beta_{\Ical,j}}{2m}}
   \\ &\qquad
   + \alpha_{\Ical,j}^{(0)} \prod_{l=1}^n \xi_l^{\alpha_{\Ical,j}^{(l)}} \cdot \xi_{i_{N+1}} \abs{\vec \xi}^{\alpha_{\Ical,j}^{(0)} - 2} \cdot (\lambda + \abs{\vec \xi}^{2m})^{- \frac{\beta_{\Ical,j}}{2m}}
   \\ &\qquad
   - \beta_{\Ical,j} \xi_{i_{N+1}} \abs{\vec \xi}^{2m-2} \prod_{l=1}^n \xi_l^{\alpha_{\Ical,j}^{(l)}} \cdot \abs{\vec \xi}^{\alpha_{\Ical,j}^{(0)}} \cdot (\lambda + \abs{\vec \xi}^{- \frac{\beta_{\Ical,j} + 2m}{2m}} \big)
  \end{align*}
 and rearranging the term, this shows that this partial derivative of order $N + 1$ has the form \eqref{eqn:Aux-3}.
 Now the claim follows by induction over $N$.
\end{proof}

\begin{lemma}
\label{lem:aux-2}
 For all $a, b \in [0, \infty)$ and $M, N \in \N$:
  \[
   a^M b^N
    \leq \frac{M! \cdot N!}{(M+N)!} (a + b)^{M + N}.
  \]
\end{lemma}

\begin{proof}
 It holds that
  \begin{align*}
   \frac{\partial^k}{\partial a^k} a^M b^N
    &= \frac{M!}{(M-k+1)!} a^{M-k} b^N,
    \\
   \frac{\partial^k}{\partial a^k} (a + b)^{M + N}
    &= \frac{(M+N)!}{(M+N-k+1)!} (a + b)^{M + N - k},
    \quad
    k = 1, \ldots, M.
  \end{align*}
 Moreover, for $a = 0$,
  \[
   \left. \frac{\partial^k}{\partial a^k} a^M b^N \right|_{a=0}
    = 0
    \quad
    \text{for } k = 0, 1, \ldots, M-1,
  \]
 whereas
  \[
    \left. \frac{\partial^k}{\partial a^k} (a + b)^{M + N} \right|_{a=0}
    \geq 0
    \quad
    \text{for } k = 0, 1, \ldots, M-1.
  \]
 Therefore,
  \begin{align*}
   a^M b^N
    &= \int_0^a \int_0^{s_1} \cdots \int_0^{s_{M-1}} \frac{\partial^M}{\partial s_M^M} (s_M^M b^N) \, \dd s_{M_1} \cdots \dd s_1
    \\
    &= \int_0^a \int_0^{s_1} \cdots \int_0^{s_{M-1}} M! \cdot b^N \, \dd s_{M_1} \cdots \dd s_1
    \\
    &\leq \int_0^a \int_0^{s_1} \cdots \int_0^{s_{M-1}} M! \cdot (s_M + b)^N \, \dd s_{M_1} \cdots \dd s_1
    \\
    &= \int_0^a \int_0^{s_1} \cdots \int_0^{s_{M-1}} \frac{M! \cdot N!}{(M+N)!} \cdot \frac{\partial^M}{\partial s_M^M} (s_M + b)^{M+N} \, \dd s_{M_1} \cdots \dd s_1
    \\
    &\leq \frac{M! \cdot N!}{(M+N)!} (a + b)^{M + N}.
  \end{align*}
\end{proof}
\end{document}